\documentclass[12pt,a4paper]{amsart}
\usepackage{amssymb,amsxtra}
\usepackage[cmtip,arrow]{xy}
\usepackage{pb-diagram,pb-xy}
\usepackage{hyperref}

\calclayout

\newcommand{\+}{\nobreakdash-}
\renewcommand{\:}{\colon}
\renewcommand{\.}{\mskip.5\thinmuskip}

\DeclareMathOperator{\Hom}{Hom}
\DeclareMathOperator{\Ext}{Ext}
\DeclareMathOperator{\Tor}{Tor}
\DeclareMathOperator{\Spec}{Spec}
\DeclareMathOperator{\coker}{coker}
\DeclareMathOperator{\Id}{Id}
\DeclareMathOperator{\id}{id}
\DeclareMathOperator{\im}{im}
\DeclareMathOperator{\pd}{pd}

\newcommand{\ot}{\otimes}
\newcommand{\rarrow}{\longrightarrow}
\newcommand{\larrow}{\longleftarrow}
\newcommand{\tim}{\rightthreetimes}

\newcommand{\lrarrow}{\.\relbar\joinrel\relbar\joinrel\rightarrow\.}

\newcommand{\bu}{{\text{\smaller\smaller$\scriptstyle\bullet$}}}

\newcommand{\sA}{\mathsf A}
\newcommand{\sB}{\mathsf B}
\newcommand{\sC}{\mathsf C}
\newcommand{\sD}{\mathsf D}
\newcommand{\sE}{\mathsf E}
\newcommand{\sF}{\mathsf F}

\newcommand{\alg}{{\operatorname{\mathsf{--alg}}}}
\newcommand{\modl}{{\operatorname{\mathsf{--mod}}}}
\newcommand{\modr}{{\operatorname{\mathsf{mod--}}}}

\newcommand{\contra}{{\operatorname{\mathsf{--contra}}}}
\newcommand{\qcoh}{{\operatorname{\mathsf{--qcoh}}}}
\newcommand{\ctrh}{{\operatorname{\mathsf{--ctrh}}}}
\newcommand{\ctra}{{\operatorname{\mathsf{-ctra}}}}
\newcommand{\adj}{{\operatorname{\mathsf{-adj}}}}

\renewcommand{\b}{{\mathsf{b}}}
\newcommand{\proj}{{\mathsf{proj}}}

\newcommand{\Sets}{\mathsf{Sets}}
\newcommand{\Add}{\mathsf{Add}}

\newcommand{\Vir}{{\mathbb V\mathrm{ir}}}

\newcommand{\s}{{\mathbf s}}
\newcommand{\rop}{{\mathrm{op}}}
\newcommand{\gr}{{\mathrm{gr}}}
\newcommand{\pr}{{\mathrm{pr}}}

\newcommand{\boL}{\mathbb L}
\newcommand{\boR}{\mathbb R}
\newcommand{\boT}{\mathbb T}
\newcommand{\boZ}{\mathbb Z}

\newcommand{\A}{\mathfrak A}
\newcommand{\I}{\mathfrak I}
\newcommand{\J}{\mathfrak J}
\newcommand{\K}{\mathfrak K}
\newcommand{\R}{\mathfrak R}
\renewcommand{\S}{\mathfrak S}
\newcommand{\U}{\mathfrak U}

\renewcommand{\O}{\mathcal O}
\newcommand{\cA}{\mathcal A}
\newcommand{\C}{\mathcal C}
\newcommand{\cE}{\mathcal E}
\newcommand{\cF}{\mathcal F}
\newcommand{\cJ}{\mathcal J}
\newcommand{\cP}{\mathcal P}
\newcommand{\cQ}{\mathcal Q}
\newcommand{\cR}{\mathcal R}

\newcommand{\Section}[1]{\bigskip\section{#1}\medskip}
\theoremstyle{plain}
\newtheorem{thm}{Theorem}[section]

\newtheorem{lem}[thm]{Lemma}
\newtheorem{prop}[thm]{Proposition}
\newtheorem{cor}[thm]{Corollary}
\theoremstyle{definition}
\newtheorem{ex}[thm]{Example}
\newtheorem{exs}[thm]{Examples}
\newtheorem{rem}[thm]{Remark}

\begin{document}

\title{Abelian right perpendicular subcategories \\
in module categories}
\author{Leonid Positselski}

\address{Institute of Mathematics, Czech Academy of Sciences,
\v Zitn\'a~25, 115~67 Prague~1, Czech Republic}

\email{positselski@math.cas.cz}

\begin{abstract}
 We show that an abelian category can be exactly, fully faithfully
embedded into a module category as the right perpendicular subcategory
to a set of modules or module morphisms if and only if it is a locally
presentable abelian category with a projective generator, or in other
words, the category of models of an additive algebraic theory of
possibly infinite bounded arity.
 This includes the categories of contramodules over topological rings
and other examples.
 Various versions of the definition of the right perpendicular
subcategory are considered, all leading to the same class of
abelian categories.
 We also discuss sufficient conditions under which the natural forgetful
functors from the categories of contramodules to the related
categories of modules are fully faithful.
\end{abstract}

\maketitle

\tableofcontents

\section*{Introduction}
\medskip

\subsection{{}}
 The word \emph{contramodule} roughly means ``an object of a locally
presentable abelian category with a projective generator''.
 In fact, cocomplete abelian categories with a projective generator can
be described as the categories of algebras/modules over additive
monads on the category of sets (see the discussion in the introduction
to the paper~\cite{PR} or in~\cite[Section~6]{PS}).
 Such a category is locally presentable if and only if the monad
$\boT\:\Sets\rarrow\Sets$ is accessible, i.~e., the functor~$\boT$
commutes with $\kappa$\+filtered colimits in $\Sets$ for a large enough
cardinal~$\kappa$.

 Algebras over a monad $\boT$ are sets with (possibly infinitary)
operations: for any set $X$, elements of the set $\boT(X)$ can be
interpreted as $X$\+ary operations in $\boT$\+algebras.
 That is why \emph{categories of models of $\kappa$\+ary algebraic
theories} is another name for the categories of algebras over
$\kappa$\+accessible monads on $\Sets$.
 Such a category is abelian if and only if it is additive; for this,
the monad $\boT$ needs to be additive, that is it has to contain
the operations of addition, zero element, and inverse element, and
all the other operations in $\boT$ have to be additive with respect
to the abelian group structure defined by these.
 In this case, it might be preferable to modify the terminology and
speak about ``$\boT$\+modules'' rather than ``$\boT$\+algebras''.

 So objects of a cocomplete abelian category with a projective generator
can be viewed as module-like structures with infinitary additive
operations.
 Contramodules over topological associative rings provide typical
examples.
 In fact, to any complete, separated topological ring $\R$ with
a base of neighborhoods of zero formed by open right ideals one can
assign an additive, accessible monad $\boT_\R\:\Sets\rarrow\Sets$.
 Then one defines \emph{left\/ $\R$\+contramodules} as
$\boT_\R$\+modules~\cite{Psemi,Pweak,Prev,PR,PS,Pcoun}.

\subsection{{}}
 Thus, contramodules are sets with infinitary additive operations of
the arity bounded by some cardinal.
 Hence it comes as a bit of a surprise when the natural forgetful
functors from the categories of contramodules to the related categories
of modules turn out be fully faithful.
 An infinitary additive operation is uniquely determined by its finite
aspects.
 Why should it be?
 Nevertheless, such results are known to hold in a number of situations.
 
 The observation that weakly $p$\+complete/Ext-$p$-complete abelian
groups carry a uniquely defined natural structure of $\boZ_p$\+modules
that is preserved by any group homomorphisms between them goes back
to~\cite[Lemma~4.13]{Jan}.
 The claim that the forgetful functor provides an isomorphism between
the category of $\boZ_p$\+contramodules $\boZ_p\contra$ and the full
subcategory of weakly $p$\+complete abelian groups in $\boZ\modl$ was
announced in~\cite[Remark~A.3]{Psemi}, where the definition of
a contramodule over a topological associative ring also first appeared.
 Similarly, it was mentioned in~\cite[Remark~A.1.1]{Psemi} that
contramodules over the coalgebra $\C$ dual to the algebra of formal
power series $\C\spcheck=k[[x]]$ in one variable~$x$ over a field~$k$
form a full subcategory in the category of $k[x]$\+modules; and
a generalization to the coalgebras dual to power series in several
commutative variables $\C\spcheck=k[[x_1,\dotsc,x_m]]$ was formulated.

 A more elaborated discussion of these results for the rings $k[[x]]$
and $\boZ_p$ can be found in~\cite[Section~1.6]{Prev}; and
a generalization to the adic completions of Noetherian rings with
respect to arbitrary ideals was proved in~\cite[Theorem~B.1.1]{Pweak}.
 A version for adic completions of right Noetherian associative rings
at their centrally generated ideals was obtained
in~\cite[Theorem~C.5.1]{Pcosh}.

\subsection{{}}
 In the most straightforward terms, the situation can be explained as
follows.
 One says that an \emph{$x$\+power infinite summation operation} is
defined on an abelian group $B$ if for every sequence of elements
$b_0$, $b_1$, $b_2$,~\dots~$\in B$ an element denoted formally by
$\sum_{n=0}^\infty x^nb_n\in B$ is defined in such a way that
the equations of additivity
$$
 \sum\nolimits_{n=0}^\infty x^n(b_n+c_n)=
 \sum\nolimits_{n=0}^\infty x^nb_n + \sum\nolimits_{n=0}^\infty x^nc_n
 \qquad \forall\, b_n,\,c_n\in P,
$$
contraunitality
$$
 \sum\nolimits_{n=0}^\infty x^nb_n=b_0
 \qquad\text{if\, $b_1=b_2=b_3=\dotsb=0$},
$$
and contraassociativity
$$
 \sum\nolimits_{i=0}^\infty x^i\sum\nolimits_{j=0}^\infty x^jb_{ij}
 = \sum\nolimits_{n=0}^\infty x^n\sum\nolimits_{i+j=n}b_{ij} \qquad
 \forall\,b_{ij}\in B,\quad i,\,j=0,\,1,\,2,\,\dots
$$
are satisfied.
 Then it is an elementary linear algebra exercise
(see~\cite[Section~1.6]{Prev} or~\cite[Section~3]{Pcta}) to show that
an $x$\+power infinite summation operation in an abelian group $B$
is uniquely determined by the additive operator $x\:B\rarrow B$,
$$
 xb=\sum\nolimits_{n=0}^\infty x^nb_n, \qquad
 \text{where\, $b_1=b$\, and \,$b_0=b_2=b_3=b_4=\dotsb=0$}.
$$
 Furthermore, a $\boZ[x]$\+module structure on $B$ comes from
an $x$\+power infinite summation operation if and only if
\begin{equation} \label{x-power-summation-hom-ext-vanishing-condition}
 \Hom_{\boZ[x]}(\boZ[x,x^{-1}],B)=0=\Ext^1_{\boZ[x]}(\boZ[x,x^{-1}],B).
\end{equation}

 ``An abelian group with an $x$\+power infinite summation operation''
is another name for a contramodule over the topological ring of formal
power series $\boZ[[x]]$; so the forgetful functor $\boZ[[x]]\contra
\rarrow \boZ[x]\modl$ is fully faithful, and
the conditions~\eqref{x-power-summation-hom-ext-vanishing-condition}
describe its image.
 Notice that the dual condition
$$
 \boZ[x,x^{-1}]\ot_{\boZ[x]}M=0
$$
describes the category of discrete/torsion $\boZ[[x]]$\+modules as
a full subcategory in the category of $\boZ[x]$\+modules.
 In the case of several commutative variables $x_1$,~\dots,~$x_m$,
rewriting an $[x_1,\dotsc,x_m]$\+power infinite summation operation as
the composition of $x_j$\+power infinite summations over the indices
$j=1$,~\dots,~$m$ allows to obtain a similar description of
the abelian category of $\boZ[[x_1,\dotsc,x_m]]$\+contramodules as
a full subcategory in the abelian category of
$\boZ[x_1,\dotsc,x_m]$\+modules~\cite[Section~4]{Pcta}.

\subsection{{}}
 The above-mentioned results on the possibility of recovering infinite
summation operations from their finite aspects apply to finite
collections of commutative/central variables only.
 So it was a kind of breakthrough when it was shown
in~\cite[Theorem~2.1]{Psm} that the forgetful functor from the category
of left contramodules over the coalgebra $\C$ dual to the algebra of 
formal power series $\C\spcheck=k\{\{x_1,\dotsc,x_m\}\}$ in
noncommutative variables $x_1$,~\dots,~$x_m$ into the category of left
modules over the free associative algebra $k\{x_1,\dotsc,x_m\}$ is
fully faithful.
 In other words, an $\{x_1,\dotsc,x_m\}$\+power infinite summation
operation in a $k$\+vector space/abelian group $B$ is determined by
the collection of linear/additive operators
$x_j\:B\rarrow B$, \ $1\le j\le m$.

 The following general result, going back, in some form, to
Isbell~\cite{Isb}, is discussed in the paper~\cite{PS}: for any locally
presentable abelian category $\sB$ with a projective generator, one can
choose a projective generator $Q\in\sB$ such that $\Hom_\sB(Q,{-})$ is
a fully faithful functor from $\sB$ into the category of modules over
the endomorphism ring $\Hom_\sB(Q,Q)$.
 Moreover, locally presentable abelian categories with a projective
generator can be characterized as exactly and accessibly embedded,
reflective full abelian subcategories in the categories of modules over
associative rings.
 Under the Vop\v enka principle, the ``accessibly embedded'' condition
can be dropped.

\subsection{{}}
 Two classes of examples of locally presentable abelian categories with
a projective generator are mentioned in the paper~\cite{PR}: in
addition to the categories of contramodules over topological
rings~\cite[Sections~5\+-7]{PR}, there are also right perpendicular
subcategories to sets of modules of projective dimension at most~$1$
in the module categories~\cite[Examples~4.1]{PR}.
 The fact that right perpendicular subcategories to sets/classes of
objects of projective dimension~$\le\nobreak1$ in abelian categories
are abelian was discovered by Geigle and Lenzing
in~\cite[Proposition~1.1]{GL} and, much later, discussed by the present
author in~\cite[Section~1]{Pcta}.
 The results of~\cite[Theorem~B.1.1]{Pweak}
and~\cite[Theorem~C.5.1]{Pcosh} identify the categories of contramodules
over certain topological rings with the right perpendicular
subcategories to certain (sets of) modules of projective dimension
not exceeding~$1$.

 The latter results depend on Noetherianity assumptions.
 These can be weakened to weak proregularity assumptions
(in the sense of~\cite{Sch,PSY,Pmgm}), and somewhat further, 
as we show in this paper.
 On the other hand, we demonstrate an example of the perpendicular
category to a module of projective dimension~$1$ which is not
equivalent to the category of contramodules over any topological ring
(see Example~\ref{accessible-additive-monads-examples}(6)).
 Perhaps the conclusion should be that the good class of categories
of module-like structures with infinitary operations to work with
is that of the categories of modules over additive, accessible monads
on $\Sets$, and the categories of contramodules over topological rings
are only a certain subclass.
 (The discussion of flat contramodules and contramodule approximation
sequences in~\cite[Sections~6\+-7]{PR}, \cite[Section~5]{PSl0},
\cite[Section~7]{PSl} shows this subclass to be better behaved in
some ways.)

\subsection{{}}
 As it was discussed above, the class of all categories of modules over
additive, accessible monads on $\Sets$ coincides with that of locally
presentable abelian categories with a projective generator.
 The main result of this paper can be described as stating that
the class of all right perpendicular subcategories to sets of objects
in the categories of modules over associative rings also coincides with
the class of locally presentable abelian categories with a projective
generator.
 The caveat is that one has to specify what one means by ``right
perpendicular subcategories''.

 Right perpendicular subcategories to (sets of) modules of projective
dimension at most~$1$ are not enough, and right perpendicular
subcategories to arbitrary sets of modules, as defined in~\cite{GL},
do not even need to be abelian.
 However, one can consider the maximal reasonable class of right
perpendicular subcategories in the categories of modules over
associative rings, namely, those right perpendicular subcategories to
sets of modules that happen to be abelian and exactly embedded.

 One can also consider the minimal reasonable class of right
perpendicular subcategories in modules, namely, those
right $\Ext^{0,1}$\+perpendicular subcategories which happen to
coincide with the right $\Ext^{0..\infty}$\+perpendicular subcategories
to the same set of modules.
 These are automatically abelian and exactly embedded (in fact,
it suffices that the $\Ext^{0,1}$\+perpendicular and
the $\Ext^{0..2}$\+perpendicular subcategories coincide, for
this conclusion to hold).
 Then we show that both the ``maximal'' and the ``minimal reasonable''
classes of right perpendicular subcategories in the categories of
modules over associative rings embody the same classes of abstract
abelian categories, viz., the locally presentable abelian categories
with a projective generator.

\subsection{{}}
 Actually, given that an exact embedding of a certain abelian category
$\sB$ into the category of modules over some associative $R$ ring
is shown to exist, one naturally wants more than just an arbitrary
such embedding.
 One wants the image of this fully faithful functor to be closed under
extensions in $R\modl$.
 One wants the functor $\sB\rarrow R\modl$ to induce isomorphisms on
all the $\Ext$ groups.
 In the final analysis, one wants the induced triangulated functor
between the unbounded derived categories $\sD(\sB)\rarrow\sD(R\modl)$
to be fully faithful.

 In this paper we explain that, for a locally presentable abelian
category $\sB$ with a projective generator, the task of finding a fully
faithful embedding $\sB\rarrow R\modl$ inducing an isomorphism on
the groups $\Ext^i$ for all $i\le n$ is related to the task of
representing $\sB$ as the right $\Ext^{0,1}$\+perpendicular subcategory
to a set of objects in $R\modl$ coinciding with the right
$\Ext^{0..n}$\+perpendicular subcategory to the same set of modules.
 Then we proceed to construct, for an arbitrary locally presentable
abelian category with a projective generator, a fully faithful embedding
$\sB\rarrow R\modl$ satisfying these conditions for $n=\infty$.
 In fact, we prove that the induced triangulated functor
$\sD^-(\sB)\rarrow\sD^-(R\modl)$ is fully faithful.

\subsection{{}}
 Now let us describe the content of this paper in more detail.
 In Section~\ref{additive-monads-on-sets-secn}, we discuss
the interpretation of an arbitrary cocomplete abelian category with
a projective generator as the category of modules over an additive
monad on the category of sets.
 Locally presentable abelian categories with a projective generator
correspond to accessible additive monads on $\Sets$.
 Various examples of such abelian categories and such monads are
provided, including contramodules over topological rings and
right perpendicular subcategories to morphisms of free modules.

 In Section~\ref{examples-of-fully-faithful-forgetful-secn} we discuss
the situation when the forgetful functor from the category of modules
over a monad on $\Sets$ to the category of modules over the ring of
(some) unary operations in this monad is fully faithful.
 This section consists mostly of examples, which include contramodules
over finitely centrally generated ideals $I$ in associative rings~$R$,
contramodules over the adic completions $\R$ of such rings $R$ with
respect to such ideals~$I$, contramodules over central multiplicative
subsets $S$ in associative rings~$R$, and contramodules over
the $S$\+completions $\R$ of such rings $R$ with respect to such
multiplicative subsets~$S$.
 The proofs of some of the assertions related to these examples are
postponed to Section~\ref{topological-rings-secn}, where we prove
a rather general sufficient condition for full-and-faithfulness
of the forgetful functor from the category of contramodules over
a topological ring to the category of modules.

 In Section~\ref{perpendicular-subcategories-secn}, we introduce
the notion of a right $n$\+perpendicular subcategory in an abelian
category, where $n\ge0$ is an integer or $n=\infty$.
 We also define the notion of an $n$\+good projective generator of
a locally presentable abelian category, and establish a connection
between the two.
 Every $n$\+good projective generator of a locally presentable abelian
category $\sB$ embeds it as a right $n$\+perpendicular subcategory
into a module category, but the converse is not true.
 A counterexample to this effect is given in
Section~\ref{examples-of-good-projective-generators-secn}, which also
contains the definition of a good module over an associative ring $R$
and the claim that all good $n$\+tilting $R$\+modules are good
$R$\+modules, with the references to~\cite{BMT} and~\cite{FMS,Ba,PS}.

 Section~\ref{lambda-flat-secn} contains the proof of the main result
of this paper, namely, that every locally presentable abelian category
with a projective generator can be embedded as
a right $\infty$\+perpendicular category into the category of modules
over an associative ring, because a big enough copower of its
given projective generator is good.
 The argument is based on a result of~\cite{Isb} and~\cite{PS},
claiming, in our present terminology, that a big enough copower of
a given projective generator is $0$\+good, and the theory of
$\kappa$\+flat modules over associative rings, which is developed in
the beginning of this section for this purpose.

 The final Section~\ref{nonabelian-secn} contains counterexamples of
various badly behaved, from our point of view, full subcategories in
the categories of modules over associative rings which one can
obtain by relaxing the conditions imposed in our definition of
a right $n$\+perpendicular subcategory.
 Moreover, we show that many right $\Ext^{1..n}$\+orthogonal classes
in the categories of modules, $1\le n\le\infty$, viewed as abstract
categories, can be realized as right $\Ext^{0..n+1}$\+perpendicular
subcategories in the categories of modules over appropriately
modified rings.

\subsection{{}}
 One terminological remark is in order.
 Following~\cite{AR}, we understand the terms such as
``$\kappa$\+presentable'', ``$\kappa$\+accessible'',
``$\kappa$\+generated'', etc., to mean
``$<\nobreak\kappa$\+present\-able'', ``$<\nobreak\kappa$\+accessible'',
``$<\nobreak\kappa$\+generated'', etc.
 Here $\kappa$~is a regular cardinal.
 So a module is called $\kappa$\+generated if it generated by
\emph{less than~$\kappa$} elements, and it is called
$\kappa$\+presentable if it is isomorpic to the cokernel of
a morphism of free modules with \emph{less than~$\kappa$} generators.
 The cokernel of a morphism of free modules with at most~$\lambda$
generators is called $\lambda^+$\+presentable, where $\lambda^+$
is the successor cardinal of a cardinal~$\lambda$.

\medskip
\noindent\textbf{Acknowlegdement.}
 I~am grateful to Jan \v St\!'ov\'\i\v cek, Sefi Ladkani, Luisa Fiorot,
Silvana Bazzoni, Alexander Sl\'avik, Jan Trlifaj, and
Ji\v r\'\i\ Rosick\'y for helpful conversations and comments.
 The author's research is supported by the Israel Science Foundation
grant~\#\,446/15 and by the Grant Agency of the Czech Republic
under the grant P201/12/G028.

\Section{Additive Monads on the Category of Sets}
\label{additive-monads-on-sets-secn}

 A \emph{monad on the category of sets} is a covariant functor
$\boT\:\Sets\rarrow\Sets$ endowed with the natural transformations
of \emph{multiplication} $\phi_\boT\:\boT\circ\boT\rarrow\boT$ and
\emph{unit} $\epsilon_\boT\:\Id_{\Sets}\rarrow\boT$ satisfying
the equations of \emph{associativity}
$$
 \boT\circ\boT\circ\boT\rightrightarrows\boT\circ\boT\rarrow\boT
$$
$\phi_\boT(\phi_\boT\circ\boT)=\phi_\boT(\boT\circ\phi_\boT)$
and \emph{unitality}
$$
 \boT\rightrightarrows\boT\circ\boT\rarrow\boT
$$
$\phi_\boT(\boT\circ\epsilon_\boT)=\Id_{\Sets}=
\phi_\boT(\epsilon_\boT\circ\boT)$.
 An \emph{algebra over\/~$\boT$} is a set $B$ endowed with a map
of sets $\pi_B\:\boT(B)\rarrow B$ satisfying the equations of
associativity
$$
 \boT(\boT(B))\rightrightarrows\boT(B)\rarrow B
$$
$\pi_B\circ\phi_\boT(B)=\pi_B\circ\boT(\pi_B)$
and unitality
$$
 B\rarrow\boT(B)\rarrow B
$$
$\pi_B\circ\epsilon_\boT(B)=\id_B$.
 We denote category of all algebras over a monad $\boT$ by $\boT\alg$.

 For any set $X$, elements of the set $\boT(X)$ are interpreted as
\emph{$X$\+ary operations} on $\boT$\+algebras in the following way.
 Let $B$ be a $\boT$\+algebra, $b\:X\rarrow B$ be a map of sets, and
$t\in\boT(X)$ be an element.
 Then the element $t_\boT(b)\in B$ is defined as $t_\boT(b)=
\pi_B(\boT(b)(t))$, where $\boT(b)\:\boT(X)\rarrow\boT(B)$ is the map 
obtained by applying the functor $\boT$ to the map~$b$, and
$\boT(b)(t)\in\boT(B)$ is the element obtained by applying the map
$\boT(b)$ to the element~$t$.
 So $t_\boT=t_\boT(B)$ is a map $t_\boT\:B^X\rarrow B$, that is
an $X$\+ary operation in (the underlying set of) an arbitrary
$\boT$\+algebra~$B$.

 The following lemma describes the class of monads on the category of
sets that we are interested in.
 A monad satisfying these conditions is called \emph{additive}.

\begin{lem}
 The following conditions on a monad\/ $\boT\:\Sets\rarrow\Sets$ are
equivalent: \par
\textup{(a)} the category\/ $\boT\alg$ is additive; \par
\textup{(b)} the category\/ $\boT\alg$ is abelian; \par
\textup{(c)} there exist elements
\textup{``$x+y$''${}\in\boT(\{x,y\})$}, \ 
\textup{``$-x$''${}\in\boT(\{x\})$}, and\/ $0\in\boT(\varnothing)$,
where\/ $\{x,y\}$ is a two-element set, $\{x\}$ is a one-element set,
and\/ $\varnothing$ is the empty set, such that the corresponding
operations define an abelian group structure on every\/ $\boT$\+algebra
$B$, and all the other operations $t_\boT\:B^X\rarrow B$, \
$t\in\boT(X)$, \ $X\in\Sets$ are abelian group homomorphisms with
respect to these abelian group structures. \par
 The elements \textup{``$x+y$''${}\in\boT(\{x,y\})$}, \
\textup{``$-x$''${}\in\boT(\{x\})$}, and\/ $0\in\boT(\varnothing)$ in
the condition~(c) are unique if they exist.
\end{lem}

\begin{proof}
 (a)~$\Longrightarrow~$(c): For any monad $\boT\:\Sets\rarrow\Sets$
and every set $X$, the map $\pi_{\boT(X)}=\phi_\boT(X)\:\boT(\boT(X))
\rarrow\boT(X)$ defines a $\boT$\+algebra structure on
the set $\boT(X)$.
 This $\boT$\+algebra is called the \emph{free\/ $\boT$\+algebra}
generated by the set~$X$.
 For any $\boT$\+algebra $B$, there is a natural bijection
$\Hom_\boT(\boT(X),B)\simeq B^X$, where $\Hom_\boT(C,B)$ denotes
the set of all morphisms $C\rarrow B$ in the category $\boT\alg$.

 In particular, one has $\Hom_\boT(\boT(*),B)=B$, where $*$~denotes
a one-element set.
 So the forgetful functor $\boT\alg\rarrow\Sets$ assigning to every
$\boT$\+algebra $B$ its underlying set $B$ is corepresented by
the free $\boT$\+algebra $\boT(*)$; and likewise, the functor
assigning to a $\boT$\+algebra $B$ the set $B^X$ is corepresented by
the free $\boT$\+algebra $\boT(X)$.
 Therefore, natural transformations $B^X\rarrow B$ correspond
bijectively to $\boT$\+algebra morphisms $\boT(*)\rarrow\boT(X)$,
i.~e., to elements of the set $\boT(X)$.

 Now if the category $\boT\alg$ is additive, then there must be
a naturally defined abelian group structure on the set of morphisms
$B=\Hom_\boT(\boT(*),B)$ for every $\boT$\+algebra~$B$.
 This means the natural transformations of addition
$B\times B\rarrow B$, inverse element $B\rarrow B$, and zero
element $0\in B$.
 According to the above, these must be come from some elements
``$x+y$''${}\in\boT(\{x,y\})$, \ ``$-x$''${}\in\boT(\{x\})$, and
$0\in\boT(\varnothing)$.

 Biadditivity of compositions of morphisms in an additive category
implies that all other operations on the underlying sets of
$\boT$\+algebras must be abelian group homomorphisms with
respect to this abelian group structure, and the uniqueness of
an additive category structure implies uniqueness of the elements
``$x+y$'', \ ``$-x$'', and~$0$.
 Alternatively, one can say that any two abelian group structures on
a given set that are additive with respect to one another always
coincide. 

 (c)~$\Longrightarrow~$(b): One observes that, for any monad
$\boT\:\Sets\rarrow\Sets$, every map between the underlying sets of
two $\boT$\+algebras $C\rarrow B$ forming commutative squares with
the $X$\+ary operations $t_\boT(C)$ and $t_\boT(B)$ for all the sets $X$
and elements $t\in\boT(X)$ is a $\boT$\+algebra morphism
(it suffices to take $X=C$).
 Furthermore, a collection of $X$\+ary operations $t_A\:A^X\rarrow A$
defined on a given set $A$ for all sets $X$ and elements $t\in\boT(X)$
corresponds to a $\boT$\+algebra structure on $A$ if and only if
the maps~$t_A$ satisfy the composition and unit relations encoded in
the structural natural transformations $\phi_\boT\:\boT\circ\boT\rarrow
\boT$ and~$\epsilon_\boT\:\Id\rarrow\boT$ of the monad~$\boT$.

 Then, in the assumption of the condition~(c), one repeats the proof of
the assertion that the category of modules over an associative ring is
abelian, replacing unary of finitary additive operations with
infinitary ones.
 The point is that the forgetful functor assigning to every
$\boT$\+algebra $B$ its underlying abelian group $B$ is exact; so
the kernels, images, and cokernels of morphisms in $\boT\alg$ can be
constructed as the kernels, images, and cokernels of the morphisms of
the underlying abelian groups, endowed with the induced
operations~$t_A$.
\end{proof}

 In the case of an additive monad $\boT$, we modify our terminology
and say ``$\boT$\+modules'' instead of ``$\boT$\+algebras''.
 The abelian category of $\boT$\+algebras/modules is denoted by
$\boT\modl$ instead of $\boT\alg$ in this case.

\begin{exs} \label{additive-monads-examples}
 (1)~Let $\sC$ be a category with set-indexed coproducts, and let
$M\in\sC$ be an object.
 Then the functor $\boT_M\:\Sets\rarrow\Sets$ assigning to every
set $X$ the set $\boT_M(X)=\Hom_\sC(M,M^{(X)})$ of all morphisms
in $\sC$ from $M$ into the coproduct $M^{(X)}$ of $X$ copies of $M$
in $\sC$ a monad on the category of sets.
 The monad $\boT_M$ can be viewed as a combinatorial datum
encoding the category structure of the full subcategory formed
by all the objects $M^{(X)}\in\sC$.

 Indeed, given two sets $Y$ and $X$, the set of all morphisms
$M^{(Y)}\rarrow M^{(X)}$ in $\sC$ can be computed as
the set $\boT_M(X)^Y$.
 Suppose that we are given a map $f\:Y\rarrow \boT_M(X)$; then
the composition with the related morphism $M^{(Y)}\rarrow M^{(X)}$
defines a map $\boT_M(Y)\rarrow\boT_M(X)$.
 In particular, one can set $Y=\boT_M(X)$ and take~$f$ to be
the identity map; then we obtain a natural map
$\boT_M(\boT_M(X))\rarrow\boT_M(X)$.
 This is our monad multiplication map $\phi_{\boT_M}(X)$.

 In a more fancy language, one explains the construction of
the monad~$\boT_M$ as follows.
 The functor $F\:\Sets\rarrow\sC$ taking a set $X$ to the object
$F(X)=M^{(X)}$ is left adjoint to the functor $G\:\sC\rarrow\Sets$
taking an object $N\in\sC$ to the set $G(N)=\Hom_\sC(M,N)$.
 The functor $\boT_M$ is the composition of these two adoint
functors $\boT_M=GF\:\Sets\rarrow\Sets$; hence the monad
structure.
 The full subcategory formed by all the objects $M^{(X)}$ in $\sC$
can be recovered as the \emph{Kleisli category} of the monad $\boT_M$,
that is the full subcategory of all free $\boT_M$\+algebras in
$\boT_M\alg$.
 The free $\boT_M$\+algebra $\boT_M(X)\in\boT_M\alg$ corresponds to
the object $M^{(X)}\in\sC$ \cite{V1,V2}.

\smallskip

 (2)~Now let $\sC$ be an additive category with set-indexed coproducts
containing the images of idempotent endomorphisms of its objects, and
let $M\in\sC$ be an object.
 Then $\boT_M$ is an additive monad on $\Sets$: the element
``$x+y$''${}\in\boT_M(\{x,y\})$ corresponds to the diagonal morphism
$M\rarrow M\oplus M$, the element ``$-x$''${}\in\boT_M(\{x\})$
corresponds to the morphism $-\id_M\:M\rarrow M$, and the element
$0\in\boT_M(\varnothing)$ corresponds to the zero morphism $M\rarrow0$.

 Denote by $\Add_\sC(M)\subset\sC$ the full subcategory of all direct
summands of coproducts of copies of $M$ in~$\sC$.
 Then $\sB=\boT_M\modl$ is the unique abelian category with enough
projective objects such that the full subcategory $\sB_\proj\subset\sB$
of projective objects in $\sB$ is equivalent to the full subcategory
$\Add_\sC(M)\subset\sC$ \cite[Theorem~1.1(a)]{PS2}.

 Alternatively, the category $\sB$ can be constructed as the category
of \emph{coherent functors} on $\Add_\sC(M)$.
 One observes that the category $\Add_\sC(M)$ has weak kernels, hence
the category of coherent functors on it is abelian~\cite[Lemma~1]{Kra}.
 We refer to the introduction to~\cite{PR} for further references.

\smallskip

 (3)~In particular, let $\sB$ be an abelian category with set-indexed
coproducts and a projective generator~$P$.
 Then $\sB$ is equivalent to the category of modules over the additive
monad $\boT_P\:X\longmapsto\Hom_\sB(P,P^{(X)})$ on the category
of sets.
 The equivalence is provided by the functor assigning to every object
$N\in\sB$ the set of morphisms $\Hom_\sB(P,N)$ with its natural
$\boT_P$\+module structure (see the discussion in the introduction
to~\cite{PR} and a further discussion in~\cite[Section~6.3]{PS}).

 Conversely, for any additive monad $\boT\:\Sets\rarrow\Sets$,
the abelian category $\boT\modl$ has set-indexed coproducts, and
the free $\boT$\+module with one generator $P=\boT(*)$ is
a natural projective generator of $\boT\modl$.

\smallskip

 (4)~Let $R$ be an associative ring and $\sB\subset R\modl$ be
a reflective full subcategory in the category of left $R$\+modules.
 Denote the reflector (i.~e., the functor left adjoint to
the embedding $\sB\rarrow R\modl$) by $\Delta\:R\modl\rarrow\sB$.
 Then the category $\sB$ has small limits (which coincide with
the limits in $R\modl$, as $\sB\subset R\modl$ is closed under limits)
and colimits (which can be computed by applying the functor $\Delta$
to the coproducts of object of $\sB$ taken in $R\modl$).

 Suppose further that the full subcategory $\sB$ is closed under
cokernels in $R\modl$.
 Then $\sB$ is an abelian category with set-indexed products and
coproducts, and its embedding $\sB\rarrow R\modl$ is an exact functor
preserving products.
 Applying the functor $\Delta$ to the free $R$\+module with one
 generator $R$, we obtain a natural projective generator $P=\Delta(R)$
of the abelian category~$\sB$.
 Thus the category $\sB$ is equivalent to the category of modules over
 the additive monad $\boT_P\:\Sets\rarrow\Sets$.

 Denoting by $R[X]$ the left $R$\+module freely generated by a set $X$,
one can compute the coproduct of $X$ copies of $P$ in $\sB$ as
$P^{(X)}=\Delta(R[X])$.
 Hence the monad $\boT_P$ assigns to a set $X$ the underlying set
of the abelian group/left $R$\+module $\Hom_\sB(P,P^{(X)})=
\Hom_R(\Delta(R),\Delta(R[X]))=\Hom_R(R,\Delta(R[X]))=
\Delta(R[X])$. 
\end{exs}

 Let $\sB$ be an abelian category with set-indexed coproducts and
a projective generator $P$, and let $\kappa$~be a regular cardinal.
 Then the category $\sB$ is locally $\kappa$\+presentable and
the object $P\in\sB$ is $\kappa$\+presentable if and only if
the monad $\boT_P$ in Example~\ref{additive-monads-examples}(3)
is $\kappa$\+accessible, i.~e., the functor $\boT_P$ preserves
$\kappa$\+filtered colimits.

 If the additive category $\sC$ in
Example~\ref{additive-monads-examples}(2) is accessible and
the object $M\in\sC$ is $\kappa$\+presentable (or, more generally,
$\kappa$\+generated), then the monad $\boT_M$ is $\kappa$\+accessible.
 Assuming Vop\v enka's principle, any full subcategory closed under
limits in $R\modl$ is (reflective and) locally
presentable~\cite[Corollary~6.24]{AR}, so the monad $\boT_P$ in
Example~\ref{additive-monads-examples}(4) is always accessible.

 Here are some examples of accessible additive monads on the category
of sets.

\begin{exs}  \label{accessible-additive-monads-examples}
 (1)~Let $R$ be an associative ring.
 Then the functor $\boT_R$ assigning to a set $X$ the set $R[X]$ of
all finite formal linear combinations of elements of the set $X$ with
coefficients in $R$ has a natural structure of monad on the category
of sets.

 The monad multiplication map $\phi_{\boT_R}(X)\:R[R[X]]\rarrow R[X]$
opens the parentheses in a formal linear combination of formal linear
combinations of elements of $X$, assigning to it a formal linear
combination of elements of $X$ with the coefficients computed using
the multiplication and addition operations in the ring~$R$.
 The monad unit map $\epsilon_{\boT_R(X)}\:X\rarrow R[X]$ assigns
to an element $x\in X$ the formal linear combination into which
the element~$x$ enters with the coefficient~$1$ and all the other
elements of $X$ enter with the coefficient~$0$.

 The abelian category $\boT_R\modl$ is the category of left
$R$\+modules $R\modl$.
 The natural projective generator $P\in\boT_R\modl$ is the free
left $R$\+module with one generator, $P=R$.

\smallskip

 (2)~Let $\R$ be a separated and complete topological associative
ring with a base of neighborhoods of zero formed by open right ideals.
 For any set $X$, denote by $\R[[X]]$ the set of all infinite formal
linear combinations of elements of $X$ with the coefficients
converging to~$0$ in the topology of~$\R$.
 This means that an expression $\sum_{x\in X}r_x x$, \ $r_x\in\R$,
belongs to $\R[[X]]$ if and only for any open subset $\U\subset\R$,
\ $0\in\U$, the set of all $x\in X$ for which $r_x$~does not belong
to $\U$ is finite.
 In other words, we have $\R[[X]]=\varprojlim_\J\R/\J[X]$, where
the projective limit is taken over all the open right ideals
$\J\subset\R$ and, for any abelian group $A$ and a set $X$,
the notation $A[X]$ stands for the direct sum of $X$ copies of~$A$.

 Then, for any map of sets $f\:X\rarrow Y$, the induced map
$\R[[f]]\:\R[[X]]\rarrow\R[[Y]]$ is constructed by computing
the infinite sums $\sum_{x:f(x)=y}r_x$ for all $y\in Y$ as the limits
of finite partial sums in the topology of~$\R$.
 Furthermore, the functor $\boT_\R\:X\longmapsto\R[[X]]$ has a natural
structure of monad on the category of sets.
 The monad multiplication map $\phi_{\boT_\R}(X)\:\R[[\R[[X]]]]\rarrow
\R[[X]]$ opens the parentheses, computing the coefficients using
the multiplication operation in the ring $\R$ and the infinite
summation in the sense of the passage to the limit of finite
partial sums.
 The above condition on the topology of $\R$ ensures
the convergence~\cite[Section~5]{PR}.

 The monad $\boT_\R$ is additive and $\lambda^+$\+accessible, where
$\lambda^+$~is the successor cardinal of the cardinality of a base
of neighborhoods of zero in~$\R$.
 Modules over the monad $\boT_\R$ are called \emph{left contramodules}
over the topological ring~$\R$, and the abelian category of left
$\R$\+contramodules is denoted by $\boT_\R\modl=\R\contra$.
 So the abelian category $\R\contra$ is locally $\lambda^+$\+presentable
with a natural $\lambda^+$\+presentable projective generator, which is
the free left $\R$\+contramodule with one generator $P=\R=\R[[*]]$.

 The category $\R\contra$ has the additional property that, for every
family of projective objects $P_\alpha\in\R\contra$, the natural
morphism $\coprod_\alpha P_\alpha\rarrow\prod_\alpha P_\alpha$ is
a monomorphism~\cite[Section~1.2]{PR}.

\smallskip

 (3)~Let $R$ be an associative ring and $M$ be a left $R$\+module.
 Denote by $\S=\Hom_R(M,M)^\rop$ opposite ring to the ring of
endomorphisms of the left $R$\+module $M$ (so that the ring $\S$
acts in $M$ on the right).
 The ring $\S$ is naturally endowed with a complete, separated
topological ring structure in which the base of neighborhoods of
zero is formed by the annihilators of finitely-generated submodules
in~$M$.
 Such annihilators are left ideals in $\Hom_R(M,M)$ and right ideals
in~$\S$; so open right ideals form a base of neighborhoods of
zero in~$\S$.

 One readily checks that the monad $\boT_M$ from
Examples~\ref{additive-monads-examples}(1\+2) is naturally isomorphic
to the monad $\boT_\S$ from~(2) in this case~\cite[proof of
Theorem~7.1]{PS}.
 Thus the abelian category $\sB$ from
Example~\ref{additive-monads-examples}(2) corresponding to an object
$M$ of the category of modules over an associative ring $R$ is nothing
but the category of left contramodules $\S\contra$ over the topological
associative ring $\S=\Hom_R(M,M)^\rop$.

 Similar results can be obtained for many abelian or additive categories
$\sC$ other than the category of modules over an associative ring~$R$
(including, in particular, locally finitely presentable Grothendieck
abelian categories, the categories of comodules and semimodules,
etc.)~\cite[Sections~9\+-10]{PS}.

\smallskip

 (4)~Let $R$ be an associative ring and $f\:U^{-1}\rarrow U^0$ be
a morphism of free left $R$\+modules.
 Denote by $\sB=f^\perp$ the full subcategory in $R\modl$ formed by
all the left $R$\+modules $B$ for which $\Hom_R(f,B)\:\Hom_R(U^0,B)
\rarrow\Hom_R(U^{-1},B)$ is an isomorphism.
 Then the full subcategory $\sB$ is closed under the kernels,
cokernels, extensions, and infinite products in $R\modl$
\cite[Remark~1.3]{Pcta}.
 So $\sB$ is an abelian category and the embedding functor
$\sB\rarrow R\modl$ is exact.
 Furthermore, if $\lambda^+$~is the successor cardinal of the supremum
of the cardinalities of the sets of free generators of $U^{-1}$ and
$U^0$, then the full subcategory $\sB\subset R\modl$ is closed under
$\lambda^+$\+filtered colimits.

 According to~\cite[Theorem and Corollary~2.48]{AR}, it follows that
$\sB$ is a locally $\lambda^+$\+presentable category and a reflective
subcategory in $R\modl$.
 Denoting the reflector by $\Delta\:R\modl\rarrow\sB$, the object
$P=\Delta(R)$ is a $\lambda^+$\+presentable projective generator of
the abelian category~$\sB$ (see
Example~\ref{additive-monads-examples}(4) above,
cf.~\cite[Examples~4.1]{PR}).
 Thus $\sB$ is equivalent to the category of modules over a certain
$\lambda^+$\+accessible additive monad $\boT_f\:\Sets\rarrow\Sets$.

 The monad $\boT_f$ can be described more explicitly in terms of
generators and relations in the following way.
 Let $V$ and $W$ denote the sets of free generators of
the $R$\+modules $U^0$ and $U^{-1}$, respectively.
 For every element $v\in V$, denote also by~$v$ the related
$R$\+module morphism $R\rarrow U^0$.
 For any object $B\in\sB$, consider the map of abelian groups
$\Hom_R(U^{-1},B)\rarrow\Hom_R(U^0,B)$ inverse to the isomorphism
$\Hom_R(f,B)$, and compose it with the coordinate projection
map $\Hom_R(v,B)\:\Hom_R(U^0,B)\rarrow B$.
 The resulting composition $t_v\:B^W=\Hom_R(U^{-1},B)\rarrow B$ is
a $W$\+ary operation in the set~$B$, so $t_v\in\boT_f(W)$.
 Together with the operations of addition and multiplication by
the elements of $R$, the operations~$t_v$, \,$v\in V$ generate
the monad~$\boT_f$.
 The defining relations between these generating operations are
the defining relations for the monad $\boT_R$ (describing the notion
of an $R$\+module) together with the relations guaranteeing that
the map of sets $(t_v)_{v\in V}\:\Hom_R(U^{-1},B)\rarrow\Hom_R(U^0,B)$
is inverse to the map $\Hom_R(f,B)$ for any $\boT_f$\+module~$B$.

\smallskip

 (5)~Let $R$ be a commutative ring and $s\in R$ be an element.
Denote by $E$ the ring $R[s^{-1}]$, and let $f\:U^{-1}\rarrow U^0$
be a free $R$\+module resolution of the $R$\+module~$E$
\cite[proof of Lemma~2.1]{Pcta}.
 The objects of the full subcategory $\sB=f^\perp\subset R\modl$ are
called \emph{$s$\+contramodule $R$\+modules}~\cite{Pmgm,Pcta}.
 The reflector $\Delta=\Delta_s\:R\modl\rarrow\sB$ is described
in~\cite[Proposition~2.1]{Pmgm} and~\cite[Theorem~6.4]{Pcta}.

 In particular, $P=\Delta_s(R)$ is a projective generator of the abelian
category~$\sB$.
 For any set $X$, the coproduct of $X$ copies of the object $P$ in
$\sB$ can be computed as $P^{(X)}=\Delta_s(R[X])$ (while the infinite
products of any objects of $\sB$ coincide with those computed
in $R\modl$).

 Furthermore, for any left $R$\+module $M$, there is a natural short
exact sequence of left $R$\+modules~\cite[Lemma~6.7]{Pcta}
$$
 0\lrarrow\varprojlim\nolimits^1_n\, {}_{s^n}M\lrarrow \Delta_s(M)
  \lrarrow\varprojlim\nolimits_n M/s^nM\lrarrow0,
$$
where ${}_{s^n}M\subset M$ denotes the submodule of all elements in
$M$ annihilated by~$s^n$, the maps in the projective system
${}_sM\larrow{}_{s^2}M\larrow{}_{s^3}M\larrow\dotsb$ are the operators
of multiplication with~$s$, and the maps in the projective system
$M/sM\larrow M/s^2M\larrow M/s^3M\larrow\dotsb$ are the usual
projections.

\smallskip

 (6)~The following example shows that some of the abelian categories
$\sB=f^\perp$ as in~(4) differ from all the abelian categories of
contramodules over topological rings $\R\contra$ as in~(2).

 Let $k$~be a field and $R$ be a commutative $k$\+algebra generated
by an infinite sequence of elements $x_1$, $x_2$, $x_3$,~\dots\ and
an additional generator $s$ with the relations $x_ix_j=0$ for all
$i$, $j\ge1$ and $s^ix_i=0$ for all $i\ge1$ \cite[Example~2.6]{Pmgm}.
 Let $f\:U^{-1}\rarrow U^0$ be a free $R$\+module resolution of the ring
$R[s^{-1}]\simeq k[s,s^{-1}]$ viewed as an $R$\+module, as in~(4).
 Set $\sB=f^\perp$ to be the abelian category of $s$\+contramodule
$R$\+modules and $P=\Delta_s(R)$ to be its natural projective
generator.

 Then, for any set $Y$, one computes, in the spirit of
the computation in~\cite[Example~2.6]{Pmgm}, that the $R$\+module
$\varprojlim_n^1\,{}_{s^n}R[Y]$ is isomorphic to $W_Y[[s]]$,
where $W_Y=\prod_{i=1}^\infty k[Y]/\bigoplus_{i=1}^\infty k[Y]$.
 In particular, the natural map $\varprojlim_n^1\,{}_{s^n}R[Y]
\rarrow(\varprojlim_n^1\,{}_{s^n}R)^Y$ is \emph{not} injective when
the set $Y$ is infinite, and it follows that the natural map $P^{(Y)}
\rarrow P^Y$ from the coproduct of $Y$ copies of $P$ in $\sB$ to
the product of the same copies is not injective, either (hence
\emph{not} a monomorphism).

\smallskip

 (7)~On the other hand, in the context of~(5), assume that
the $s$\+torsion in $R$ is bounded (that is, there exists $m\ge1$
such that $s^nr=0$ for $r\in R$ and $n\ge1$ implies $s^mr=0$).
 Denote by $\R=\varprojlim_n R/s^nR$ the $s$\+adic completion of
the ring $R$, endowed with the $s$\+adic (\,$=$~projective limit)
topology.
 Then the abelian category $\sB=f^\perp$ is equivalent to
the abelian category $\R\contra$, as we will see below in
Example~\ref{comm-ring-ideal-contramodules-examples}(5).
\end{exs}

\Section{Examples of Fully Faithful Contramodule Forgetful Functors}
\label{examples-of-fully-faithful-forgetful-secn}

 The aim of this section is to generalize the results
of~\cite[Theorem~B.1.1]{Pweak} and~\cite[Theorem~C.5.1]{Pcosh},
claiming that the natural forgetful functors from the categories
of contramodules to related categories of modules are fully faithful
under certain assumptions.
 We leave aside the result of~\cite[Theorem~2.1]{Psm}, postponing
the job of its generalization to the next
Section~\ref{topological-rings-secn}.

 Let $\boT\:\Sets\rarrow\Sets$ be an additive monad.
 Then the underlying set of the free $\boT$\+module with one generator
$\boT(*)$ has a natural associative ring structure.
 This is the ring of unary operations in the monad~$\boT$, which can be
also defined as the opposite ring $\boT(*)=
\Hom_\boT(\boT(*),\boT(*))^\rop$ to the ring of endomorphisms of
the free $\boT$\+module with one generator $\boT(*)\in\boT\modl$.
 Let $R$ be another associative ring and $\theta\:R\rarrow\boT(*)$ be
an associative ring homomorphism.
 Then the natural bijection $C\simeq\Hom_\boT(\boT(*),C)$ endows
the underlying set of every $\boT$\+module $C$ with a left
$R$\+module structure, providing a forgetful functor
$\boT\modl\rarrow R\modl$.

 Let $\sB\subset R\modl$ be a reflective full subcategory closed under
cokernels in $R\modl$.
 Let $\Delta\:R\modl\rarrow\sB$ denote the reflector, and set
$P=\Delta(R)$.
 Then $P$ is a projective generator of the abelian category $\sB$,
and the coproduct of $X$ copies of $P$ in $\sB$ can be computed as
$P^{(X)}=\Delta(R[X])$ (see Example~\ref{additive-monads-examples}(4)).

 The following proposition provides a sufficient condition (in fact,
a criterion) for existence of an isomorphism of monads
$\boT\simeq\boT_P$.

\begin{prop} \label{fully-faithful-forgetful-functor-image-identified}
 Suppose that for every set $X$ there is an isomorphism of
left $R$\+modules $\Delta(R[X])\simeq\boT(X)$ forming a commutative
diagram with the adjunction morphism $R[X]\rarrow\Delta(R[X])$
and the natural $R$\+module morphism
$$
 R[X]\lrarrow\Hom_\boT(\boT(*),\boT(*))[X]\lrarrow
 \Hom_\boT(\boT(*),\boT(X))\.=\.\boT(X).
$$
 Then the forgetful functor\/ $\boT\modl\rarrow R\modl$ is fully
faithful and its essential image coincides with the full subcategory\/
$\sB\subset R\modl$.
  So there is an equivalence of abelian categories\/ $\boT\modl
\simeq\sB$ identifying the free\/ $\boT$\+module $\boT(X)$ with
the projective object $P^{(X)}\in\sB$ for every set~$X$.

 Conversely, if the forgetful functor $\boT\modl\rarrow R\modl$ is
fully faithful and its essential image coincides with the full
subcategory\/ $\sB\subset R\modl$, then for every set $X$ there is
a natural isomorphism of left $R$\+modules $\Delta(R[X])\simeq\boT(X)$
forming a commutative diagram with the adjunction morphism $R[X]
\rarrow\Delta(R[X])$ and the natural morphism $R[X]\rarrow\boT(X)$.
\end{prop}

\begin{proof}
 The forgetful functor $\boT\modl\rarrow R\modl$ is exact by
construction, and in particular, it preserves cokernels.
 As every $\boT$\+module is the cokernel of a morphism of free
$\boT$\+modules $\boT(X)\rarrow\boT(Y)$ for some sets $X$ and $Y$,
and the cokernel of any $R$\+module morphism $P^{(X)}\rarrow P^{(Y)}$
belongs to $\sB$, it follows that the image of the forgetful functor
is contained in the full subcategory $\sB\subset R\modl$.

 Furthermore, for any sets $X$ and $Y$ one has
\begin{multline*}
 \Hom_\boT(\boT(X),\boT(Y))\simeq\Hom_R(R[X],\boT(Y)) \\
 \simeq\Hom_R(R[X],P^{(Y)})\simeq\Hom_R(P^{(X)},P^{(Y)}),
\end{multline*}
so the forgetful functor is fully faithful in restriction to the full
subcategory of free $\boT$\+modules in $\boT\modl$.
 Commutativity of the diagram involved in this computation can be
checked using the assumption of commutativity of the diagram in
the formulation of the proposition.
 It follows easily that the forgetful functor is fully faithful on
the whole abelian category of $\boT$\+modules.

 Finally, it remains to notice that every object of $\sB$ is
the cokernel of a morphism $P^{(X)}\rarrow P^{(Y)}$ for some
sets $X$ and $Y$, since $P$ is a projective generator on~$\sB$.
 Hence every object of $\sB$ belongs to the essential image of
the forgetful functor.

 Alternatively, one can argue in the following way.
 Suppose that we already know (e.~g., from the first paragraph of
this proof) that the image of the forgetful functor $\boT\modl
\rarrow R\modl$ is contained in the full subcategory
$\sB\subset R\modl$.
 Then, by the adjunction property of the functor $\Delta$, for
every set $X$ there exists a unique left $R$\+module morphism
$\Delta(R[X])\rarrow\boT(X)$ forming a commutative triangle
diagram with the morphisms $R[X]\rarrow\Delta(R[X])$ and
$R[X]\rarrow\boT(X)$.

 According to Example~\ref{additive-monads-examples}(4), the category
$\sB$ is equivalent to the category of modules over the monad
$\boT_\sB=\boT_P\:X\longmapsto\Delta(R[X])$.
 The functor $\boT\modl\rarrow\sB\simeq\boT_\sB\modl$ is a functor
between two categories of modules over monads on the category of
sets, and it forms a commutative triangle diagram of functors with
the forgetful functors $\boT\modl\rarrow\Sets$ and $\boT_\sB\modl
\rarrow\Sets$.
 Hence our functor is induced by a morphism of monads $\boT_\sB\rarrow
\boT$, which is provided by the above construction.
 Thus the functor $\boT\modl\rarrow\sB$ is an equivalence of categories
whenever the map $\Delta(R[X])\rarrow\boT(X)$ is bijective
for every set~$X$.

 Conversely, if the forgetful functor is an equivalence of categories
$\boT\modl\simeq\sB$, then this equivalence of categories also
identifies the forgetful functors to the category of sets,
$\boT\modl\rarrow\Sets$ and $\sB\rarrow R\modl\rarrow\Sets$.
 Thus it also forms a commutative diagram with the functors adjoint
to the forgetful functors to sets, $X\longmapsto\boT(X)\:
\Sets\rarrow\boT\modl$ and $X\longmapsto\Delta(R[X])\:\Sets
\rarrow\sB$.
 Hence a natural isomorphism of left $R$\+modules $\Delta(R[X])
\simeq\boT(X)$ forming a commutative diagram with the adjunction
maps of sets $X\rarrow\Delta(R[X])$ and $X\rarrow\boT(X)$, and
consequently also with the morphisms of left $R$\+modules
$R[X]\rarrow\Delta(R[X])$ and $R[X]\rarrow\boT(X)$.
\end{proof}

 The observation that (what was then called) an \emph{Ext-$p$-complete}
or \emph{weakly $p$\+complete} abelian group carries a uniquely defined
structure of a module over the ring of $p$\+adic integers $\boZ_p$,
and these $\boZ_p$\+module structures are preserved by any abelian
group homomorphisms between such groups, goes back, at least,
to~\cite[Lemma~4.13]{Jan}.
 The following series of examples shows how far it can be generalized
using the contemporary techniques.  {\hbadness=1100\par}

\begin{exs} \label{comm-ring-ideal-contramodules-examples}
 (1)~Let $R$ be a commutative ring and $I\subset R$ be the ideal
generated by a finite set of elements $s_1$,~\dots, $s_m\in R$.
 An $R$\+module $C$ is said to be an \emph{$I$\+contramodule} if
$\Hom_R(R[s_j^{-1}],C)=0=\Ext_R^1(R[s_j^{-1}],C)$ for all $1\le j\le m$.
 This property does not depend on a chosen set of generators
$s_1$,~\dots, $s_m$ of an ideal $I\subset R$, but only on the ideal $I$
itself~\cite[Theorem~5.1]{Pcta}.

 Let $E$ denote the $R$\+module $\bigoplus_{j=1}^m R[s_j^{-1}]$,
and let $f\:U^{-1}\rarrow U^0$ be a two-term free resolution of
the $R$\+module~$E$.
 Then the full subcategory of $I$\+contra\-module $R$\+modules
$R\modl_{I\ctra}\subset R\modl$ coincides with the full subcategory
$\sB=f^\perp\subset R\modl$ discussed in
Example~\ref{accessible-additive-monads-examples}(4).
 In particular, the embedding functor $R\modl_{I\ctra}\rarrow R\modl$
has a left adjoint functor $\Delta_I\:R\modl\rarrow R\modl_{I\ctra}$
described in~\cite[Proposition~2.1]{Pmgm}
and~\cite[Theorem~7.2]{Pcta}.
 The abelian category $R\modl_{I\ctra}$ is equivalent to the category
of modules over the additive, $\aleph_1$\+accessible monad
$\boT_{R,I}=\boT_f$ on the category of sets, assigning to every set $X$
the underlying set of the $R$\+module $\boT_{R,I}(X)=\Delta_I(R[X])$.

\smallskip

 (2)~Let $R$ be a commutative ring and $I\subset R$ be a finitely
generated ideal.
 Denote by $\R=\varprojlim_n R/I^n$ the $I$\+adic completion of the ring
$R$ endowed with the $I$\+adic (\,$=$~projective limit) topology.
 Consider the abelian category of $\R$\+contramodules $\R\contra$
defined in Example~\ref{accessible-additive-monads-examples}(2).
 Then the image of the forgetful functor $\R\contra\rarrow R\modl$
is contained in the full subcategory of $I$\+contramodule $R$\+modules
$R\modl_{I\ctra}\subset R\modl$ discussed in~(1).

 Indeed, it suffices to check that the free $\R$\+contramodules are
$I$\+contramodule $R$\+modules, as every $\R$\+contramodule is
the cokernel of a morphism of free $\R$\+contra\-modules.
 Now, for any set $X$, the free $\R$\+contramodule
$\R[[X]]=\varprojlim_n R/I^n[X]$ is an $I$\+adically separated and
complete $R$\+module, and all such $R$\+modules are
$I$\+contramodules~\cite[Lemma~5.7]{Pcta}.

 Furthermore, it follows from an appropriate generalization
of~\cite[Theorem~2.1]{Psm} that the forgetful functor $\R\contra\rarrow
R\modl$ is fully faithful for any finitely-generated ideal $I$ in
a commutative ring $R$ (see
Examples~\ref{central-adic-topological-ring}(2\+3) below).
 Thus the abelian category $\R\contra$ is a full subcategory in
the abelian category $R\modl_{I\ctra}$.

\smallskip

 (3)~In the same context, denote by $\Lambda_I$ the $I$\+adic completion
functor assigning to an $R$\+module $M$ the $R$\+module
$\Lambda_I(M)=\varprojlim_n M/I^nM$.
 Then for any $R$\+module $M$ there exists a unique $R$\+module
morphism $\Delta_I(M)\rarrow\Lambda_I(M)$ forming a commutative
diagram with the natural morphisms $M\rarrow\Delta_I(M)$ and
$M\rarrow\Lambda_I(M)$.
 According to
Proposition~\ref{fully-faithful-forgetful-functor-image-identified},
we can conclude that the forgetful functor $\R\contra\rarrow
R\modl_{I\ctra}$ is an equivalence of abelian categories if and only if
the natural morphism $\Delta_I(R[X])\rarrow\Lambda_I(R[X])$ is
an isomorphism for every set~$X$.

 By~\cite[Lemma~2.5]{Pmgm}, the latter condition holds when $I\subset R$
is a weakly proregular ideal in the sense of~\cite{Sch,PSY}.
 Hence the forgetful functor $\R\contra\rarrow R\modl_{I\ctra}$ is
an equivalence of categories for any weakly proregular finitely
generated ideal $I$ in a commutative ring~$R$.
 Moreover, it suffices to require the weak proregularity condition
(the homology of the Koszul/telescope complexes forming a pro-zero
projective system) to hold in the homological degree~$1$, i.~e.,
for the modules $H_1(\Hom_R(T^\bu_n(R;s_1,\dotsc,s_m),R))$.
 Indeed, the natural morphism $\Delta_I(R[X])\rarrow
\Lambda_I(R[X])$ is always surjective with the kernel $\varprojlim_n^1
H_1(\Hom_R(T^\bu_n(R;s_1,\dotsc,s_m),R[X]))$ \cite[Lemma~7.5]{Pcta}.
 (See Example~\ref{comm-ring-ideal-good-projective-generator}(6)
below for further details.) {\hbadness=1200\par}

\smallskip

 (4)~In particular, all ideals in a Noetherian commutative ring are
weakly proregular.
 Thus we have obtained a new proof, based on
Proposition~\ref{fully-faithful-forgetful-functor-image-identified},
of the result of~\cite[Theorem~B.1.1]{Pweak}, according to which
the forgetful functor $\R\contra\rarrow R\modl$ is fully faithful
and its image coincides with the full subcategory
$R\modl_{I\ctra}\subset R\modl$ for any ideal $I$ in a Noetherian
commutative ring~$R$.

\smallskip

 (5)~In the special case of a principal ideal $I=(s)$ in a commutative
ring $R$, the weak proregularity condition mentioned in~(3) means that
the $s$\+torsion in $R$ is bounded.
 Thus the forgetful functor $\R\contra\rarrow R\modl_{I\ctra}$ is
an equivalence of abelian categories in this case
(cf.\ Example~\ref{accessible-additive-monads-examples}(7) above and
Example~\ref{comm-ring-ideal-good-projective-generator}(7) below).
\end{exs}

\begin{exs} \label{central-ideal-contramodules-examples}
 (1)~Let $R$ be an associative ring and $s\in R$ be a central element.
 A left $R$\+module $C$ is said to be an \emph{$s$\+contramodule} if
$\Hom_R(R[s^{-1}],C)=0=\Ext_R^1(R[s^{-1}],C)$ (notice that $R[s^{-1}]$
is a left $R$\+module of projective dimension~$\le\nobreak1$, so all
the higher Ext groups vanish automatically).
 More generally, for an arbitrary element $s\in R$, one can say that
a left $R$\+module $C$ is an $s$\+contramodule if, viewed as a module
over the polynomial ring $\boZ[s]$, it is an $s$\+contramodule in
the sense of Example~\ref{comm-ring-ideal-contramodules-examples}(1)
and~\cite[Section~2]{Pcta}, that is
$\Hom_{\boZ[s]}(\boZ[s,s^{-1}],C)=0=\Ext_{\boZ[s]}^1(\boZ[s,s^{-1}],C)$.

 Let $s_1$,~\dots, $s_m\in R$ be a finite set of central elements
in~$R$, and let $s=\sum_{j=1}^ms_mx_m$ be an element of the ideal
$I\subset R$ generated by~$s_1$,~\dots,~$s_m$.
 Then any left $R$\+module $C$ that is an $s_j$\+contramodule for every
$j=1$,~\dots,~$m$, is also an $s$\+contramodule.
 One can prove this assertion by constructing the $s$\+power infinite
summation operation in~$C$ \cite[Section~3]{Pcta} by the rule
$$
 \sum\nolimits_{n=0}^\infty s^nc_n =
 \sum\nolimits_{n_1=0}^\infty s_1^{n_1}\dotsb
 \sum\nolimits_{n_m=0}^\infty s_m^{n_m}
 (p_{n_1,\dotsc,n_m}(x_1,\dotsc,x_m)c_{n_1+\dotsb+n_m}),
$$
for any elements $c_0$, $c_1$, $c_2$,~\dots~$\in C$, where
$p_{n_1,\dotsc,n_m}$ is the appropriate noncommutative polynomial of
polydegree $(n_1,\dots,n_m)$ in the variables~$x_1$,~\dots,~$x_m$.
 Then it remains to apply~\cite[Theorem~3.3(c)]{Pcta}
(cf.~\cite[first proof of Theorem~5.1]{Pcta}).

 It follows that the property of a left $R$\+module $C$ to be
an $s_j$\+contramodule for every element~$s_j$ of a given set of
central generators of an ideal $I\subset R$ depends only on
the centrally generated ideal $I$ and not on the chosen set of
its central generators.
 A left $R$\+module $C$ with this property is called
an \emph{$I$\+contramodule}.
 The full subcategory of $I$\+contramodule left $R$\+modules
$R\modl_{I\ctra}\subset R\modl$ is closed under the kernels, cokernels,
extensions, and infinite products in $R\modl$.
 So, in particular $R\modl_{I\ctra}$ is an abelian category and its
embedding $R\modl_{I\ctra}\rarrow R\modl$ is an exact functor.

\smallskip

 (2)~Let $R$ be an associative ring and $I\subset R$ be the ideal
generated by a finite set of central elements $s_1$,~\dots, $s_m\in R$.
 As in Example~\ref{comm-ring-ideal-contramodules-examples}(1), denote
by $E$ the left $R$\+module $\bigoplus_{j=1}^m R[s_j^{-1}]$, and
let $f\:U^{-1}\rarrow U^0$ be a two-term free resolution of
the left $R$\+module~$E$.
 Then the full subcategory of $I$\+contramodule $R$\+modules
$R\modl_{I\ctra}\subset R\modl$ coincides with the full subcategory
$\sB=f^\perp\subset R\modl$ from
Example~\ref{accessible-additive-monads-examples}(4).
 In particular, the embedding functor $R\modl_{I\ctra}\rarrow R\modl$
has a left adjoint functor $\Delta_I\:R\modl\rarrow R\modl_{I\ctra}$.
 The abelian category $R\modl_{I\ctra}$ is equivalent to the category
of modules over the additive, $\aleph_1$\+accessible monad
$\boT_{R,I}=\boT_f\:\Sets\rarrow\Sets$, assigning to every set $X$
the underlying set of the left $R$\+module
$\boT_{R,I}(X)=\Delta_I(R[X])$.

\smallskip

 (3)~Let $R$ be an associative ring and $I\subset R$ be an ideal
generated by a finite set of central elements.
 As in Example~\ref{comm-ring-ideal-contramodules-examples}(2),
denote by $\R=\varprojlim_n R/I^n$ the $I$\+adic completion of
the ring $R$ endowed with the $I$\+adic (\,$=$~projective limit)
topology.
 Consider the abelian category of left $\R$\+contramodules $\R\contra$
defined in Example~\ref{accessible-additive-monads-examples}(2).
 Then the forgetful functor $\R\contra\rarrow R\modl$ is fully
faithful (see Examples~\ref{central-adic-topological-ring}(2\+3) below),
and its image is contained in the full subcategory of $I$\+contramodule
left $R$\+modules $R\modl_{I\ctra}\subset R\modl$.
 Thus the abelian category $\R\contra$ is a full subcategory in
the abelian category $R\modl_{I\ctra}$. 

 As in Example~\ref{comm-ring-ideal-contramodules-examples}(3),
denote by $\Lambda_I$ the $I$\+adic completion functor
$M\longmapsto \Lambda_I(M)=\varprojlim_n M/I^nM$.
 Then for any left $R$\+module $M$ the left $R$\+module
$\Lambda_I(M)$ is an $I$\+contramodule, hence there exists
a unique left $R$\+module morphism $\Delta_I(M)\rarrow\Lambda_I(M)$
forming a commutative diagram with the natural morphisms
$M\rarrow\Delta_I(M)$ and $M\rarrow\Lambda_I(M)$.
 By Proposition~\ref{fully-faithful-forgetful-functor-image-identified},
it follows that the forgetful functor $\R\contra\rarrow R\modl_{I\ctra}$
is an equivalence of categories if and only if the natural morphism
$\Delta_I(R[X])\rarrow\Lambda_I(R[X])$ is an isomorphism for every
set~$X$.

 By~\cite[Lemma~7.5]{Pcta}, the morphism $\Delta_I(R[X])\rarrow
\Lambda_I(R[X])$ is surjective with the kernel $\varprojlim_n^1
H_1(\Hom_R(T^\bu_n(R;s_1,\dotsc,s_m),R[X]))$, where $s_1$,~\dots, $s_m$
is some finite set of central generators of the ideal $I\subset R$.
 Hence the morphism in question is an isomorphism whenever
the $R$\+$R$\+bimodules $H_1(\Hom_R(T_n(R;s_1,\dotsc,s_m),R))$, \
$n\ge1$, form a pro-zero projective system.
 In particular, for a principal centrally generated ideal $I=(s)$,
where $s$~is a central element in $R$, the latter condition means
that the $s$\+torsion in $R$ should be bounded.
 (See Example~\ref{central-ideal-good-projective-generator}(4)
below for further details.)

\smallskip

 (4)~Arguing as in~\cite[Lemma~2.4]{Sch} or~\cite[Theorem~4.24]{PSY}
(see also~\cite[Lemma~2.3]{Pmgm}), one shows that the projective system
$H_1(\Hom_R(T_n(R;s_1,\dotsc,s_m),R))$ is pro-zero if and only if
for every injective left $R$\+module $J$ the augmented \v Cech complex
$\check C^\bu_\s(J)\sptilde$ has no cohomology in cohomological
degree~$1$, or equivalently, for every injective right $R$\+module $J$
the complex $\check C^\bu_\s(J)\sptilde$ has no cohomology in degree~$1$
(where $\s$ is a shorhand notation for the sequence
$s_1$,~\dots,~$s_m$).

 Set $X=\Spec S$, where $S$ is the subring in $R$ generated by
$s_1$,~\dots, $s_m$ over~$\boZ$.
 Suppose that the ring $R$ is left Noetherian.
 Then $X$ is a Noetherian affine scheme and the quasi-coherent algebra
$\cR$ over $X$ corresponding to the $S$\+algebra $R$ is (left)
Noetherian in the sense of~\cite[Section~1.2]{EP}.
 Let $\cJ$ be the quasi-coherent left $\cR$\+module associated with
an injective left $R$\+module~$J$.
 Denote by $Z$ the closed subscheme in $X$ defined by the equations
$\{s_1=0$, \dots, $s_m=0\}$, and set $U=X\setminus Z$.
 According to~\cite[Theorem~A.3]{EP}, the quasi-coherent left
$\cR|_U$\+module $\cJ|_U$ is injective and $\cJ$ is flasque sheaf of
abelian groups on~$X$.
 Arguing as in~\cite[Section~1]{Pmgm}, one can conclude that
$H^i(\check C^\bu_\s(J)\sptilde)=0$ for all $i>0$.

 Thus the forgetful functor $\R\contra\rarrow R\modl_{I\ctra}$ is
an equivalence of abelian categories whenever the ring $R$ is either
left or right Noetherian.
 We have obtained a new proof, based on
Proposition~\ref{fully-faithful-forgetful-functor-image-identified},
of the result of~\cite[Theorem~C.5.1]{Pcosh}.
\end{exs} 

 Given an object $N$ in an abelian category $\sA$, we denote by
$\pd_\sA N\in\{-\infty\}\cup\boZ_{\ge0}\cup\{+\infty\}$ the projective
dimension of the object $N\in\sA$.
 The projective dimension of a left module $N$ over an associative
ring $R$ is denoted by $\pd_RN$. 

\begin{exs} \label{comm-ring-mult-subset-contramodules-examples}
 (1)~Let $R$ be a commutative ring and $S\subset R$ be a multiplicative
subset.
 Assume that the projective dimension of the $R$\+module $S^{-1}R$ does
not exceed~$1$.
 An $R$\+module $C$ is said to be an \emph{$S$\+contramodule} if
$\Hom_R(S^{-1}R,C)=0=\Ext_R^1(S^{-1}R,C)$ \cite[Section~1]{PMat}.

 Let $f\:U^{-1}\rarrow U^0$ be a two-term free resolution of
the $R$\+module $S^{-1}R$.
 Then the full subcategory of $S$\+contramodule $R$\+modules
$R\modl_{S\ctra}\subset R\modl$ coincides with the full subcategory
$\sB=f^\perp\subset R\modl$ discussed in
Example~\ref{accessible-additive-monads-examples}(4).
 In particular, $R\modl_{S\ctra}$ is an abelian category with an exact
embedding functor $R\modl_{S\ctra}\rarrow R\modl$, which has a left
adjoint functor $\Delta_S\:R\modl\rarrow R\modl_{S\ctra}$.

 The latter functor can be computed as $\Delta_S(M)=
\Ext^1_R(K^\bu,M)=\Hom_{\sD^\b(R\modl)}(K^\bu,\allowbreak M[1])$ for every
$R$\+module~$M$, where $K^\bu$ denotes the two-term complex
$R\rarrow S^{-1}R$ with the term~$R$ placed in the cohomological
degree~$0$ and the term $S^{-1}R$ placed in the cohomological
degree~$1$\, \cite[Theorem~3.4]{PMat}.
 The abelian category $R\modl_{S\ctra}$ is equivalent to the category
of modules over the additive, accessible monad $\boT_{R,S}=\boT_f$ on
the category of sets, assigning to every set $X$ the underlying set of
the $R$\+module $\boT_{R,S}(X)=\Delta_S(R[X])$. {\hfuzz=2.8pt\par}

\smallskip

 (2)~Let $R$ be a commutative ring and $S\subset R$ be a multiplicative
subset.
 Denote by $\R=\varprojlim_{s\in S}R/sR$ the $S$\+completion of the ring
$R$, endowed with the projective limit topology~\cite[Section~2]{PMat}.
 Then $\R$ is a complete, separated topological commutative ring with
open ideals forming a base of neighborhoods of zero.
 So we can consider the abelian category of $\R$\+contramodules
$\R\contra$ defined in
Example~\ref{accessible-additive-monads-examples}(2).
 The conditions under which the forgetful functor
$\R\contra\rarrow R\modl$ is fully faithful are discussed
in Examples~\ref{central-mult-subset-topological-ring} below.

 Assume that $\pd_RS^{-1}R\le1$.
 Then the image of the forgetful functor $\R\contra\rarrow R\modl$ is
contained in the full subcategory of $S$\+contramodule $R$\+modules
$R\modl_{S\ctra}\allowbreak\subset R\modl$ discussed in~(1).
 Indeed, it suffices to check that the free $\R$\+contramodules are
$S$\+contramodule $R$\+modules, as every $\R$\+contramodule is
the cokernel of a morphism of free $\R$\+contramodules.
 Now, for any set $X$, the free $\R$\+contramodule $\R[[X]]=
\varprojlim_{s\in S} R/sR[X]$ is an $S$\+contramodule
$R$\+module~\cite[Lemma~2.1(a)]{PMat}.

\smallskip

 (3)~Let $R$ be a commutative ring and $S\subset R$ be a multiplicative
subset such that $\pd_RS^{-1}R\le1$.
 Denote by $\Lambda_S$ the $S$\+completion functor assigning to
every $R$\+module $M$ the $R$\+module $\Lambda_S(M)=
\varprojlim_{s\in S}M/sM$.
 Then for any $R$\+module $M$ there exists a unique $R$\+module morphism
$\Delta_S(M)\rarrow\Lambda_S(M)$ forming a commutative diagram with
the natural morphisms $M\rarrow\Delta_S(M)$ and $M\rarrow\Lambda_S(M)$
\cite[Lemma~2.1(b)]{PMat}.
 Applying
Proposition~\ref{fully-faithful-forgetful-functor-image-identified},
we can conclude that the forgetful functor $\R\contra\rarrow
R\modl_{S\ctra}$ is an equivalence of abelian categories if and only if
the natural morphism $\Delta_S(R[X])\rarrow\Lambda_S(R[X])$ is
an isomorphism for every set~$X$.

 According to~\cite[Theorem~2.5(c) or Corollary~2.7]{PMat}, the latter
condition holds when the $S$\+torsion in the ring $R$ is bounded
(i.~e., there exists an element $t\in S$ such that $sr=0$, \,$s\in S$,
\,$r\in R$ implies $tr=0$).
 Hence the forgetful functor $\R\contra\allowbreak\rarrow
R\modl_{S\ctra}$ is an equivalence of categories for any commutative ring
$S$ with a multiplicative subset $S$ such that $\pd_RS^{-1}R\le1$ and
the $S$\+torsion in $R$ is bounded.
 Furthermore, all the $S$\+contramodule $R$\+modules are
$\Delta_S(R)$\+modules, so it suffices to require the $S$\+torsion in
$\Delta_S(R)$, rather than in $R$, to be bounded (see
Examples~\ref{comm-ring-mult-subset-good-projective-generator}
for further details.)
\end{exs}

\begin{exs} \label{central-mult-subset-contramodules-examples}
 (1)~Let $R$ be an associative ring and $S\subset R$ be a multiplicative
subset consisting of central elements.
 Assume that the projective dimension of the left $R$\+module $S^{-1}R$
does not exceed~$1$.
 A left $R$\+module $C$ is said to be an \emph{$S$\+contramodule} if
$\Hom_R(S^{-1}R,C)=0=\Ext_R(S^{-1}R,C)$.

 Let $f\:U^{-1}\rarrow U^0$ be a two-term free resolution of the left
$R$\+module $S^{-1}R$.
 Then the full subcategory of $S$\+contramodule left $R$\+modules
$R\modl_{S\ctra}\subset R\modl$ coincides with the full subcategory
$\sB=f^\perp\subset R\modl$ from
Example~\ref{accessible-additive-monads-examples}(4).
 In particular, $R\modl_{S\ctra}$ is an abelian category with an exact
embedding functor $R\modl_{S\ctra}\rarrow R\modl$, which has a left
adjoint functor $\Delta_S\:R\modl\rarrow R\modl_{S\ctra}$.

 In the same way as in the commutative case of~\cite[Theorem~3.4]{PMat}
(as mentioned in
Example~\ref{comm-ring-mult-subset-contramodules-examples}(1)),
one shows that the latter functor can be computed as
$\Delta_S(M)=\Ext_R^1(K^\bu,M)=\Hom_{\sD^\b(R\modl)}(K^\bu,M[1])$ for
every left $R$\+module $M$, where $K^\bu$ denotes the two-term
complex $R\rarrow S^{-1}R$ (see~\cite[Proposition~3.2(b)]{BP2} for
a more general result).
 The abelian category $R\modl_{S\ctra}$ is equivalent to the category
of modules over the additive, accessible monad $\boT_{R,S}=\boT_f$
on the category of sets, assigning to every set $X$ the underlying
set of the left $R$\+module $\boT_{R,S}(X)=\Delta_S(R[X])$.

\smallskip

 (2)~Let $R$ be an associative ring and $S\subset R$ be a multiplicative
subset consisting of central elements.
 Denote by $\R$ the ring $\varprojlim_{s\in S}R/sR$, endowed with
the projective limit topology.
 Then $\R$ is a complete, separated topological associative ring with
open two-sided ideals forming a base of neighborhoods of zero.
 So we can consider the abelian category of left $\R$\+contramodules
$\R\contra$ defined in
Example~\ref{accessible-additive-monads-examples}(2).
 (See Examples~\ref{central-mult-subset-topological-ring} below for 
a discussion of the conditions under which the forgetful functor
$\R\contra\rarrow R\modl$ is fully faithful.)

 Assume that $\pd_RS^{-1}R\le1$.
 Then the image of the forgetful functor $\R\contra\rarrow R\modl$ is
contained in the full subcategory of $S$\+contramodule $R$\+modules
$R\modl_{S\ctra}\allowbreak\subset R\modl$.
 Indeed, as in
Example~\ref{comm-ring-mult-subset-contramodules-examples}(2),
it suffices to check that the free left $\R$\+contramodules
$\R[[X]]=\varprojlim_{s\in S}R/sR[X]$ are $S$\+contramodule
left $R$\+modules, and, as in~\cite[Lemma~2.1(a)]{PMat}, it suffices
to notice that a left $R$\+module annihilated by an element $s\in S$
is always an $S$\+contramodule, and the class of $S$\+contramodule
left $R$\+modules is closed under projective limits.

\smallskip

 (3)~Let $R$ be an associative ring and $S\subset R$ be a multiplicative
subset of central elements such that $\pd_RS^{-1}R\le1$.
 Denote by $\Lambda_S$ the $S$\+completion functor assigning to every
left $R$\+module $M$ the left $R$\+module $\Lambda_S(M)=
\varprojlim_{s\in S}M/sM$.
 Then for any left $R$\+module $M$ the left $R$\+module $\Lambda_S(M)$
is an $S$\+contramodule, hence there exists a unique left $R$\+module
morphism $\Delta_S(M)\rarrow\Lambda_S(M)$ forming a commutative
diagram with the natural morphisms $M\rarrow\Delta_S(M)$ and $M\rarrow
\Lambda_S(M)$.
 Applying
Proposition~\ref{fully-faithful-forgetful-functor-image-identified},
we once again conclude that the forgetful functor $\R\contra\rarrow
R\modl_{S\ctra}$ is an equivalence of abelian categories if and only if
the natural morphism $\Delta_S(R[X])\rarrow\Lambda_S(R[X])$ is
an isomorphism for every set~$X$.

 Let $Z\subset R$ be a central subring containing~$S$.
 Then both the functors $\Delta_S$ and $\Lambda_S$ computed in
the categories of $Z$\+modules and left $R$\+modules agree (since
the two-term complex $(R\to S^{-1}R)$ is isomorphic to the tensor
product $R\ot_Z(Z\to S^{-1}Z)$, which coincides with the derived
tensor product $R\ot_Z^\boL(Z\to S^{-1}Z)$).
 Notice that we do \emph{not} know what the projective dimension
of the $Z$\+module $S^{-1}Z$ might be.
 But nevertheless it follows from~\cite[Theorem~2.5(c)]{PMat} applied to
the ring $Z$ with the multiplicative subset $S\subset Z$ that
the morphism $\Delta_S(R[X])\rarrow\Lambda_S(R[X])$ is an isomorphism
for every set $X$ provided that the $S$\+torsion in $R$ is bounded.

 Thus the forgetful functor $\R\contra\rarrow R\modl_{S\ctra}$ is
an equivalence of abelian categories for any associative ring $R$ with
a multiplicative subset of central elements $S$ such that
$\pd_SS^{-1}R\le1$ and the $S$\+torsion in $R$ is bounded.
 Furthermore, all the $S$\+contramodule $R$\+modules are modules over
the ring $\Delta_S(R)$, so it suffices to require the $S$\+torsion in
$\Delta_S(R)$, rather than in $R$, to be bounded (see
Examples~\ref{central-mult-subset-good-projective-generator}
for further details.)
\end{exs}

\Section{Full-and-Faithfulness for Contramodules over Topological Rings}
\label{topological-rings-secn}

 In this section, we are interested in modules over the monads
$\boT=\boT_\R$ associated with topological associative rings~$\R$,
as defined in Example~\ref{accessible-additive-monads-examples}(2).
 We refer to~\cite[Remark~A.3]{Psemi}, \cite[Section~1.2]{Pweak},
\cite[Section~2.1]{Prev}, \cite[Section~5]{PR}, \cite[Section~6]{PS},
or~\cite[Section~2.7]{Pcoun} for further and more detailed discussions
of the definition of a contramodule over a topological ring
(see also~\cite[Section~D.5.2]{Psemi}, \cite[Section~1.10]{Pweak}
or~\cite[Section~2.3]{Prev} for comparison with the definition of
a contramodule over a topological associative algebra over a field).

 Let $\R$ be a complete, separated topological associative ring with
a base of neighborhoods of zero formed by open right ideals.
 As in Section~\ref{examples-of-fully-faithful-forgetful-secn},
we observe that the datum of an associative ring $R$ together with
an associative ring homomorphism $\theta\:R\rarrow\R$ defines
an exact forgetful functor $\R\contra\rarrow R\modl$.

 Let $\J\subset\R$ be a right ideal.
 We will say that a finite set of elements $s_1$,~\dots, $s_m\in\J$
\emph{strongly generates} the right ideal $\J$ if for every family of
elements $r_x\in\J$, indexed by some set $X$ and converging to zero
in the topology of~$\R$, there exist families of elements
$r_{j,x}\in\R$, \ $j=1$,~\dots,~$m$, each of them indexed by the set $X$
and converging to zero in the topology of $\R$, such that
$r_x=\sum_{j=1}^ms_jr_{j,x}$ for all $x\in X$.
 Since any finite family of elements in $\R$ converges to zero in
the topology of $\R$, any finite set of elements of a right ideal
$\J\subset\R$ strongly generating the ideal $\J$ also generates
the right ideal $\J$ in the conventional sense
(cf.~\cite[Section~B.4]{Pweak}).

 The following theorem is a generalization of~\cite[Theorem~2.1]{Psm},
and its proof is similar to that in~\cite{Psm} (see
Examples~\ref{conilpotent-coalgebra} below for a discussion).

\begin{thm} \label{topol-ring-fully-faithful-theorem}
 Let\/ $\R$ be a complete, separated topological associative ring,
$R$ be an associative ring, and\/ $\theta\:R\rarrow\R$ be
a ring homomorphism with a dense image.
 Assume that\/ $\R$ has a countable base of neighborhoods of zero
consisting of open two-sided ideals, each of which, viewed as a right
ideal, is strongly generated by a finite set of elements lying in
the image of the map\/~$\theta$.
 Then the forgetful functor\/ $\R\contra\rarrow R\modl$ is
fully faithful.
\end{thm}

\begin{proof}
 Given a set $X$ and a complete, separated topological abelian group
$\A$ with a base of neighborhoods of zero formed by open subgroups
$\U\subset\A$, denote by $\A[[X]]$ the abelian group
$\varprojlim_{\U\subset\A}\A/\U[X]$ of all infinite formal linear
combinations of elements of $X$ with the coefficients converging to
zero in the topology of~$\A$.
 Following the notation in~\cite[Section~D.1]{Pcosh},
\cite[Sections~5\+-6]{PR}, \cite[Section~7.3]{PS},
and~\cite[Section~2.8]{Pcoun}, for any closed subgroup $\A\subset\R$
and any left $\R$\+contramodule $B$, we denote by $\A\tim B\subset B$
the image of the composition $\A[[B]]\rarrow B$ of the natural
embedding $\A[[B]]\rarrow\R[[B]]$ and the contraaction map
$\pi_B\:\R[[B]]\rarrow B$.
 The map $\A[[B]]\rarrow B$ is an abelian group homomorphism,
so $\A\tim B$ is a subgroup in~$B$.

 As usually, for any left $\R$\+module $M$ and any subgroup
$A\subset\R$, we denote by $AM\subset M$ the subgroup generated by
the products $am$, where $a\in A$ and $m\in M$.
 So we have $\A B\subset\A\tim B$ for any closed subgroup $\A\subset\R$
and any left $\R$\+contramodule~$B$.
 For any left ideal $I\subset\R$ and any left $\R$\+module $M$,
the subgroup $IM\subset M$ is a left $\R$\+submodule in~$M$.
 For any closed left ideal $\I\subset\R$ and any left $\R$\+contramodule
$B$, the subgroup $\I\tim B\subset B$ is a left $\R$\+subcontramodule
in~$B$.
 On the other hand, when a closed right ideal $\J\subset\R$ is strongly
generated by a finite set of elements $s_1$,~\dots, $s_m\in\J$, one has
$\J\tim B=\J B=s_1B+\dotsb+s_mB$ for any left $\R$\+contramodule~$B$.

 Let $B$ and $C$ be two left $\R$\+contramodules, and let
$f\:B\rarrow C$ be a left $R$\+module morphism.
 The contraaction map $\pi_B\:\R[[B]]\rarrow B$ is a surjective
morphism of left $\R$\+contramodules.
 In order to show that $f$~is an $\R$\+contramodule morphism, it
suffices to check that the composition $\R[[B]]\rarrow B\rarrow C$
is an $\R$\+contramodule morphism.
 Hence we can replace $B$ with $\R[[B]]$ and suppose that $B=\R[[X]]$
is a free left $\R$\+contramodule generated by a set~$X$.

 Then the composition $X\rarrow C$ of the natural embedding $X\rarrow
\R[[X]]$ with a left $R$\+module morphism $f\:\R[[X]]\rarrow C$ can be
extended uniquely to a left $\R$\+contramodule morphism
$f'\:\R[[X]]\rarrow C$.
 Setting $g=f-f'$, we have a left $R$\+module morphism $g\:\R[[X]]
\rarrow C$ taking the elements of the set $X$ to zero elements in~$C$.
 We have to show that $g=0$.
 Without loss of generality, we can assume that the left
$\R$\+contramodule $C$ is generated by its subset $g(\R[[X]])\subset C$
(otherwise, replace $C$ with its subcontramodule generated by
$g(\R[[X]])$).
 Then we have to show that $C=0$.

 Let $\I\subset\R$ be an open two-sided ideal strongly generated, as
a right ideal, by a finite set of elements $s_1$,~\dots, $s_m\in\I$
belonging to the image of~$\theta$.
 Denote by $I\subset R$ the full preimage of the ideal $\I$ with respect
to ring homomorphism~$\theta$; so $I$ is a two-sided ideal in~$R$
and $R/I\simeq\R/\I$ (since the image of~$\theta$ is dense in~$\R$).
 Then we have $\I[[X]]=\I\tim\R[[X]]=s_1\R[[X]]+\dotsb+s_m\R[[X]]=
I\R[[X]]$ and $\I\tim C=s_1C+\dotsb+s_mC=IC$.

 Now the induced left $(R/I)$\+module morphism $g/I\:\R[[X]]/I\R[[X]]
\rarrow C/IC$ vanishes, because the left $(R/I)$\+module
$\R[[X]]/I\R[[X]]=\R[[X]]/\I[[X]]=\R/\I[X]\allowbreak=R/I[X]$ is
generated by elements from~$X$.
 So the image of the morphism~$g$ is contained in~$IC$.
 Since the left $\R$\+contramodule $C$ is generated by $g(\R[[X]])$
and $IC=\I\tim C$ is a left $\R$\+subcontramodule in $C$, it
follows that $C=IC$.

 We have shown that $C=\I\tim C$ for a countable set of open two-sided
ideals $\I\subset\R$ forming a base of neighborhoods of zero in~$\R$.
 According to the contramodule Nakayama lemma
(\cite[Lemma~D.1.2]{Pcosh} or~\cite[Lemma~6.14]{PR}),
it follows that $C=0$.
\end{proof}

 A generalization of Theorem~\ref{topol-ring-fully-faithful-theorem}
to complete, separated topological associative rings $\R$ with
a countable base of neighborhoods of zero consisting of open
\emph{right} ideals can be found in~\cite[Theorem~6.2]{Pcoun}.

 We will explain below in 
Theorem~\ref{ext-isomorphism-right-perpendicular-thm}(a)
and Example~\ref{topological-ring-0-good}(b)
how to distinguish the objects of the full subcategory $\R\contra$
among the objects of the ambient category $R\modl$ in
the context of Theorem~\ref{topol-ring-fully-faithful-theorem}.

 In the rest of this section we list some examples of topological rings
$\R$ together with ring homomorphisms $R\rarrow\R$ into $\R$ from
a discrete ring $R$ for which one can show that the forgetful functor
$\R\contra\rarrow R\modl$ is fully faithful using
Theorem~\ref{topol-ring-fully-faithful-theorem} (or parts of
the argument from its proof).
 We also provide a certain nonexample and a certain counterexample.

\begin{rem} \label{topological-quotient-ring-remark}
 Let $\R$ be a complete, separated topological associative ring with
a countable base of neighborhoods of zero consisting of open right
ideals.
 Let $\K\subset\R$ be a closed two-sided ideal and $\S=\R/\K$ be
the quotient ring endowed with the quotient topology.
 Then $\S$ is also a complete, separated topological associative ring
with a countable base of neighborhoods of zero consisting of open
right ideals, and the functor of restriction of scalars
$\S\contra\rarrow\R\contra$ is fully faithful.
 Indeed, any family of elements of $\S$ converging to zero in
the topology of $\S$ can be lifted to a family of elements of $\R$
converging to zero in the topology of~$\R$.

 Moreover, if a homomorphism of associative rings $\theta\:R\rarrow\R$
satisfies the assumptions of
Theorem~\ref{topol-ring-fully-faithful-theorem}, then so does
the composition $\bar\theta\:R\rarrow\S$ of the homomorphism~$\theta$
with the natural surjective homomorphism $\R\rarrow\S$.
 Indeed, for any open right/two-sided ideal $\I\subset\R$, one can
consider the open right/two-sided ideal $\J=(\I+\K)/\K\subset\S$.
 Then any family of elements of $\J$ converging to zero in
the topology of $\S$ can be lifted to a family of elements of $\I$
converging to zero in the topology of $\R$, hence the image of any
finite set of elements strongly generating the right ideal
$\I\subset\R$ strongly generated the right ideal $\J\subset\S$.
\end{rem}

\begin{exs} \label{noncommutative-power-series}
 (1)~Let $k$ be a commutative ring and $\R=k\{\{x_1,\dotsc,x_m\}\}$ be
the $k$\+algebra of noncommutative formal Taylor power series in
the variables~$x_1$,~\dots, $x_m$ with the coefficients in~$k$,
endowed with the formal power series topology (or, in other words,
the $\I$\+adic topology for the ideal $\I=(x_1,\dotsc,x_m)\subset\R$).
 We observe that the ideal $\I^n\subset\R$ of all the formal power
series with vanishing coefficients at all the noncommutative monomials
of the total degree less than~$n$ in~$x_1$,~\dots, $x_m$ is strongly
generated, as a left ideal in $\R$, by the finite set of all
the noncommutative monomials of the total degree~$n$
in~$x_1$,~\dots,~$x_m$.
 Denoting by $R=k\{x_1,\dotsc,x_m\}$ the $k$\+algebra of noncommutative
polynomials in~$x_1$,~\dots, $x_m$ and by $\theta\:R\rarrow\R$
the natural embedding, we conclude, by applying
Theorem~\ref{topol-ring-fully-faithful-theorem}, that the forgetful
functor $\R\contra\rarrow R\modl$ is fully faithful.

\smallskip

 (2)~More generally, let $R$ be a quotient algebra of the algebra
$k\{x_1,\dotsc,x_m\}$ of noncommutative polynomials in
the variables~$x_1$,~\dots, $x_m$ over a commutative ring~$k$ by
a two-sided ideal $K\subset k\{x_1,\dotsc,x_m\}$.
 Let $\R=\varprojlim_n R/I^n$ be the adic completion of the algebra $R$
with respect to the two-sided ideal $I=(x_1,\dotsc,x_m)\subset\R$,
endowed with the projective limit topology (\,$=$~$I$\+adic topology
of the left or right $R$\+module~$\R$).
 Then $\R$ is the topological quotient ring of the algebra of
noncommutative formal power series $k\{\{x_1,\dotsc,x_m\}\}$ by
the closure $\K$ of the image of the ideal $K\subset k\{x_1,\dotsc,x_m\}$
in $k\{\{x_1,\dotsc,x_m\}\}$.
 In view of Remark~\ref{topological-quotient-ring-remark},
the forgetful functor $\R\contra\rarrow R\modl$ is fully faithful.

\smallskip

 (3)~Even more generally, let $\R$ be the quotient ring of
the topological algebra of noncommutative formal power series
$k\{\{x_1,\dotsc,x_m\}\}$ by a closed two-sided ideal
$\K\subset k\{\{x_1,\dotsc,x_m\}\}$, endowed with the quotient
topology.
 Then the forgetful functor $\R\contra\rarrow k\{x_1,\dots,x_m\}\modl$
is fully faithful.

\smallskip

 (4)~One can also drop the commutatity assumption on the ring~$k$,
presuming only that the elements of~$k$ commute with
the variables~$x_1$,~\dots, $x_m$ (while the variables do not commute
with each other and the elements of~$k$ do not necessarily commute
with each other).
 All the assertions of~(1\+-3) remain valid in this setting.
\end{exs}

\begin{exs} \label{conilpotent-coalgebra}
 (1)~Let $k$ be a field and $\C$ be a coassociative, counital
$k$\+coalgebra.
 Then the dual vector space $\C\spcheck$ to the $k$\+vector space $\C$
has a natural topological $k$\+algebra structure.
 The category $\C\contra$ of left contramodules over the coalgebra
$\C$ is isomorphic to the category $\C\spcheck\contra$ of left
contramodules over the topological algebra~$\C\spcheck$
\cite[Section~1.10]{Pweak}, \cite[Section~2.3]{Prev}.

\smallskip

 (2)~In particular, when $\C$ is a conilpotent coalgebra over~$k$
with a finite dimensional cohomology space $H^1(\C)$,
the topological algebra $\C\spcheck$ is a topological quotient algebra
of the algebra of noncommutative formal power series
$k\{\{x_1,\dotsc,x_m\}\}$, where $m=\dim_kH^1(\C)$, by a closed
two-sided ideal, as in Example~\ref{noncommutative-power-series}(3).
 This allows to recover the result of~\cite[Thereom~2.1]{Psm}
as a particular case of our
Theorem~\ref{topol-ring-fully-faithful-theorem}.
 (A change of variables may be needed in order to ensure that
an arbitrary dense subalgebra $R\subset\C\spcheck$ contains 
the images of the generators $x_j\in k\{\{x_1,\dotsc,x_m\}\}$.)

\smallskip

 (3)~Let $\C=k\oplus V\oplus k$ be the coalgebra
from~\cite[Section~A.1.2]{Psemi} for which the category of left
$\C$\+contramodules is equivalent to the category of pairs of
$k$\+vector spaces $P_1$ and $P_2$ endowed with a $k$\+linear map
$\Hom_k(V,P_1)\rarrow P_2$.
 The category of left $\C\spcheck$\+modules is equivalent to
the category of pairs of $k$\+vector spaces $M_1$ and $M_2$
endowed with a $k$\+linear map $V\spcheck\ot_k M_1\rarrow M_2$.
 Then the forgetful functor $\C\contra\rarrow\C\spcheck\modl$ is
clearly \emph{not} fully faithful when $V$ is infinite-dimensional.

 This example shows that the finite generatedness condition on
the ideals in the topological ring $\R$ in
Theorem~\ref{topol-ring-fully-faithful-theorem} cannot be readily
replaced with a countable generatedness condition.
\end{exs}

\begin{ex}
 To give another example in which
Theorem~\ref{topol-ring-fully-faithful-theorem} is
\emph{not} applicable, let us consider contramodules over
the Virasoro Lie algebra~$\Vir$.
 The Virasoro Lie algebra is the topological vector space
$\Vir=k((z))d/dz\oplus kC$ of vector fields with coefficients in
the field of Laurent power series in one variable~$z$ over a field~$k$
of characteristic~$0$, extended by adding a one-dimensional vector
space $kC$ as a second direct summand.
 The vectors $L_n=z^{n+1}d/dz$, \ $n\in\boZ$, and $C$ form a topological
basis in~$\Vir$, and the Lie bracket is defined by the formula
$$
 [L_i,L_j]=(j-i)L_{i+j}+\delta_{i,\,-j}\frac{i^3-i}{12}C, \qquad
 [L_i,C]=0, \qquad i,\,j\in\boZ,
$$
where $\delta$ is the Kronecker symbol.

 A \emph{$\Vir$\+contramodule} $P$ is a $k$\+vector space endowed with
a linear operator $C\:P\rarrow P$ and an infinite summation operation
assigning to every sequence of vectors $p_{-N}$, $p_{-N+1}$,
$p_{-N+2}$,~\dots~$\in P$, \ $N\in\boZ$, a vector denoted formally by
$\sum_{i=-N}^\infty L_ip_i\in P$.
 The equations of agreement
$$
 \sum\nolimits_{i=-N}^\infty L_ip_i=\sum\nolimits_{i=-M}^\infty L_ip_i
 \qquad\text{when\, $-N<-M$\, and \,$p_{-N}=\dotsb=p_{-M+1}=0$},
$$
linearity
$$
 \sum\nolimits_{i=-N}^\infty L_i(ap_i+bq_i)=
 a\sum\nolimits_{i=-N}^\infty L_ip_i +
 b\sum\nolimits_{i=-N}^\infty L_iq_i
 \qquad\forall\, a,\,b\in k,\, \ p_i,\, q_i\in P,
$$
centrality of~$C$
$$
 C\sum\nolimits_{i=-N}^\infty L_ip_i=\sum\nolimits_{i=-N}^\infty L_i(Cp_i),
$$
and contra-Jacobi identity
\begin{multline*}
 \sum\nolimits_{i=-N}^\infty L_i
 \left(\sum\nolimits_{j=-M}^\infty L_jp_{ij}\right)
 - \sum\nolimits_{j=-M}^\infty L_j
 \left(\sum\nolimits_{i=-N}^\infty L_ip_{ij}\right)
 \\ 
 =\,\sum\nolimits_{n=-N-M}^\infty L_n\left
 (\sum\nolimits_{i+j=n}^{i\ge-N,\>j\ge-M}(j-i)p_{ij}\right)
 +C\sum\nolimits_{i+j=0}^{i\ge-N,\>j\ge-M}
 \left(\frac{i^3-i}{12}p_{ij}\right)
\end{multline*}
have to be satisfied~\cite[Section~D.2.7]{Psemi},
\cite[Section~1.7]{Prev}.

 We do \emph{not} know whether the forgetful functor $\Vir\contra
\rarrow\Vir\modl$ from the abelian category of $\Vir$\+contramodules
to the abelian category of modules over the Lie algebra $\Vir$ is
fully faithful.
 Perhaps a more natural question would be about full-and-faithfulness
of the forgetful functor $\Vir\contra\rarrow \mathrm{Vir}\modl$ from
the category of $\Vir$\+contramodules to the category of modules over
a discrete version $\mathrm{Vir}=k[z,z^{-1}]d/dz\oplus kC\subset\Vir$ of
Virasoro Lie algebra.
 We do \emph{not} know whether this functor is fully faithful.
 In other words, we do not know whether the above-described infinite
summation operation with the coefficients $L_i$ in a vector space~$P$
can be uniquely recovered from the action of the linear operators
$C$ and $L_i$ in~$P$.

 It would be sufficient to show that the forgetful functor
$\Vir_+\contra\rarrow\mathrm{Vir}_+\modl$ is fully faithful for
the positively-graded subalgebras $\Vir_+=z^2k[[z]]d/dz\subset\Vir$
and $\mathrm{Vir}_+=z^2k[z]d/dz\subset\mathrm{Vir}$ of
(the topological and discrete versions of) the Virasoro Lie algebra.
 This would mean recovering an infinite summation operation
$(p_i)_{i>0}\longmapsto\sum_{i>0}L_ip_i$ with the coefficients
$L_i$, \,$i>0$, satisfying the equations similar to the above, from
the linear operators $L_i\:P\rarrow P$.
 But even though the Lie algebra $\mathrm{Vir}_+$ is generated by two
elements $L_1$ and $L_2$, the relevant completion of the enveloping
algebra $U(\mathrm{Vir}_+)$ \cite[Sections~D.5.1\+-3]{Psemi} is
\emph{not} a topological quotient algebra of the algebra of
noncommutative formal power series $k\{\{L_1,L_2\}\}$, so this example
does not reduce to Example~\ref{noncommutative-power-series}(3).
 The assumptions of Theorem~\ref{topol-ring-fully-faithful-theorem} are
not satisfied for the topological enveloping algebra of~$\Vir_+$
(still less of~$\Vir$).
 The topological enveloping algebra $U\sphat\,(\Vir)$ does not even
have a base of neighborhoods of zero consisting of two-sided open
ideals, and the two-sided open ideals in the topological enveloping
algebra $U\sphat\,(\Vir_+)$ are not strongly finitely generated
as right ideals.
\end{ex}

\begin{exs} \label{central-adic-topological-ring}
 (1)~Let $R$ be an associative ring with a two-sided ideal $I\subset R$.
Consider the associated graded ring $\gr_IR=\bigoplus_{n=0}^\infty
I^n/I^{n+1}$ and the $I$\+adic completion $\R=\varprojlim_n R/I^n$
(viewed as a topological ring in the projective limit topology).
 Assume that the ideal $\gr_II=\bigoplus_{n=1}^\infty I^n/I^{n+1}$ in
the graded ring $\gr_IR$ is generated by a finite set of central
elements $\bar s_1$,~\dots, $\bar s_m$ of grading~$1$.

 Choose some liftings $s_1$,~\dots, $s_m\in I$ of the elements
$\bar s_1$,~\dots, $\bar s_m\in I/I^2$.
 Then the open two-sided ideal $\I_n=\varprojlim_i I^n/I^{n+i}$
in the topological ring $\R$, viewed as a right ideal, is strongly
generated by the images of the monomials of degree~$n$ in
the elements~$s_1$,~\dots,~$s_m$.
 According to Theorem~\ref{topol-ring-fully-faithful-theorem},
it follows that the forgetful functor $\R\contra\rarrow R\modl$
is fully faithful.

\smallskip

 (2)~In particular, let $R$ be an associative ring and $I\subset R$
be the ideal generated by a finite set of central elements
$s_1$,~\dots, $s_m\in R$.
 Let $\R=\varprojlim_n R/I^n$ be the $I$\+adic completion of the ring
$R$, endowed with the projective limit (\,$=$~$I$\+adic) topology.
 Then the forgetful functor $\R\contra\rarrow R\modl$ is fully
faithful.
 This assertion was mentioned in
Examples~\ref{comm-ring-ideal-contramodules-examples}(2)
and~\ref{central-ideal-contramodules-examples}(3) above.

\smallskip

 (3)~Here is an alternative argument, not based on
Theorem~\ref{topol-ring-fully-faithful-theorem}, proving
the assertion in~(2).
 Let us show that the exact forgetful functor $\R\contra\rarrow
R\modl_{I\ctra}$ is fully faithful.
 The abelian category $\R\contra$ is the category of modules over
the monad on $\Sets$ assigning to a set $X$ the set $\R[[X]]=
\Lambda_I(R[X])$, while the abelian category $R\modl_{I\ctra}$ is
the category of modules over the monad assigning to a set $X$
the set $\Delta_I(R[X])$.
 The functor $\R\contra\rarrow R\modl_{I\ctra}$ is induced by
the morphism of monads $\Delta_I(R[X])\rarrow\Lambda_I(R[X])$, and
surjectivity of this map for every set $X$ immediately implies
that this functor is fully faithful.

\smallskip

 (4)~It may be worth noting that, while $\Delta_I\:R\modl\rarrow
R\modl_{I\ctra}$ is the left adjoint functor to the exact embedding
$R\modl_{I\ctra}\rarrow R\modl$, the left adjoint functor to
the exact embedding $\R\contra\rarrow R\modl$ can be computed
as the 0\+th left derived functor $\boL_0\Lambda_I$ of the $I$\+adic
completion functor $\Lambda_I\:R\modl\rarrow R\modl$.
 The functor $\Lambda_I$ itself is neither left, nor right exact.
 Its derived functor was considered in~\cite[Section~3]{PSY}.
 (Indeed, both the reflector onto $\R\contra$ in $R\modl$ and
$\boL_0\Lambda_I$ are right exact and take the left $R$\+module
$R[X]$ to the left $\R$\+contramodule $\R[[X]]$.)

 For any left $R$\+module $M$, there are natural surjective
$R$\+module morphisms $\Delta_I(M)\rarrow\boL_0\Lambda_I(M)\rarrow
\Lambda_I(M)$.
 We refer to Examples~\ref{comm-ring-ideal-good-projective-generator}(6)
and~\ref{central-ideal-good-projective-generator}(4) below for further
discussion.
\end{exs}

\begin{exs} \label{central-mult-subset-topological-ring}
 (1)~Let $R$ be an associative ring and $S\subset R$ be a multiplicative
subset consisting of some central elements in~$R$.
 Let $\R=\varprojlim_{s\in S}R/sR$ denote the $S$\+completion of
the ring $R$, viewed as a topological ring in the projective limit
topology.
 Assume that the projective limit topology of $\R$ coincides with
the $S$\+topology of $R$\+module $\R$, and moreover, assume that
for every set $X$ the projective limit topology of the free
$\R$\+contramodule $\R[[X]]=\varprojlim_{s\in S}R/sR[X]$ coincides
with the $S$\+topology of the $R$\+module $\R[[X]]$
(cf.~\cite[Theorem~2.3]{PMat}).

 The latter condition can be expressed by saying that for any
$X$\+indexed family of elements $r_x\in\R$, converging to zero in
the topology of $\R$ and belonging to the kernel ideal of the natural
ring homomorphism $\R\rarrow R/sR$, there exists an $X$\+indexed
family of elements $t_x\in\R$, converging to zero in the topology
of $\R$, such that $r_x=st_x$ for all $x\in X$.
 In other words, it means that the kernel ideal of the ring
homomorphism $\R\rarrow R/sR$ is strongly generated by (the image
in $\R$ of) an element~$s$, for every $s\in S$.

 Following the proof of Theorem~\ref{topol-ring-fully-faithful-theorem},
we would be able to conclude that the forgetful functor
$\R\contra\rarrow R\modl$ is fully faithful if we knew that,
for every left $\R$\+contramodule $C$, the equations $C=sC$ for all
$s\in S$ imply $C=0$.
 This condition holds whenever every $S$\+divisible left $R$\+module
is $S$\+h-divisible (see the discussion in~\cite[Section~1]{PMat}
and the references therein).

 To sum up, the forgetful functor $\R\contra\rarrow R\modl$ is 
fully faithful whenever the $S$\+completion of the free left
$R$\+module $R[X]$ is $S$\+complete for every set $X$ and all
the $S$\+divisible left $R$\+modules are $S$\+h-divisible.
 Notice that, unlike in
Examples~\ref{comm-ring-mult-subset-contramodules-examples}
and~\ref{central-mult-subset-contramodules-examples} above,
no condition on the projective dimension of the $R$\+module $S^{-1}R$
was needed for our present discussion.

\smallskip

 (2)~In particular, for any countable multiplicative subset $S$
consisting of some central elements in an associative ring $R$,
the forgetful functor $\R\contra\rarrow R\modl$ is fully faithful
(see~\cite[Proposition~2.2]{PMat}).
 In this case, the category $\R\contra$ is a full subcategory
in $R\modl_{S\ctra}$ (see
Examples~\ref{comm-ring-mult-subset-contramodules-examples}(2)
and~\ref{central-mult-subset-contramodules-examples}(2),
and~\cite[Lemma~1.9]{PMat}).

 While $\Delta_S$ is the left adjoint functor to the exact embedding
$R\modl_{S\ctra}\rarrow R\modl$, the left adjoint functor to the exact
embedding $\R\contra\rarrow R\modl$ can be computed as the $0$\+th
left derived functor $\boL_0\Lambda_S\:R\modl\rarrow\R\contra$ of
the $S$\+completion functor~$\Lambda_S$.
 For any left $R$\+module $M$, there are natural surjective left
$R$\+module morphisms $\Delta_S(M)\rarrow\boL_0\Lambda_S(M)\rarrow
\Lambda_S(M)$.
 We refer to
Examples~\ref{comm-ring-mult-subset-good-projective-generator}(2)
and~\ref{central-mult-subset-good-projective-generator}(2) below
for further discussion.
\end{exs}

\Section{Good Projective Generators and Perpendicular Subcategories}
\label{perpendicular-subcategories-secn}

 The definition of the perpendicular subcategory in~\cite[Section~1]{GL}
has many obvious versions and generalizations.
 In the next series of definitions, we list the most important ones
in our context.

 Let $\sA$ be an abelian category and $\sB\subset\sA$ be a full
subcategory.
 We will say that $\sB$ is a \emph{right\/ $0$\+perpendicular
subcategory} in $\sA$ if there exists a class of morphisms $\sF$ in
the category $\sA$ such that 
\begin{itemize}
\item an object $B\in\sA$ belongs to $\sB$ if and only if for
every morphism $f\in\sF$, \ $f\:U'\rarrow U''$, the morphism of abelian
groups $\Hom_\sA(f,B)\:\Hom_\sA(U'',B)\allowbreak\rarrow\Hom_\sA(U',B)$
is an isomorphism.
\end{itemize}
 In this case, we will say that $\sB$ is the right $0$\+perpendicular
subcategory to the class of morphisms $\sF$ in $\sA$ and write
$\sB=\sF^{\perp_0}\subset\sA$.

 Furthermore, we will say that $\sB$ is a \emph{right\/
$1$\+perpendicular subcategory} in $\sA$ if there exists a class of
objects $\sE$ in the category $\sA$ such that 
\begin{itemize}
\item an object $B\in\sA$ belongs to $\sB$ if and only if for
every object $E\in\sE$ one has $\Hom_\sA(E,B)=0=\Ext_\sA^1(E,B)$.
\end{itemize}
 In this case, we will say that $\sB$ is the right $1$\+perpendicular
subcategory to the class of objects $\sE$ in $\sA$ and write
$\sB=\sE^{\perp_{0,1}}\subset\sA$.

 Let $n\ge1$ be an integer.
 We will say that a full subcategory $\sB\subset\sA$ is a \emph{right\/
$n$\+perpendicular subcategory} in $\sA$ if there exists a class of
objects $\sE$ in the category $\sA$ such that 
\begin{itemize}
\item an object $B\in\sA$ belongs to $\sB$ if and only if for
every object $E\in\sE$ one has $\Hom_\sA(E,B)=0=\Ext_\sA^1(E,B)$, and
\item at the same time, for any objects $E\in\sE$ and $B\in\sB$ one has
$\Ext_\sA^i(E,B)=0$ for all $0\le i\le n$.
\end{itemize}
 So an $n$\+perpendicular subcategory is a $1$\+perpendicular
subcategory satisfying an additional condition.
 We will say that $\sB$ is the right $n$\+perpendicular subcategory to
the class of objects $\sE$ in $\sA$ and write
$\sB=\sE^{\perp_{0,1}}=\sE^{\perp_{0..n}}$.

 Finally, we will say that a full subcategory $\sB\subset\sA$ is
a \emph{right\/ $\infty$\+perpendicular subcategory} in $\sA$ if
there exists a class of objects $\sE\subset\sA$ such that
\begin{itemize}
\item an object $B\in\sA$ belongs to $\sB$ if and only if for
every object $E\in\sE$ one has $\Hom_\sA(E,B)=0=\Ext_\sA^1(E,B)$, and
\item for any objects $E\in\sE$ and $B\in\sB$, one has
$\Ext_\sA^i(E,B)=0$ for all $i\ge0$.
\end{itemize}
 In this case, we will say that $\sB$ is the right
$\infty$\+perpendicular subcategory to the class of objects $\sE$ in
$\sA$ and write $\sB=\sE^{\perp_{0,1}}=\sE^{\perp_{0..\infty}}$.

\begin{lem} \label{n+1-perpendicular-implies-n}
\textup{(a)} For any abelian category\/ $\sA$ and any integer $n\ge0$,
every right $(n+1)$\+perpendicular subcategory in $\sA$ is at the same
time a right $n$\+perpendicular subcategory in\/~$\sA$. \par
\textup{(b)} For any locally presentable abelian category\/ $\sA$ and
any integer $n\ge0$, every right $(n+1)$\+perpendicular subcategory
to a set of objects in\/ $\sA$ is at same same time a right
$n$\+perpendicular subcategory to a set of objects/morphisms in\/~$\sA$.
\end{lem}

\begin{proof}
 Both assertions are obvious for $n\ge1$, as any right
$(n+1)$\+perpendicular subcategory to a class of objects $\sE\subset\sA$
is, by the definition, at the same time a right $n$\+perpendicular
subcategory to the same class of objects $\sE\subset\sA$ for $n\ge1$.

 The nontrivial case is $n=0$.
 Part~(a): given a class of objects $\sE\subset\sA$, denote by $\sF$
the class of all $\sE$\+monomorphisms in $\sA$, i.~e., all
the monomorpisms in $\sA$ with the cokernels belonging to~$\sE$.
 Then $\sF^{\perp_0}=\sE^{\perp_{0,1}}$.
 Part~(b): given a set of $\kappa$\+presentable objects $\sE$ in
a locally $\kappa$\+presentable abelian category $\sA$ (where
$\kappa$~is a regular cardinal), denote by $\sF$ the class of all
$\sE$\+monomorphisms with $\kappa$\+presentable codomains in~$\sA$.
 Then $\sF^{\perp_0}=\sE^{\perp_{0,1}}$ \cite[Lemma~3.4]{PR}.

 In particular, in the simplest case when there are enough projective
objects in $\sA$, it suffices to pick an epimorphism from a projective
object $F_E\rarrow E$ onto every object $E\in\sE$.
 Let $G_E\rarrow F_E$ be the kernel of the morphism $F_E\rarrow E$;
then the class/set $\sF$ of all the morphisms $G_E\rarrow F_E$ has
the property that $\sF^{\perp_0}=\sE^{\perp_{0,1}}$.
\end{proof}

 Our definition of a right $0$\+perpendicular subcategory is 
(the abelian categories-related particular case of) what appears under
the name of an ``orthogonality class'' in the book~\cite{AR}
(where nonabelian, nonadditive categories are generally considered).
 What we would call ``the right $0$\+perpendicular subcategory to
a set of morphisms'' is called a ``small-orthogonality class''
in~\cite{AR}.

 The right $0$\+perpendicular subcategory to a single morphism between
\emph{free} left modules over an associative ring $R$ in the abelian
category $\sA=R\modl$ was discussed in
Example~\ref{accessible-additive-monads-examples}(4).
 Part~(d) of the next lemma generalizes some of the properties mentioned
there to the case of the right $0$\+perpendicular subcategory to
an arbitrary class of morphisms between projective objects in an abelian
category.

 What we call a right $1$\+perpendicular subcategory was called
simply a ``right perpendicular subcategory'' in~\cite{GL}.
 The case of the right perpendicular subcategory to a class of
objects of projective dimension~$\le1$ played a special role
in~\cite{GL} (and also in~\cite[Section~1]{Pcta}), where it was noticed
that such a subcategory $\sB\subset\sA$ is always abelian with an exact
embedding functor $\sB\rarrow\sA$.

 In fact, the right perpendicular subcategory to a class/set of objects
of projective dimension~$\le1$ (in the sense of~\cite{GL}) is
a right $n$\+perpendicular subcategory to the same class/set of objects 
for every $n\ge1$ (in the sense of our definitions).
 Part~(c) of the next lemma generalizes the related results
of~\cite[Proposition~1.1]{GL} and~\cite[Theorem~1.2]{Pcta} to
the case of right $2$\+perpendicular subcategories to arbitrary
classes of objects.

\begin{lem} \label{right-perpendicular-closure-properties}
 Let\/ $\sA$ be an abelian category.  Then \par
\textup{(a)} any right\/ $0$\+perpendicular subcategory\/
$\sB\subset\sA$ is closed under arbitrary limits, and in particular,
under infinite products and kernels in\/~$\sA$; \par
\textup{(b)} any right\/ $1$\+perpendicular subcategory\/
$\sB\subset\sA$ is closed under infinite products, kernels, and
extensions in\/~$\sA$; \par
\textup{(c)} any right\/ $2$\+perpendicular subcategory\/
$\sB\subset\sA$ is closed under infinite products, kernels, extensions,
and cokernels in\/~$\sA$.
 Hence any right\/ $2$\+perpendicular subcategory\/ $\sB\subset\sA$
is abelian and its embedding functor\/ $\sB\rarrow\sA$ is exact; \par
\textup{(d)} the right\/ $0$\+perpendicular subcategory\/
$\sB\subset\sA$ to any class of morphisms between projective objects
in $\sA$ is closed under infinite products, kernels, extensions, and
cokernels in\/~$\sA$.
 Hence any such full subcategory\/ $\sB\subset\sA$ is abelian
and its embedding functor\/ $\sB\rarrow\sA$ is exact. \par
\end{lem}

\begin{proof}
 Part~(a) is~\cite[Observation~1.34]{AR}.
 Part~(b) is~\cite[Proposition~1.1]{GL}.
 In part~(c), one notices that any full subcategory in $\sA$ closed
under the kernels of all morphisms and the cokernels of monomorphisms
is also closed under the cokernels of all morphisms.
 Checking that $\sB=\sE^{\perp_{0,1}}=\sE^{\perp_{0..2}}$ is closed under
the cokernels of monomorphisms is easy.
 Part~(d) holds, because the full subcategory of isomorphisms in
the category of morphisms in the category of abelian groups (or in
any other abelian category generally) is closed under infinite
products, kernels, extensions, and cokernels
(see~\cite[proof of Theorem~1.2 and Remark~1.3]{Pcta} for
an alternative argument).
\end{proof}

\begin{ex} \label{sheaves-cosheaves-ex}
 The following notable examples of left and right perpendicular
subcategories serve to illustrate the above definitions rather well.
 Let $X$ be a scheme with the structure sheaf~$\O_X$.
 Consider the following preadditive category (or, which is the same,
a \emph{ring with several objects}~\cite{Mit})~$\cA_X$.
 The objects of $\cA_X$ are affine open subschemes $U\subset X$.
 For any two affine open subschemes $U$ and $V\subset X$, the group of
morphisms $V\rarrow U$ in $\cA_X$ is the group $\O_X(V)$ if $V\subset U$ 
and the zero group otherwise.
 The compositions of morphisms in $\cA_X$ are defined in
the obvious way.

 Then the category $\modr\cA_X$ of right $\cA_X$\+modules (or, which is
the same, contravariant functors from $\cA_X$ to the category of abelian 
groups $\boZ\modl$) is the category of presheaves of $\O_X$\+modules
on the affine open subschemes in~$X$.
 The category $\cA_X\modl$ of left $\cA_X$\+modules (\,$=$~covariant
functors from $\cA_X$ to $\boZ\modl$) is the category of copresheaves
of $\O_X$\+modules on the affine open subschemes in~$X$.

 Given an affine open subscheme $W\subset X$, denote by $\cP_W$
the left $\cA_X$\+module assigning to an affine open subscheme
$V\subset X$ the group $\cP_W(V)=\O_X(W)$ if $W\subset V$ and
$0$~otherwise.
 So $\cP_W$ is the covariant functor $\cA_X\rarrow\boZ\modl$
corepresented by the object $W$, or, in another language, the free
left $\cA_X$\+module with one generator sitting at the object~$W$.
 Furthermore, given two embedded affine open subschemes $W\subset U$
in $X$, denote by $\cE_{W,U}$ the left $\cA_X$\+module assigning to
an affine open subscheme $V\subset X$ the group $\cE_{W,U}(V)=
\O_X(W)$ if $W\subset V\varsubsetneq U$ and $0$~otherwise.
 Then there is a short exact sequence $0\rarrow \O_X(W)\ot_{\O_X(U)}
\cP_U\rarrow \cP_W\rarrow\cE_{W,U}\rarrow0$ in the abelian
category $\cA_X\modl$ (where the tensor product $\O_X(W)\ot_{\O_X(U)}
\cP_U$ assigns to an affine open subscheme $V\subset X$ the group
$\O_X(W)\ot_{\O_X(U)}\cP_U(V)$).

 Furthermore, denote by $\cJ_{W,U}$ the right $\cA_X$\+module
assigning to an affine open subscheme $V\subset X$ the abelian group
$\cJ_{W,U}(V)=\Hom_\boZ(\cE_{W,U}(V),\mathbb Q/\boZ)$.
 Then the left $\cA_X$\+module $\cE_{W,U}$ has projective dimension
at most~$2$ (as an object of the abelican category $\cA_X\modl$) and
flat dimension at most~$1$ (in the sense of the theory of the tensor
product functor $\ot_{\cA_X}$ and the derived tensor product functor
$\Tor_*^{\cA_X}$ over a preadditive category $\cA_X$, as developed,
e.~g., in~\cite[Section~6]{Mit}), while the right $\cA_X$\+module
$\cJ_{W,U}$ has injective dimension at most~$1$.

 Now the description of the category of quasi-coherent sheaves
$X\qcoh$ on a scheme $X$ suggested by Enochs and Estrada
in~\cite[Section~2]{EE} can be reformulated by saying that $X\qcoh$
is equivalent to the full subcategory in $\modr\cA_X$ consisting of
all the presheaves of $\O_X$\+modules $\cF$ on affine open subschemes
in $X$ such that $\cF\ot_{\cA_X}\cE_{W,U}=0=
\Tor^{\cA_X}_1(\cF,\cE_{W,U})$ for all affine open subschemes
$W\subset U$ in $X$, or equivalently, $\Hom_{\cA_X^\rop}(\cF,\cJ_{W,U})
=0=\Ext^1_{\cA_X^\rop}(\cF,\cJ_{W,U})$ for all $W\subset U\subset X$.
 Indeed, for any presheaf $\cF$ the groups $\cF\ot_{\cA_X}\cE_{W,U}$
and $\Tor^{\cA_X}_1(\cF,\cE_{W,U})$ are, respectively, the cokernel
and the kernel of the natural map $\O_X(W)\ot_{\O_X(U)}\cF(U)
\rarrow\cF(W)$.
 So, in other words, one can say that $X\qcoh$ is the left
$1$\+perpendicular (and $\infty$\+perpendicular) subcategory to
a set of objects of injective dimension~$\le\nobreak1$ in
$\modr\cA_X$.

 Similarly, the description of the category of contraherent cosheaves
$X\ctrh$ on a scheme $X$ in~\cite[Section~2.2]{Pcosh} can be
reformulated by saying that $X\ctrh$ is equivalent to the full
subcategory in $\cA_X\modl$ consisting of all the copresheaves of
$\O_X$\+modules $\cQ$ on affine open subschemes in $X$ such that
$\Ext_{\cA_X}^i(\cE_{W,U},\cQ)=0$ for all pairs of embedded affine
open subschemes $W\subset U$ in $X$ and all $0\le i\le 2$ (or,
equivalently, for all $i\ge0$).
 In fact, the contraherence condition on $\cQ$ means the vanishing of
$\Ext_{\cA_X}^i(\cE_{W,U},\cQ)$ for $i=0$ and~$1$, while
the contraadjustness condition is equivalent to this vanishing
for $i=2$.
 More precisely, for any copresheaf $\cQ$ the groups
$\Hom_{\cA_X}(\cE_{W,U},\cQ)$ and $\Ext^1_{\cA_X}(\cE_{W,U},\cQ)$ are, 
respectively, the kernel and the cokernel of the natural map
$\cQ(W)\rarrow\Hom_{\O_X(U)}(\O_X(W),\cQ(U))$, while the group
$\Ext_{\cA_X}^2(\cE_{W,U},\cQ)$ is isomorphic to
$\Ext_{\O_X(U)}^1(\O_X(W),\cQ(U))$.

 So, denoting by $\sA$ the abelian category $\cA_X\modl$ and by
$\sE$ the set of objects $\{\cE_{W,U}\}$, one can write that
$X\ctrh=\sE^{\perp_{0..2}}\varsubsetneq\sE^{\perp_{0,1}}$ in general.
 Hence it is clear that the full subcategory $X\ctrh\subset\cA_X\modl$
is closed under extensions, direct summands, and infinite products.
 But it is not closed under kernels or cokernels, in general (as
the example of $X=\Spec\boZ$ already demonstrates).
 Moreover, one can see that there are morphisms in the category of
contraherent cosheaves over $X=\Spec\boZ$ that do not have kernels
in $X\ctrh$ at all.
 Thus, for an arbitrary scheme $X$, the category $X\ctrh$ has a natural
exact category structure, but it is not an abelian category, in general
(while the category $X\qcoh$ is abelian).
\end{ex}

 Now we return to the discussion of right $n$\+perpendicular
subcategories $\sB$ in abelian categories~$\sA$.
 In the rest of this paper (with the exception of the last
Section~\ref{nonabelian-secn}, where counterexamples are discussed),
we are interested in the right $n$\+perpendicular subcategories
$\sB\subset\sA$ that are \emph{closed under cokernels}.
 According to Lemmas~\ref{n+1-perpendicular-implies-n}(a)
and~\ref{right-perpendicular-closure-properties}(a),
for any $n\ge0$, such a full subcategory $\sB$ is closed under
infinite products, kernels, and cokernels in $\sA$; so $\sB$ is
an abelian category and its embedding functor $\sB\rarrow\sA$ is exact.
 In other words, we are interested in the abelian, exactly
embedded right $n$\+perpendicular subcategories in abelian categories.
 According to Lemma~\ref{right-perpendicular-closure-properties}(c),
for $n\ge2$ this additional condition holds automatically.

\begin{lem} \label{perpendicular-to-set-loc-pres-and-reflective}
 Let\/ $\sA$ be a locally presentable abelian category and\/
$\sB\subset\sA$ be an abelian, exactly embedded right
$n$\+perpendicular subcategory to a set of objects or morphisms
in\/~$\sA$.
 Then\/ $\sB$ is a locally presentable abelian category, accessibly
embedded into\/ $\sA$ and reflective in\/~$\sA$.
\end{lem}

\begin{proof}
 In view of Lemma~\ref{n+1-perpendicular-implies-n}(b), it suffices
to consider the case of the right $0$\+perpendicular subcategory
$\sB=\sF^{\perp_0}$ to a set of morphisms $\sF$ in~$\sA$.
 Suppose that $\sA$ is locally $\kappa$\+presentable and $\sF$
consists of morphisms between $\kappa$\+presentable objects
(where $\kappa$~is a regular cardinal).
 Then the full subcategory $\sB$ is $\kappa$\+accessibly embedded
into $\sA$ (which means that it is closed under $\kappa$\+filtered
colimits in~$\sA$).
 As the full subcategory $\sB\subset\sA$ is also closed under
arbitrary limits, \cite[Theorem and Corollary~2.48]{AR} apply,
proving that $\sB$ is locally $\kappa$\+presentable and
reflective in~$\sA$.

 Alternatively, one can apply directly~\cite[Theorem~1.39]{AR}.
 (Notice that the proof of the implication (ii)\,$\Longrightarrow$\,(i)
in~\cite[Theorem~1.39]{AR} is erroneous, and the implication
itself is only true for uncountable cardinals, see~\cite{HAR,HR} and
the references therein; but we are not using this implication here.)
\end{proof}

 Assuming Vop\v enka's principle, any right $0$\+perpendicular
subcategory in a locally presentable abelian category is
the right $0$\+perpendicular subcategory to a set of
morphisms~\cite[Corollary~6.24]{AR}, so this condition can be dropped.

 The following lemma is a direct generalization of
Example~\ref{additive-monads-examples}(4).

\begin{lem} \label{reflective-projective-generator}
 Let\/ $\sA$ be an abelian category with a projective generator, and
let\/ $\sB$ be an abelian, exactly embedded, reflective full
subcategory in\/~$\sA$.
 Then\/ $\sB$ is also an abelian category with a projective generator.
\end{lem}

\begin{proof}
 Denote the reflector by $\Delta\:\sA\rarrow\sB$, and let $Q$ be
a projective generator of~$\sA$.
 Then $P=\Delta(Q)$ is a projective generator of~$\sB$. 
\end{proof}

 From this point on, we are interested in the abelian, exactly
embedded right $n$\+perpendicular subcategories $\sB$ to sets
of objects or morphisms in the categories of modules over
associative rings, $\sA=R\modl$.
 According to
Lemmas~\ref{perpendicular-to-set-loc-pres-and-reflective}
and~\ref{reflective-projective-generator}, any such abelian
category $\sB$ is locally presentable and has a projective
generator.
 Our aim is to prove the converse assertion, for every $n\ge0$
and $n=\infty$.

 Let $\sB$ be a locally $\kappa$\+presentable abelian category, $R$
be an associative ring, and $\Theta\:\sB\rarrow R\modl$ be an exact
functor preserving infinite products and $\kappa$\+filtered
colimits.
 Then the functor $\Theta$ is a right adjoint~\cite[Theorem~1.66]{AR},
so it has a left adjoint functor $\Delta\:R\modl\rarrow\sB$.
 Set $P=\Delta(R)$, where $R$ is the free left $R$\+module with one
generator; then $P$ is a projective object in $\sB$ endowed with
a right action of the ring $R$, that is, with a homomorphism of
associative rings $R\rarrow\Hom_\sB(P,P)^\rop$.
 The functor $\Theta$ is corepresented by the object $P\in\sB$: one has
$\Theta(B)=\Hom_\sB(P,B)$ for all $B\in\sB$, with the left $R$\+module
structure on the abelian group $\Hom_\sB(P,B)$ coming from the right
action of $R$ in~$P$.

 The projective object $P$ is a projective generator of $\sB$ if and
only if the exact functor $\Theta$ is \emph{conservative}, i.~e., it
takes nonisomorphisms to nonisomorphisms, or equivalently, takes
nonzero objects to nonzero objects.
 An exact functor between abelian categories is conservative if and
only if it is faithful.
 Thus conservative exact functors $\Theta\:\sB\rarrow R\modl$
preserving infinite products and $\kappa$\+filtered colimits are
indexed by projective generators $P\in\sB$ endowed with
an associative ring homomorphism $R\rarrow\Hom_\sB(P,P)^\rop$.
 We are interested in those of such functors $\Theta$ that are not only
conservative, but, actually, fully faithful.

 The following theorem is the main result of this section.

\begin{thm} \label{ext-isomorphism-right-perpendicular-thm}
 Let\/ $\sB$ be a locally presentable abelian category with
a projective generator $P$, let $R$ be an associative ring, and let
$\theta\:R\rarrow\Hom_\sB(P,P)^\rop$ be an associative ring
homomorphism.
 Assume that the related functor\/ $\Theta=\Hom_\sB(P,{-})\:\sB\rarrow
R\modl$ is fully faithful.  Then \par
\textup{(a)} the full subcategory\/ $\Theta(\sB)\subset R\modl$ is
the right\/ $0$\+perpendicular subcategory to a set of morphisms in
$R\modl$; \par
\textup{(b)} assuming additionally that $\theta$~is injective, the full
subcategory\/ $\Theta(\sB)\subset R\modl$ is a right\/
$1$\+perpendicular subcategory (to a set of objects) in $R\modl$ if and
only if it is closed under extensions; \par
\textup{(c)} assuming that $\theta$~is injective and given an integer
$n\ge1$, the full subcategory\/ $\Theta(\sB)\subset R\modl$
is the right $n$\+perpendicular subcategory to a set of objects in
$R\modl$ whenever the induced maps between the Ext groups\/
$\Theta\:\Ext_\sB^i(B,C)\rarrow\Ext_R^i(\Theta(B),\Theta(C))$ are
isomorphisms for all $B$, $C\in\sB$ and $i\le n$; \par
\textup{(d)} assuming that $\theta$~is injective, the full subcategory\/
$\Theta(\sB)\subset R\modl$ is the right\/ $\infty$\+perpendicular
subcategory to a set of objects in $R\modl$ whenever the induced maps\/
$\Theta\:\Ext_\sB^i(B,C)\rarrow\Ext_R^i(\Theta(B),\Theta(C))$ are
isomorphisms for all $B$, $C\in\sB$ and all\/ $0\le i<\infty$.
\end{thm}

\begin{proof}
 According to Example~\ref{additive-monads-examples}(3), the functor
$\Hom_\sB(P,{-})$ identifies the category $\sB$ with the category of
modules over the additive monad $\boT\:X\longmapsto\Hom_\sB(P,P^{(X)})$.
 Assume that the category $\sB$ is locally $\kappa$\+presentable and
the object $P$ is $\kappa$\+presentable; then the monad $\boT$ is
$\kappa$\+accessible.
 Identifying $\sB$ with $\boT\modl$, the functor $\Theta$ gets
identified with the forgetful functor $\Hom_\boT(\boT(*),{-})$ discussed
in the beginning of
Section~\ref{examples-of-fully-faithful-forgetful-secn}.
 In particular, for every set $X$ we have the free $\boT$\+module
$\boT(X)\in\boT\modl$ corresponding to the object $P^{(X)}\in\sB$,
and there is a natural left $R$\+module morphism
$R[X]\rarrow\boT(X)$, as in
Proposition~\ref{fully-faithful-forgetful-functor-image-identified}.
 Denote this morphism by~$\theta_X$.

 Part~(a): we claim that a left $R$\+module $C$ belongs to
the full subcategory $\Theta(\sB)\subset R\modl$ if and only if
the morphism of abelian groups
$$
 \Hom_R(\theta_Z,C)\:\Hom_R(\boT(Z),C)\lrarrow\Hom_R(R[Z],C)
$$
is an isomorphism for all sets $Z$ of cardinality less than~$\kappa$.
 Indeed, if $C\in\Theta(\sB)$, then
$$
 \Hom_R(\boT(X),C)\simeq C^X\simeq\Hom_R(R[X],C)
$$
for all sets $X$, since the functor $\Theta$ is fully faithful by
assumption.

 Conversely, assume that $\Hom_R(\theta_Z,C)$ is an isomorphism for
all sets $Z$ of cardinality less than~$\kappa$.
 Then $\Hom_R(\theta_X,C)$ is an isomorphism for all sets $X$,
because $R[X]=\varinjlim_{Z\subset X}R[Z]$ and
$\boT(X)=\varinjlim_{Z\subset X}\boT(Z)$ in the category of left
$R$\+modules, where the $\kappa$\+filtered colimit is taken over 
all the subsets $Z\subset X$ of cardinality less than~$\kappa$.

 In particular, we have $\Hom_R(\boT(C),C)=\Hom_R(R[C],C)=C^C$,
so there is a natural surjective morphism of left $R$\+modules
$\boT(C)\rarrow C$ corresponding to the identity map $C\rarrow C$.
 Let $K$ denote the kernel of this morphism; then the morphism
$\Hom_R(\theta_X,K)$ is an isomorphism for all sets $X$, since
the morphisms $\Hom_R(\theta_X,\boT(C))$ and
$\Hom_R(\theta_X,C)$ are.
 Hence we also have a natural surjective morphism of left
$R$\+modules $\boT(K)\rarrow K$.
 Now $C$ is the cokernel of the composition $\boT(K)\rarrow K
\rarrow\boT(C)$, and it follows that $C\in\Theta(\sB)$, since
$\boT(K)$, $\boT(C)\in\Theta(\sB)$ and the full subcategory
$\Theta(\sB)\subset R\modl$ is closed under cokernels.

 Part~(b): obviously, $\Theta(\sB)$ is closed under extensions in
$R\modl$ if and only if $\Theta$ induces isomorphisms on the groups
$\Ext^1$.
 If $\Theta(\sB)$ is a right $1$\+perpendicular subcategory in
$R\modl$, then $\Theta(\sB)\subset R\modl$ is closed under extensions
by Lemma~\ref{right-perpendicular-closure-properties}(b).

 Conversely, if $\Theta$ induces isomorphisms on $\Ext^1$, then
$\Ext^1_R(\boT(X),C)=0$ for all $C\in\Theta(\sB)$, as
the free $\boT$\+modules $\boT(X)$ are projective objects in
$\boT\modl$.
 We claim that a left $R$\+module belongs to the full subcategory
$\Theta(\sB)\subset R\modl$ if and only if
$$
 \Hom_R(\boT(Z)/R[Z],C)=0=\Ext^1_R(\boT(Z)/R[Z],C)
$$
for all sets $Z$ of cardinality less than~$\kappa$, where
$\boT(Z)/R[Z]$ is the cokernel of the morphism~$\theta_Z$.
 Notice that the morphism $\theta_Z$ is injective in our
present assumptions (since the homomorphism~$\theta$ is injective).

 Indeed, it suffices to consider the long exact sequence
\begin{multline*}
 0\lrarrow\Hom_R(\boT(Z)/R[Z],C)\lrarrow\Hom_R(\boT(Z),C) \\
 \lrarrow\Hom_R(R[Z],C) \lrarrow\Ext^1_R(\boT(Z)/R[Z],C)
 \lrarrow\Ext^1_R(\boT(Z),C)
\end{multline*}
and recall that a left $R$\+module $C$ belongs to $\Theta(\sB)$
if and only if $\Hom_R(\theta_Z,C)$ is an isomorphism for all sets $Z$
of the cardinality less than~$\kappa$, and that $\Ext^1_R(\boT(X),C)=0$
for all sets $X$ and all $C\in\Theta(\sB)$.

 Part~(c): for every set $X$, we have $\Ext^i_R(\boT(X)/R[X],C)=
\Ext^i_R(\boT(X),C)$ for all left $R$\+modules $C$ and all $i\ge2$.
 If the functor $\Theta$ induces isomorphisms of the groups $\Ext^i$
for all $i\le n$, then one has
$$
 \Ext^i_R(\boT(X)/R[X],C)=\Ext^i_R(\boT(X),C)=
 \Ext^i_\sB(P^{(X)},\Theta^{-1}(C))=0
$$ 
for all $C\in\Theta(\sB)$ and $2\le i\le n$, where $\Theta^{-1}(C)
\in\sB$ is an object such that $\Theta(\Theta^{-1}(C))=C$.
 Hence $\Theta(\sB)=\sE^{\perp_{0,1}}=\sE^{\perp_{0..n}}$, where $\sE
\subset R\modl$ is the set of all the left $R$\+modules
$\boT(Z_\mu)/R[Z_\mu]$, where $\mu$ runs over all the cardinals
smaller than~$\kappa$ and $Z_\mu$ is a set of cardinality~$\mu$.

 Part~(d) is provable in the same way as part~(c).
\end{proof}

\begin{rem} \label{good-perp-counterex-remark}
 It is clear from the proof of
Theorem~\ref{ext-isomorphism-right-perpendicular-thm}(c) that,
denoting by $\sE\subset R\modl$ the set of all left $R$\+modules
$\boT(Z_\mu)/R[Z_\mu]$ and assuming that $\Theta(\sB)=
\sE^{\perp_{0,1}}$, one has $\sE^{\perp_{0,1}}=\sE^{\perp_{0..n}}$
if and only if $\Ext^i_R(\boT(Z_\mu),C)=0$ for all the cardinals
$\mu$~involved, all $C\in\Theta(\sB)$, and all $1\le i\le n$.
 If we knew that, moreover, $\Ext^i_R(\boT(X),C)=0$ for all sets $X$,
all objects $C\in\Theta(\sB)$, and all $1\le i\le n$, it would follow
that the functor $\Theta\:\sB\rarrow R\modl$ induces isomorphisms
of the groups $\Ext^i$ for $i\le n$.

 Nevertheless, there are examples of associative (and even commutative)
rings $R$ with a set $\sE\subset R\modl$ of left $R$\+modules of
projective dimension~$1$ for which the embedding functor
$\Theta\:\sB\rarrow R\modl$ of the right $\infty$\+perpendicular
subcategory $\sB=\sE^{\perp_{0,1}}=\sE^{\perp_{0..\infty}}$
does \emph{not} induce an isomorphism of the groups $\Ext^i$ for 
for some $i\ge2$.
 Such examples can even be found among the embedding functors
$\Theta\:\sB\rarrow R\modl$ corepresented by projective generators
$P\in\sB$ with $R=\Hom_\sB(P,P)^\rop$
(see Example~\ref{comm-ring-ideal-good-projective-generator}(5) below).
\end{rem}

\begin{prop} \label{fully-faithful-b-implies-minus}
 Let\/ $\sB$ be a locally presentable abelian category with a projective
generator $P$, let $R$ be an associative ring, let\/
$\theta\:R\rarrow\Hom_\sB(P,P)^\rop$ be a ring homomorphism, and let\/
$\Theta=\Hom_\sB(P,{-})\:\sB\rarrow R\modl$ be the related
exact functor.
 Then the induced maps\/ $\Theta\:\Ext^i_\sB(B,C)\rarrow\Ext^i_R
(\Theta(B),\Theta(C))$ are isomorphisms for all $B$, $C\in\sB$ and
all\/ $0\le i<\infty$ if and only if the induced triangulated functor
between the bounded above derived categories\/
$\Theta\:\sD^-(\sB)\rarrow\sD^-(R\modl)$ is fully faithful.
\end{prop}

\begin{proof}
 Clearly, an exact functor between abelian categories $\sB\rarrow\sA$
induces isomorphisms on all the Ext groups if and only if the induced
triangulated functor between the bounded derived categories
$\sD^\b(\sB)\rarrow\sD^\b(\sA)$ is fully faithful.
 According to Proposition~\ref{fully-faithful-on-bounded-above-derived}
below, the functor $\sD^-(\sB)\rarrow\sD^-(\sA)$ is fully faithful if
and only if the functor $\sD^\b(\sB)\rarrow\sD^\b(\sA)$ is, provided
that there are enough projective objects in the category $\sA$ and
the exact functor $\sB\rarrow\sA$ has a left adjoint
$\Delta\:\sA\rarrow\sB$.

 In the situation at hand, obviously there are enough projectives in
$R\modl$, so it remains to recall from the discussion before
Theorem~\ref{ext-isomorphism-right-perpendicular-thm} that
the functor $\Theta$ has a left adjoint.
 Alternatively, one can construct the functor $\Delta$ explicitly, as
the unique right exact functor taking the free left $R$\+module $R[X]$
to the object $P^{(X)}\in\sB$ for all sets $X$ and acting on morphisms
between free left $R$\+modules in the natural way (determined by
the ring homomorphism~$\theta$).
\end{proof}

 Now let us restrict to the particular case when the ring $R$
coincides with the endomorphism ring $\Hom_\sB(P,P)^\rop$ and
$\theta$~is the identity isomorphism.
 Let $\sB$ be a locally presentable abelian category with
a projective generator~$P$.
 We will say that the projective generator $P\in\sB$ is
\emph{$0$\+good} if the conservative exact functor
$\Theta_P=\Hom_\sB(P,{-})\:\sB\rarrow R\modl$, where
$R=\Hom_\sB(P,P)^\rop$, is fully faithful.

 A projective generator $P\in\sB$ is $0$\+good if and only if the full
subcategory on one object $\{P\}\subset\sB$ is \emph{dense} in
the category $\sB$ \cite[Proposition~1.26(i)]{AR}
(cf.~\cite[Remark~1.23 and Example~1.24(4)]{AR}).
 Such full subcategories were called ``left adequate'' in
the terminology of the paper~\cite{Isb}.

 Furthermore, we will say that a projective generator $P\in\sB$
is \emph{$1$\+good} if the functor $\Theta_P$ is fully faithful
and the full subcategory $\Theta_P(\sB)\subset R\modl$ is closed
under extensions.
 More generally, for any $n\ge0$, we will say that a projective
generator $P\in\sB$ is \emph{$n$\+good} if the functor
$\Theta_P$ induces isomorphisms on the groups $\Ext^i$ for all
$0\le i\le n$.

 Finally, we will say that a projective generator $P\in\sB$ is
\emph{good} (or \emph{$\infty$\+good}) if it is $n$\+good for
all $n\ge0$.
 According to Proposition~\ref{fully-faithful-b-implies-minus},
a projective generator $P\in\sB$ is good if and only if the functor
$\Theta_P$ induces a fully faithful functor between the bounded
above derived categories
$$
 \sD^-(\sB)\lrarrow\sD^-(R\modl).
$$
 Clearly, any $(n+1)$\+good projective generator, $n\ge0$, is
at the same time $n$\+good.

 The following corollary summarizes the assertions of
Theorem~\ref{ext-isomorphism-right-perpendicular-thm} in
the case $\theta=\id$.

\begin{cor} \label{good-implies-n-perpendicular}
 Let\/ $\sB$ be a locally presentable abelian category and $P\in\sB$
be a projective generator.  Then \par
\textup{(a)} if $P$ is\/ $0$\+good, then the full subcategory\/
$\Theta_P(\sB)\subset R\modl$ is an abelian, exactly embedded right\/
$0$\+perpendicular subcategory to a set of morphisms in $R\modl$; \par
\textup{(b)} assume that $P$ is\/ $0$\+good; then $P$ is\/ $1$\+good if
and only if\/ $\Theta_P(\sB)\subset R\modl$ is a right\/
$1$\+perpendicular subcategory (to a set of objects) in $R\modl$; \par
\textup{(c)} if $P$ is\/ $n$\+good for some $n\ge1$, then\/
$\Theta_P(\sB)\subset R\modl$ is a right $n$\+perpendicular
subcategory to a set of objects in $R\modl$; \par
\textup{(d)} if $P$ is good, then\/ $\Theta_P(\sB)\subset R\modl$ is
a right\/ $\infty$\+perpendicular subcategory to a set of objects
in $R\modl$.  \qed
\end{cor}

 For any cardinal $\lambda$, we denote its successor cardinar
by~$\lambda^+$.
 The proof of the following theorem will be given in
Section~\ref{lambda-flat-secn}.

\begin{thm} \label{copower-good-projective-generator}
 Let\/ $\sB$ be an abelian category with a projective generator $P$
and $\lambda$~be a cardinal such that the category\/ $\sB$ is
locally $\lambda^+$\+presentable and the object $P\in\sB$ is
$\lambda^+$\+presentable.
 Then $Q=P^{(\lambda)}$ is a good projective generator of\/~$\sB$.
\end{thm}

 The following corollary shows that the classes of abelian categories
embodied by the (abelian and exactly embedded) right $n$\+perpendicular
subcategories to sets of objects/morphisms in the categories of
modules over associative rings are the same for all integers $n\ge0$
and for $n=\infty$.

\begin{cor} \label{embodied-abelian-categories}
 The following classes of abelian categories (viewed as abstract
categories irrespectively of a particular embedding into a category of
modules) coincide with each other, and with the class of all locally
presentable abelian categories with a projective generator:
\begin{itemize}
\item abelian, exactly and accessibly embedded full subcategories in
the categories of modules over associative rings, closed under
infinite products;
\item abelian, exactly embedded right\/ $0$\+perpendicular subcategories
to sets of morphisms in the categories of modules over associative
rings;
\item abelian, exactly embedded right\/ $1$\+perpendicular subcategories
to sets of objects in the categories of modules over associative rings;
\item right\/ $n$\+perpendicular subcategories to sets of objects in
the categories of modules over associative rings (where $n\ge2$~is some
fixed integer);
\item right\/ $\infty$\+perpendicular subcategories to sets of objects
in the categories of modules over associative rings.
\end{itemize}
\end{cor}

\begin{proof}
 Any right $(n+1)$\+perpendicular subcategory to a set of
objects in $R\modl$ is at the same time a right $n$\+perpendicular
subcategory to a set of objects/morphisms in $R\modl$ by
Lemma~\ref{n+1-perpendicular-implies-n}.
 Any right $2$\+perpendicular subcategory in an abelian category
is abelian and exactly embedded by
Lemma~\ref{right-perpendicular-closure-properties}.

 Any right $0$\+perpendicular subcategory to a set of morphisms
in $R\modl$ is accessibly embedded and closed under infinite products
(in fact, arbitrary limits) by~\cite[Observation~1.34 and
Proposition~1.35]{AR}.
 Any abelian, exactly and accessively embedded full subcategory in
$R\modl$ is locally presentable and reflective
by~\cite[Theorem and Corollary~2.48]{AR} (see
also~\cite[Theorem~1.39]{AR} and
Lemma~\ref{perpendicular-to-set-loc-pres-and-reflective} above).
 Any abelian, exactly embedded, reflective full subcategory
in $R\modl$ has a projective generator by
Lemma~\ref{reflective-projective-generator}.
 This proves that all the itemized classes of categories consist
of locally presentable abelian categories with a projective
generator.

 Conversely, if $\sB$ is a locally $\kappa$\+presentable abelian
category with a $\kappa$\+presentable projective generator $P$,
and $\lambda$ is a cardinal such that $\lambda^+\ge\kappa$, then
$Q=P^{(\lambda)}$ is a good projective generator of $\sB$ by
Theorem~\ref{copower-good-projective-generator}.
 Set $S=\Hom_\sB(Q,Q)^\rop$; then the functor $\Theta_Q\:\sB\rarrow
S\modl$ is fully faithful and its image is the right
$\infty$\+perpendicular subcategory to a set of objects in
$S\modl$ by Corollary~\ref{good-implies-n-perpendicular}(d).
\end{proof}

 Under Vop\v enka's principle, the set-theoretical conditions
(``accessibly embedded'', ``to sets of'') in
Corollary~\ref{embodied-abelian-categories} can be dropped.
 The related classes of abelian categories (that is, the classes
of all abelian, exactly embedded right $n$\+perpendicular
subcategories to classes of objects or morphisms in the categories
of modules over associative rings, for various values of~$n$) will
still be the same, and coincide with the class of all locally
presentable abelian categories with a projective generator.
 This is provable using the above arguments together
with~\cite[Corollary~6.24]{AR}.

\Section{Examples of Good Projective Generators}
\label{examples-of-good-projective-generators-secn}

 The aim of this section is to provide some examples to the theory
developed in Section~\ref{perpendicular-subcategories-secn},
mostly based on Examples~\ref{additive-monads-examples}\+-%
\ref{accessible-additive-monads-examples} and examples from
Sections~\ref{examples-of-fully-faithful-forgetful-secn}\+-%
\ref{topological-rings-secn}.

\begin{exs} \label{topological-ring-0-good}
 (1)~Let $\R$ be a complete, separated topological associative ring
with a base of neighborhoods of zero formed by open right ideals,
$R$ be an associative ring, and $\theta\:R\rarrow\R$ be an associative
ring homomorphism such that the forgetful functor
$\R\contra\rarrow R\modl$ is fully faithful.
 Then the forgetful functor $\R\contra\rarrow\R\modl$ is also fully
faithful.
 So, by the definition, the free left $\R$\+contramodule with one
generator $P=\R=\R[[*]]$ is a $0$\+good projective generator of
$\R\contra$.
 This includes all the topological rings satisfying the assumptions
of Theorem~\ref{topol-ring-fully-faithful-theorem}, and in particular,
all the topological rings from
Examples~\ref{noncommutative-power-series}
and~\ref{central-adic-topological-ring} (and those topological
rings from Examples~\ref{central-mult-subset-topological-ring}
which satisfy the conditions mentioned there).

 See Examples~\ref{comm-ring-ideal-good-projective-generator}(6\+8),
\ref{central-ideal-good-projective-generator}(4),
\ref{comm-ring-mult-subset-good-projective-generator}(2),
and~\ref{central-mult-subset-good-projective-generator}(2)
for further discussion.

\smallskip

 (2)~Let $\R$ be a complete, separated topological associative ring
with a base of neighborhoods of zero formed by open right ideals,
$R$ be an associative ring, and $\theta\:R\rarrow\R$ be an associative
ring homomorphism such that the related forgetful functor
$\Theta\:\R\contra\rarrow R\modl$ is fully faithful.
 Then the argument from the proof of
Theorem~\ref{ext-isomorphism-right-perpendicular-thm} shows how
to distinguish the objects of the full subcategory $\R\contra$
among the objects of the ambient category $R\modl$.

 Specifically, let $Y$ be a set of the cardinality greater or
equal to the cardinality of a base of neighborhoods of zero in~$\R$.
 Denote by $\theta_Y$ the natural $R$\+module morphism
$R[Y]\rarrow\R[[Y]]$.
 Then a left $R$\+module $C$ belongs to the full subcategory
$\Theta(\R\contra)\subset R\modl$ if and only if the morphism of
abelian groups $\Hom_R(\theta_Y,C)\:\Hom_R(\R[[Y]],C)\rarrow
\Hom_R(R[Y],C)=C^Y$ is an isomorphism.
 (Notice that if $Z$ is a subset in $Y$ and the map
$\Hom_R(\theta_Y,C)$ is an isomorphism, then the map
$\Hom_R(\theta_Z,C)$ is also an isomorphism, because the left
$R$\+module morphism $\theta_Z\:R[Z]\rarrow\R[[Z]]$ is
a direct summand of the morphism~$\theta_Y$.)

 In particular, in the context of
Theorem~\ref{topol-ring-fully-faithful-theorem},
it suffices to take a countable set~$Y$.
\end{exs}

\begin{exs} \label{comm-ring-ideal-good-projective-generator}
 (1)~Let $R$ be a Noetherian commutative ring and $I\subset R$ be
a finitely generated ideal; denote by $\R=\varprojlim_n R/I^n$
the $I$\+adic completion of the ring~$R$.
 Let $\sB$ be the locally $\aleph_1$\+presentable abelian
category $\R\contra\simeq R\modl_{I\ctra}$ (see
Examples~\ref{comm-ring-ideal-contramodules-examples}) and
$P=\R=\Delta_I(R)$ be its natural projective generator.
 Then the forgetful/natural embedding functor $\Theta\:\sB\rarrow
R\modl$ is corepresented by the projective generator $P\in\sB$ with
the natural action of the ring $R$ in it (provided by the natural
ring homomorphism $\theta\:R\rarrow\R=\Hom_\sB(P,P)^\rop$).

 According to~\cite[Theorem~B.8.1]{Pweak}, the functor $\Theta$
induces isomorphisms of the groups $\Ext^i$ for all $0\le i<\infty$.
 Moreover, one can show using~\cite[Proposiitons~B.9.1
and~B.10.1]{Pweak} that the triangulated functor between the bounded
above derived categories $\sD^-(\sB)\rarrow\sD^-(R\modl)$ induced
by exact functor $\Theta$ is fully faithful.

 In particular, $\R$ is also a Noetherian commutative ring; so
one can replace $R$ with $\R$ and consider the case $R=\R$.
 Then we have $\Theta=\Theta_P$, so we can conclude that $P=\R$ is
a good projective generator of~$\sB$.

\smallskip

 (2)~More generally, let $R$ be a commutative ring and $I\subset R$
be a weakly proregular finitely generated ideal.
 We keep the notation $\R=\varprojlim_n R/I^n$, \ $\sB=\R\contra
=R\modl_{I\ctra}$, and $P=\R=\Delta_I(R)$
(see Example~\ref{comm-ring-ideal-contramodules-examples}(3)).
 As in~(1), \,$\sB$ is a locally $\aleph_1$\+presentable abelian
category, $P$ is an $\aleph_1$\+presentable projective generator of
$\sB$, and the forgetful/natural embedding functor $\Theta\:\sB
\rarrow R\modl$ is corepresented by~$P$.

 According to~\cite[Theorem~2.9]{Pmgm}, the triangulated functor
between the unbounded derived categories $\sD(\sB)\rarrow\sD(R\modl)$
induced by $\Theta$ is fully faithful.
 Hence the functor $\Theta$ induces isomorphisms of the groups $\Ext^i$
for all $0\le i<\infty$.
 
 By~\cite[Lemma~5.3(b) or~5.4(b)]{Pmgm}, the ideal $\R I=
\varprojlim_nI/I^n$ in the ring $\R$ is also weakly proregular;
so one can replace $R$ with $\R$ and consider the case $R=\R$.
 Then $\Theta=\Theta_P$, so we can conclude that $P=\R$ is a good
projective generator of~$\sB$.

\smallskip

 (3)~Even more generally, let $R$ be a commutative ring and $I\subset R$
be a finitely generated ideal.
 Set $\sB=R\modl_{I\ctra}$ and $P=\Delta_I(R)$ (see
Example~\ref{comm-ring-ideal-contramodules-examples}(1)).
 As in~(2), \,$\sB$ is a locally $\aleph_1$\+presentable abelian
category, $P$ is an $\aleph_1$\+presentable projective generator of
$\sB$, and the natural embedding functor $\Theta\:\sB\rarrow R\modl$
is corepresented by the projective generator $P\in\sB$ with
the natural action of the ring $R$ in it (provided by the natural
ring homomorphism $\theta\:R\rarrow\Delta_I(R)=\Hom_\sB(P,P)^\rop$).

 Following the argument in~\cite[proof of Theorem~2.9]{Pmgm} based
on~\cite[Lemma~2.7]{Pmgm}, the triangulated functor
$\sD(\sB)\rarrow\sD(R\modl)$ induced by $\Theta$ is fully faithful 
provided that $H_i(\Hom_R(T^\bu(R;s_1,\dotsc,s_m),R[X]))=0$
for all sets $X$ and all $i>0$.
 Conversely, arguing as in~\cite[second paragraph of Remark~6.8]{PMat}
(with torsion modules replaced by contramodules), one can show that
$H_i(\Hom_R(T^\bu(R;s_1,\dotsc,s_m),R[X]))=0$ for all $i>0$ whenever
the functor $\sD^\b(\sB)\rarrow\sD^\b(R\modl)$ is fully faithful.
 Thus, given any derived category symbol $\star=\b$, $+$, $-$,
or~$\varnothing$, the triangulated functor $\sD^\star(\sB)\rarrow
\sD^\star(R\modl)$ induced by $\Theta$ is fully faithful if and only if
for all sets $X$ one has $H_i(\Hom_R(T^\bu(R;s_1,\dotsc,s_m),R[X]))=0$
for $i>0$, or equivalently, if and only if
$\varprojlim_n H_i(\Hom_R(T^\bu_n(R;s_1,\dotsc,s_m),R[X]))=0$ for
$i\ge1$ and $\varprojlim_n^1H_i(\Hom_R(T^\bu_n(R;s_1,\dotsc,s_m),R[X]))
=0$ for $i\ge2$.

 Arguing as in~\cite[Proposition~3.1]{PMat}, one can show that
$\Delta_I(R)$ is a commutative ring.
 All the $I$\+contramodule $R$\+modules admit a unique natural
extension of their $R$\+module structure to a $\Delta_I(R)$\+module
structure.  
 So one can replace $R$ with $\Delta_I(R)$ and $I$ with $\Delta_I(R)I$,
and consider the case $R=\Delta_I(R)$.
 Then $\Theta=\Theta_P$.
 So $P$ is always a $1$\+good projective generator of $\sB$, since
$\sB=R\modl_{I\ctra}$ is closed under extensions in $R\modl$.
 As we have seen, $P$ is a good projective generator of $\sB$ if
and only if $H_i(\Hom_R(T^\bu(R;s_1,\dotsc,s_m),\.\Delta_I(R)[X]))=0$
for all sets $X$ and all $i>0$.

\smallskip

 (4)~Let $E_1\larrow E_2\larrow E_3\larrow\dotsb$ be a projective
system of abelian groups.
 Clearly, one has $\varprojlim_n E_n[X]=0$ for all sets $X$ if and
only if $\varprojlim_n E_n=0$.
 Furthermore, one has $\varprojlim_n^1 E_n[X]=0$ for all sets $X$
if and only if the projective system $E_n$ satisfies
the Mittag-Leffler condition~\cite[Corollary~6]{Emm}.
 Finally, one easily checks that a projective system $(E_n)$ satisfies
both the condition of the vanishing of the projective limit
$\varprojlim_n E_n=0$ and the Mittag-Leffler condition if and
only if it is pro-zero.

 Thus, in the context of~(3), the triangulated functor
$\sD^\star(R\modl_{I\ctra})\rarrow\sD^\star(R\modl)$ is fully faithful
if and only if the projective system
$H_i(\Hom_R(T^\bu_n(R;\allowbreak s_1,\dotsc,s_m),R))$ is pro-zero for
$i\ge2$ and $\varprojlim_n H_1(\Hom_R(T^\bu_n(R;s_1,\dotsc,s_m),R))=0$.

\smallskip

 (5)~In particular, let $R$ be the commutative algebra over a field~$k$
generated by the elements $s$, $t$, and $x_i$, \ $i\ge1$, with
the relations $x_ix_j=0$ and $s^it^ix_i=0$ for all $i$, $j\ge1$
(cf.~\cite[Example~2.6]{Pmgm}).
 Let $I=(s,t)\subset R$ be the ideal generated by the elements $s$
and~$t$.
 Then $H_1(\Hom_R(T^\bu_n(R;s,t),R))=0$ for all $n\ge1$ and
$\varprojlim_n H_2(\Hom_R(T^\bu_n(R;s,t),R[X]))=0$ for all sets $X$,
while $\varprojlim_n^1 H_2(\Hom_R(T^\bu_n(R;s,t),R))\ne0$.
 So $\Delta_I(R[X])=\Lambda_I(R[X])=\R[[X]]$ for all~$X$.

 Set $\sB=R\modl_{I\ctra}=\R\contra$ and $P=\Delta_I(R)=\R$.
 One has $\bigcap_n I^n=0$ in $R$, hence the morphism
$\theta\:R\rarrow\R$ is injective.
 Thus the category $\sB$ with its projective generator $P$ and
the morphism~$\theta$ provide a counterexample to the converse
assertions to
Theorem~\ref{ext-isomorphism-right-perpendicular-thm}(c\+d)
promised in Remark~\ref{good-perp-counterex-remark}.
 Denoting by $C$ the $I$\+contramodule $R$\+module $H_1
(\Hom_R(T^\bu(R;s,t),R))=\varprojlim_n^1 H_2(\Hom_R(T^\bu_n(R;s,t),R))$,
the complex of $R$\+modules $\Hom_R(T^\bu(R;s,t),R)$ represents
a nontrivial extension class in $\Ext^2_R(P,C)$.
{\emergencystretch=1em\par}

 Furthermore, one has $H_1(\Hom_R(T^\bu_n(R;s,t),\R))=0$ for all $n\ge1$
and $\varprojlim_n\allowbreak H_2(\Hom_R(T^\bu_n(R;s,t),\R[X]))=0$ for
all sets $X$, while $\varprojlim_n^1 H_2(\Hom_R(T^\bu_n(R;s,t),
\allowbreak\R[Z]))\ne0$ for a countable set~$Z$.
 Thus the category $\sB$ with its projective generator $P$ (and
the identity isomorphism $\theta=\id_\R$) provide a counterexample to
the converse assertions to
Corollary~\ref{good-implies-n-perpendicular}(c\+d).
 Setting $Q=\Delta_I(R[Z])=\Delta_{\R I}(\R[Z])=\R[[Z]]$ and denoting by
$C$ the $(\R I)$\+contramodule $\R$\+module $H_1(\Hom_R(T^\bu(R;s,t),
\R[Z]))=\varprojlim_n^1 H_2(\Hom_R(T^\bu_n(R;s,t),\R[Z]))$, the complex
of $\R$\+modules $\Hom_R(T^\bu(R;s,t),\R[Z])$ represents a nontrivial
extension class in $\Ext^2_\R(Q,C)$.  \emergencystretch=1em\hfuzz=1.7pt

\smallskip

 (6)~Let $R$ be a commutative ring and $I\subset R$ be a finitely
generated ideal.
 Set $\R=\varprojlim_n R/I^n$, \ $\sB=\R\contra$, and $P=\R$ (see
Examples~\ref{accessible-additive-monads-examples}(2),
\ref{comm-ring-ideal-contramodules-examples}(2\+3),
and~\ref{central-adic-topological-ring}(2\+4)).
 Then the abelian category $\sB$ is an exactly embedded full subcategory
of the abelian category $R\modl_{I\ctra}$.
 We claim that every object of $R\modl_{I\ctra}$ is an extension of two
objects from $\R\contra$.

 Indeed, every object of $R\modl_{I\ctra}$ has the form $\Delta_I(C)$
for some $C\in R\modl$.
 According to~\cite[Lemma~7.5]{Pcta}, there is a natural short exact
sequence of $R$\+modules $0\rarrow K\rarrow\Delta_I(C)\rarrow
\Lambda_I(C)\rarrow0$ with an $R$\+module $K$ of the form
$K=\varprojlim_n^1 D_n$, where $D_n$ is a sequence of $R$\+modules,
each of which is annihilated by some power of the ideal~$I$.
 Now we have $D_n\in\R\contra$, and it follows that $K\in\R\contra$,
as the full subcategory $\R\contra\subset R\modl$ is closed under
the cokernels and infinite products.
 Similarly, we have $\Lambda_I(C)\in\R\contra$.

 So the full subcategory $\R\contra\subset R\modl$ is closed under
extensions if and only if it coincides with $R\modl_{I\ctra}$.
 According to~(4) and
Example~\ref{comm-ring-ideal-contramodules-examples}(3)), this holds
if and only if the projective system
$H_1(\Hom_R(T^\bu_n(R;s_1,\dotsc,s_m),R))$ satisfies
the Mittag-Leffler condition (where $s_1$,~\dots, $s_m$ is some
set of generators of the ideal~$I$).
 Furthermore, it follows that, for any derived category symbol
$\star=\b$, $+$, $-$, or~$\varnothing$, the triangulated functor
$\sD^\star(\R\contra)\rarrow\sD^\star(R\modl)$ is fully faithful if and
only if the projective system $H_i(\Hom_R(T^\bu_n(R;s_1,\dotsc,s_m),R))$
is pro-zero for all $i\ge1$, or in other words, the ideal $I\subset R$
is weakly proregular.

 It was mentioned in Example~\ref{topological-ring-0-good}(1) that
the projective generator $P\in\R\contra$ is $0$\+good.
 Now we see that $P$ is $1$\+good if and only if the full subcategory
$\R\contra\subset\R\modl$ coincides with $\R\modl_{(\R I)\ctra}$,
and this holds if and only if the projective system
$H_1(\Hom_R(T^\bu_n(R;s_1,\dotsc,s_m),\R))=0$ satisfies
the Mittag-Leffler condition.
 The projective generator $P$ is good if and only if the ideal
$\R I\subset\R$ is weakly proregular.

\smallskip

 (7)~In particular, let $I=(s)\subset R$ be a principal ideal.
 Then $H_1(\Hom_R(T_n^\bu(R;s),R))$ is the submodule of elements
annihilated by~$s^n$ in~$R$, and the maps in the projective system
formed by these modules are the multiplications with (the powers
of)~$s$.
 Hence $\varprojlim_n H_1(\Hom_R(T_n^\bu(R;s),R))=0$ if and only if
there is no $s$\+divisible $s$\+torsion in $R$, and the projective
system $H_1(\Hom_R(T_n^\bu(R;s),R))$ is pro-zero if and only if
the $s$\+torsion in $R$ is bounded.

 Furthermore, the rings $\Delta_s(R)$ and $\R=\Lambda_s(R)$ are
$s$\+contramodules, so they cannot contain $s$\+divisible $s$\+torsion.
 Thus $\R\contra=R\modl_{s\ctra}$ if and only if the $s$\+torsion
in $\Delta_s(R)$ is bounded; and $\R\contra=\R\modl_{s\ctra}$ if
and only if the $s$\+torsion in $\R$ is bounded.
 According to~\cite[Remark~6.9]{Pcta}, the latter two conditions
are equivalent.
 The projective generator $P=\R\in\R\contra$ is $1$\+good if and only
if it good, and if and only if these equivalent conditions hold.

\smallskip

 (8)~In particular, if $R$ is the commutative $k$\+algebra
from Example~\ref{accessible-additive-monads-examples}(6)
and~\cite[Example~2.6]{Pmgm}, and $I$ is the principal ideal $I=(s)$,
then the projective generator $P\in\R\contra$ is not $1$\+good.
 Indeed, the $s$\+torsion in $\R$ is not bounded, so the morphism
$\Delta_I(\R[Z])\rarrow\Lambda_I(\R[Z])$ has a nontrivial kernel
for a countable set~$Z$.
\end{exs}

\begin{exs} \label{central-ideal-good-projective-generator}
 (1)~Let $R$ be a right Noetherian associative ring and $I\subset R$ be
an ideal generated by a finite set of central elements.
 Denote by $\R=\varprojlim_n R/I^n$ the $I$\+adic completion of
the ring~$R$.
 Let $\sB$ be the locally $\aleph_1$\+presentable abelian category
$\R\contra\simeq R\modl_{I\ctra}$ (see
Examples~\ref{central-ideal-contramodules-examples}) and
$P=\R=\Delta_I(R)$ be its natural projective generator.
 Then the forgetful/natural embedding functor $\Theta\:\sB\rarrow
R\modl$ is corepresented by the projective generator $P\in\sB$ with
the natural right action of the ring $R$ in it (provided by
the natural ring homomorphism $\theta\:R\rarrow\R=\Hom_\sB(P,P)^\rop$).

 Accoding to~\cite[Corollary~C.5.6(b)]{Pcosh}, the functor $\Theta$
induces isomorphisms of the groups $\Ext^i$ for all $0\le i<\infty$.
 Moreover, one can show using~\cite[Proposition~C.5.5 and
Corollary~C.5.6(a)]{Pcosh} that the triangulated functor
$\sD^-(\sB)\rarrow\sD^-(R\modl)$ induced by the exact functor
$\Theta$ is fully faithful.

 In particular, $\R$ is also a right Noetherian ring, so one can
replace $R$ with $\R$ and consider the case $R=\R$.
 Then we have $\Theta=\Theta_P$, so we can conclude that $P=\R$
is a good projective generator of~$\sB$.

\smallskip

 (2)~More generally, let $R$ be an associative ring and $I\subset R$
be the ideal generated by a weakly proregular finite sequence of
central elements $s_1$,~\dots, $s_m\in R$, i.~e., a finite sequence of
central elements such that the projective system of $R$\+$R$\+bimodules
$H_i(\Hom_R(T_n(R;s_1,\dotsc,s_m),R))$, \ $n\ge1$, is pro-zero
for all $i>0$ (see~\cite{Sch,PSY,Pmgm}; cf.\
Example~\ref{central-ideal-contramodules-examples}(3)).
 We keep the notation $\R=\varprojlim_n R/I^n$, \ $\sB=\R\contra
= R\modl_{I\ctra}$, and $P=\R=\Delta_I(R)$.
 As above, \,$\sB$ is a locally $\aleph_1$\+presentable abelian
category, $P$ is an $\aleph_1$\+presentable projective generator
of $\sB$, and the forgetful/natural embedding functor $\Theta\:
\sB\rarrow R\modl$ is corepresented by~$P$.

 Arguing as in~\cite[Section~2]{Pmgm} (see~\cite[Theorem~6.4]{PMat}
for additional details), one can show that the triangulated functor
$\sD(\sB)\rarrow\sD(R\modl)$ induced by $\Theta$ is fully faithful.
 Hence the functor $\Theta$ induces isomorphisms of the groups
$\Ext^i$ for all $0\le i<\infty$.

 Using the appropriate generalization of~\cite[Lemma~5.3(b)]{Pmgm},
one can show that the sequence of central elements~$s_1$,~\dots,
$s_m$ is also weakly proregular in the ring~$\R$.
 So one can replace $R$ with $\R$ and consider the case $R=\R$.
 Then $\Theta=\Theta_P$, and one can conclude that $P=\R$ is
a good projective generator of~$\sB$.

\smallskip

 (3)~Even more generally, let $R$ be an associative ring and
$I\subset R$ be the ideal generated by a finite set of central
elements $s_1$,~\dots, $s_m\in R$.
 Set $\sB=R\modl_{I\ctra}$ and $P=\Delta_I(R)$
(see Example~\ref{central-ideal-contramodules-examples}(1)).
 As above, $\sB$ is a locally $\aleph_1$\+presentable abelian
category, $P$ is an $\aleph_1$\+presentable projective generator of
$\sB$, and the natural embedding functor $\Theta\:\sB\rarrow R\modl$
is corepresented by the projective generator $P\in\sB$ with
the natural right action of the ring $R$ in it (provided by the natural
ring homomorphism $\theta\:R\rarrow\Delta_I(R)=\Hom_\sB(P,P)^\rop$).
 
 Arguing as in
Example~\ref{comm-ring-ideal-good-projective-generator}(3),
one shows that the triangulated functor $\sD^\star(\sB)\rarrow
\sD^\star(R\modl)$ induced by $\Theta$, where $\star=\b$, $+$, $-$,
or~$\varnothing$, is fully faithful if and only if
$H_i(\Hom_R(T^\bu(R;s_1,\dotsc,s_m),R[X]))=0$ for all sets $X$ and
all $i>0$, or equivalently, if and only if $\varprojlim_n
H_i(\Hom_R(T^\bu_n(R;s_1,\dotsc,s_m),R[X]))=0$ for $i\ge1$ and
$\varprojlim_n^1 H_i(\Hom_R(T^\bu_n(R;s_1,\dotsc,s_m),R[X]))=0$
for $i\ge2$.
 As explained in
Example~\ref{comm-ring-ideal-good-projective-generator}(4),
this holds if and only if the projective system
$H_i(\Hom_R(T^\bu_n(R;s_1,\dotsc,s_m),\allowbreak R))$ is pro-zero for
$i\ge2$ and $\varprojlim_n H_1(\Hom_R(T^\bu_n(R;s_1,\dotsc,s_m),R))=0$.

 Furthermore, all the $I$\+contramodule left $R$\+modules admit
a unique natural extension of their left $R$\+module structure to
a left $\Delta_I(R$)\+module structure.
 Central elements in $R$ remain central in $\Delta_I(R)$.
 So one can replace $R$ with $\Delta_I(R$) and $I$ with
$\Delta_I(R)I$, and consider the case $R=\Delta_I(R)$; then
$\Theta=\Theta_P$.
 Hence $P$ is always a $1$\+good projective generator of $\sB$, as
$\sB=\R\modl_{I\ctra}$ is closed under extensions in $R\modl$.
 As we have seen, $P$ is a good projective generator of $\sB$ if
and only if $H_i(\Hom_R(T^\bu(R;s_1,\dotsc,s_m),\.\Delta_I(R)[X]))=0$
for all sets $X$ and all $i>0$, and if and only if the projective system
$H_i(\Hom_R(T^\bu_n(R;s_1,\dotsc,s_m),\Delta_I(R)))$ is pro-zero for
$i\ge2$ and $\varprojlim_n H_1(\Hom_R(T^\bu_n(R;s_1,\dotsc,s_m),
\Delta_I(R))=0$.

\smallskip

 (4)~Let $R$ be an associative ring and $I\subset R$ be an ideal
generated by a finite set of central elements.
 Set $\R=\varprojlim_n R/I^n$, \ $\sB=\R\contra$, and $P=\R$ (see
Examples~\ref{accessible-additive-monads-examples}(2),
\ref{central-ideal-contramodules-examples}(3),
and~\ref{central-adic-topological-ring}(2\+4)).

 Then the abelian category $\sB$ is an exactly embedded full subcategory
of the abelian category $\R\modl_{I\ctra}$.
 In the same way as in
Example~\ref{comm-ring-ideal-good-projective-generator}(6), one shows
that every object of $R\modl_{I\ctra}$ is an extension of two objects
from $\R\contra$.
 Hence the full subcategory $\R\contra\subset R\modl$ is closed under
extensions if and only if it coincides with $R\modl_{I\ctra}$, which
holds if and only if the projective system
$H_1(\Hom_R(T^\bu_n(R;s_1,\dotsc,s_m),R))$ satisfies the Mittag-Leffler
condition.
 Furthermore, for any derived category symbol $\star=\b$, $+$, $-$,
or~$\varnothing$, the functor $\sD^\star(\R\contra)\rarrow\sD^\star
(R\modl)$ is fully faithful if and only if the projective system
$H_i(\Hom_R(T^\bu_n(R;s_1,\dotsc,s_m),R))$ is pro-zero for all $i\ge1$.

 According to Example~\ref{topological-ring-0-good}(1), the projective
generator $P\in\R\contra$ is $0$\+good.
 Now we see that $P$ is $1$\+good if and only if the full subcategory
$\R\contra\subset\R\modl$ coincides with $\R\modl_{(\R I)\ctra}$, and
$P$ is good if and only if the projective system
$H_i(\Hom_R(T^\bu_n(R;s_1,\dotsc,s_m),\R))$ is pro-zero for all $i\ge1$.

 In particular, when $I=(s)$ is a principal ideal, one shows,
arguing as in
Example~\ref{comm-ring-ideal-good-projective-generator}(7),
that $\R\contra=R\modl_{s\ctra}$ if and only if the $s$\+torsion in
$\Delta_s(R)$ is bounded, if and only if the $s$\+torsion in $\R$
is bounded, and if and only if $\R\contra=\R\modl_{s\ctra}$.
 These are the necessary and sufficient conditions for the projective
generator $P\in\R\contra$ to be $1$\+good or, which is equivalent in
this case, good.
\end{exs}

\begin{exs} \label{comm-ring-mult-subset-good-projective-generator}
 (1)~Let $R$ be a commutative ring and $S\subset R$ be a multiplicative
subset such that $\pd_RS^{-1}R\le 1$.
 Set $\sB=R\modl_{S\ctra}$ and $P=\Delta_S(R)$
(see Example~\ref{comm-ring-mult-subset-contramodules-examples}(1)).
 Then $\sB$ is a locally presentable abelian category, $P\in\sB$
is a projective generator, and the natural embedding functor
$\Theta\:\sB\rarrow R\modl$ is corepresented by the projective
generator $P$ with the natural action of the ring $R$ in it
(provided by the natural ring homomorphism $\theta\:R\rarrow
\Delta_S(R)=\Hom_\sB(P,P)^\rop$).

 According to~\cite[first paragraph of Remark~6.8]{PMat},
the triangulated functor $\sD(\sB)\rarrow\sD(R\modl)$ induced
by $\Theta$ is fully faithful provided that there is no
$S$\+h-divisible $S$\+torsion in~$R$.
 Conversely, arguing as in~\cite[second paragraph of Remark~6.8]{PMat}
(with torsion modules replaced by contramodules), one shows that
there is no $S$\+h-divisible $S$\+torsion in $R$ whenever
the functor $\sD^\b(\sB)\rarrow\sD^\b(R\modl)$ is fully faithful.
 Thus, for any derived category symbol $\star=\b$, $+$, $-$,
or~$\varnothing$, the triangulated functor $\sD^\star(\sB)\rarrow
\sD^\star(R\modl)$ induced by $\Theta$ is fully faithful if and only if
there is no $S$\+h-divisible $S$\+torsion in the $R$\+module~$R$.

 According to~\cite[Proposition~3.1]{PMat}, \,$\Delta_S(R)$ is
a commutative ring.
 Since $S^{-1}R$ is a flat $R$\+module, the condition that
$\pd_RS^{-1}R\le\nobreak1$ implies that $\pd_{\Delta_S(R)}
S^{-1}\Delta_S(R)\le\nobreak1$.
 All the $S$\+contramodule $R$\+modules admit a unique natural
extension of their $R$\+mod\-ule structure to a $\Delta_S(R)$\+module
structure.
 So one can replace $R$ with $\Delta_S(R)$ and $S$ with its image
in $\Delta_S(R)$, and consider the case $R=\Delta_S(R)$, so that
$\Theta=\Theta_P$.
 Then the $R$\+module $R$ is $S$\+h-reduced, and therefore has no
$S$\+h-divisible $S$\+torsion.
 Thus $P$ is always a good projective generator of~$\sB$.

\smallskip

 (2)~Let $R$ be a commutative ring and $S\subset R$ be a countable
multiplicative subset.
 Set $\R=\varprojlim_{s\in S} R/sR$, \ $\sB=\R\contra$, and $P=\R$
(see Examples~\ref{comm-ring-mult-subset-contramodules-examples}(2)
and~\ref{central-mult-subset-topological-ring}(2)).
 Then the abelian category $\sB$ is an exactly embedded full
subcategory of the abelian category $R\modl_{S\ctra}$.
 Both of these are locally $\aleph_1$\+presentable abelian categories.

 We claim that every object of $R\modl_{S\ctra}$ is an extension of
two objects from $\R\contra$ (cf.\
Example~\ref{comm-ring-ideal-good-projective-generator}(6)).
 Indeed, every object of $R\modl_{S\ctra}$ has the form $\Delta_S(C)$
for some $C\in R\modl$.
 The functor $\Delta_S$ can be computed as $\Delta_S(C)=
\Ext^1_R(K^\bu,C)$ (see
Example~\ref{comm-ring-mult-subset-contramodules-examples}(1)).

 Let $s_1$, $s_2$, $s_3$,~\dots\ be a sequence of elements of $S$
containing every element of $S$ infinitely many times.
 Set $t_0=1$ and $t_n=s_1s_2\dotsm s_n$ for all $n\ge 1$.
 Then the two-term complex $K^\bu$ is the inductive limit of
the sequence of two-term complexes $R\overset{t_n}\rarrow R$
(with the morphism of complexes $(R\overset{t_{n-1}}\rarrow R)\rarrow
(R\overset{t_n}\rarrow R)$ acting by~$1$ on the leftmost terms of
the complexes and by~$s_n$ on the rightmost terms).
 Using this description of $K^\bu$, one can compute that for any
$R$\+module $C$ there is a short exact sequence of $R$\+modules
$0\rarrow\varprojlim_n^1{}_{t_n}C\rarrow\Delta_S(C)\rarrow
\varprojlim_n C/t_nC=\Lambda_S(C)\rarrow0$, where ${}_tM\subset M$
denotes the submodule of all elements annihilated by an element
$t\in R$ in an $R$\+module~$M$, and the maps ${}_{t_{n+1}}C\rarrow
{}_{t_n}C$ are the multiplications with~$s_n$
(see~\cite[Lemma~3.2]{PSl}).
 Since any $R$\+module annihilated by an element from $S$ is
an $\R$\+contramodule, and the full subcategory $\R\contra\subset
R\modl$ is closed under the kernels, cokernels, and infinite products,
it follows that both $\Lambda_S(C)$ and the kernel of the surjective
morphism $\Delta_S(C)\rarrow\Lambda_S(C)$ belong to $\R\contra$.

 So the full subcategory $\R\contra\subset R\modl$ is closed under
extensions if and only if it coincides with $R\modl_{S\ctra}$.
 According to 
Example~\ref{comm-ring-mult-subset-contramodules-examples}(3)
and~\cite[Corollary~6]{Emm}
(cf.\ Example~\ref{comm-ring-ideal-good-projective-generator}(4)),
this holds if and only if the projective system $({}_{t_n}R)_{n\ge0}$
satisfies the Mittag-Leffler condition.
 Furthermore, can easily see that $\varprojlim_n{}_{t_n}R=0$ if and only
if there is no $S$\+h-divisible $S$\+torsion in $R$; and the projective
system $({}_{t_n}R)_{n\ge0}$ is pro-zero if and only if the $S$\+torsion
in $R$ is bounded (which means that there exists $t\in S$ such that
$sr=0$, \ $s\in S$, $r\in R$, implies $tr=0$).
 Comparing these observations with~(1), we conclude that the functor
$\sD^\star(\R\contra)\rarrow\sD^\star(R\modl)$ is fully faithful if
and only if the $S$\+torsion in $R$ is bounded.

 Arguing as in
Example~\ref{comm-ring-ideal-good-projective-generator}(7) and using
an appropriate version of~\cite[Remark~6.9]{Pcta}, one can show
that $\R\contra=R\modl_{S\ctra}$ if and only if the $S$\+torsion in
$\Delta_S(R)$ is bounded, if and only if the $S$\+torsion in $\R$ is
bounded, and if and only if $\R\contra=\R\modl_{S\ctra}$.
 According to Example~\ref{topological-ring-0-good}(1), the projective
generator $P\in\R\contra$ is $0$\+good.
 Now we see that $P\in\R\contra$ is $1$\+good if and only if
it is good, and if and only if these equivalent conditions hold.
\end{exs}

\begin{exs} \label{central-mult-subset-good-projective-generator}
 (1)~Let $R$ be an associative ring and $S\subset R$ be a multiplicative
subset of central elements in $R$ such that $\pd_RS^{-1}R\le1$.
 Set $\sB=R\modl_{S\ctra}$ and $P=\Delta_S(R)$
(see Example~\ref{central-mult-subset-contramodules-examples}(1)).
 Then $\sB$ is a locally presentable abelian category, $P\in\sB$ is
a projective generator, and the natural embedding functor
$\Theta\:\sB\rarrow R\modl$ is corepresented by the projective
generator $P$ with the natural right action of the ring $R$ in it
(provided by the natural ring homomorphism $\theta\:R\rarrow
\Delta_S(R)=\Hom_\sB(P,P)^\rop$).

 Arguing as in
Example~\ref{comm-ring-mult-subset-good-projective-generator}(1), one
shows that, for any derived category symbol $\star=\b$, $+$, $-$,
or~$\varnothing$, the triangulated functor $\sD^\star(\sB)\rarrow
\sD^\star(R\modl)$ induced by $\Theta$ is fully faithful if and only if
there are is no $S$\+h-divisible $S$\+torsion in the left
$R$\+module~$R$.

 Furthermore, central elements in $R$ remain central in $\Delta_S(R)$.
 Since $S^{-1}R$ is a flat left $R$\+module, the condition that
$\pd_RS^{-1}R\le\nobreak1$ implies that $\pd_{\Delta_S(R)}
S^{-1}\Delta_S(R)\le\nobreak1$.
 All the $S$\+contramodule left $R$\+modules admit a unique natural
extension of their left $R$\+module structure to a left
$\Delta_S(R)$\+module structure.
 So one can replace $R$ with $\Delta_S(R)$ and $S$ with its image in
$\Delta_S(R)$, and consider the case $R=\Delta_S(R)$, so that
$\Theta=\Theta_P$.
 Then the left $R$\+module $R$ is $S$\+h-reduced (i.~e., there are
no nonzero morphisms into it from left $S^{-1}R$\+modules), and
therefore has no $S$\+h-divisible $S$\+torsion.
 Thus $P$ is always a good projective generator of~$\sB$.

\smallskip

 (2)~Let $R$ be an associative ring and $S\subset R$ be a countable
multiplicative subset of central elements.
 Set $\R=\varprojlim_{s\in S}R/sR$, \ $\sB=\R\contra$, and $P=\R$
(see Examples~\ref{central-mult-subset-contramodules-examples}(2\+3)
and~\ref{central-mult-subset-topological-ring}(2)).
 The abelian category $\sB$ is an exactly embedded full subcategory
of the abelian category $R\modl_{S\ctra}$.
 Both $\sB$ and $R\modl_{S\ctra}$ are locally $\aleph_1$\+presentable
categories.
 Arguing as in
Example~\ref{comm-ring-mult-subset-good-projective-generator}(2),
one shows that every object of $R\modl_{S\ctra}$ is an extension of
two objects from $\R\contra$.

 Thus the full subcategory $\R\contra\subset R\modl$ is closed
under extensions if and only if it coincides with $R\modl_{S\ctra}$.
 The functor $\sD^\star(\R\contra)\rarrow\sD^\star(R\modl)$ is fully
faithful if and only if the $S$\+torsion in $R$ is bounded.
 Furthermore, one has $\R\contra=R\modl_{S\ctra}$ if and only if
the $S$\+torsion in $\Delta_S(R)$ is bounded, if and only if
the $S$\+torsion in $\R$ is bounded, and if and only if
$\R\contra=\R\modl_{S\ctra}$.
 The projective generator $P\in\R\contra$ is always $0$\+good.
 The previous several equivalent condition are necessary and
sufficient for $P$ to be $1$\+good or, which is equivalent in
this situation, good.
\end{exs}

 The next series of examples explains the reason for the ``good
projective generator'' terminology.

\begin{exs}
 (1)~Let $\sC$ be an accessible additive category (notice that any
such category contains the images of idempotent endomorphisms of
its objects~\cite[Observation~2.4]{AR}.
 Let $M\in\sC$ be an object.
 According to Example~\ref{additive-monads-examples}(2), there exists
a unique abelian category $\sB$ with enough projective objects for
which there is an equivalence of additive categories $\sB_\proj\simeq
\Add_\sC(M)$.
 The category $\sB$ comes endowed with a natural projective
generator $P\in\sB_\proj$ corresponding to the free $\boT_M$\+module
with one generator $\boT_M(*)\in\boT_M\modl$ and to the object
$M\in\Add_\sC(M)$.
 If the object $M\in\sC$ is $\kappa$\+presentable, where $\kappa$~is
a regular cardinal, then the abelian category $\sB$ is locally
$\kappa$\+presentable and the object $P\in\sB$ is $\kappa$\+presentable.

 We will say that an object $M\in\sC$ is \emph{$n$\+good}, where
$n\ge0$ is an integer, if the projective generator $P\in\sB$ is
$n$\+good.
 In other words, denoting by $\S$ the ring $\Hom_\sC(M,M)^\rop$,
the object $M\in\sC$ is $n$\+good if and only if the functor
$\Theta=\Hom_\sB(P,{-})\:\sB\rarrow\S\modl$ induces isomorphisms on all
the groups $\Ext^i$ with $0\le i\le n$.
 An object $M\in\sC$ is \emph{good} if it is $n$\+good for all $n\ge0$,
or in other words, if the projective generator $P\in\sB$ is good.
 According to Proposition~\ref{fully-faithful-b-implies-minus},
an object $M$ is good if and only if the triangulated functor
$\sD^-(\sB)\rarrow\sD^-(\S\modl)$ induced by $\Theta$ is fully
faithful.

 According to Theorem~\ref{copower-good-projective-generator},
if the object $M\in\sC$ is $\kappa$\+presentable (or, more generally,
$\kappa$\+generated), then for any cardinal~$\lambda$ such that
$\lambda^+\ge\kappa$ the object $M^{(\lambda)}\in\sC$ is good.

\smallskip

 (2)~In particular, let $R$ be an associative ring and $M$ be
a left $R$\+module.
 Substituting $\sC=R\modl$ into the definitions in~(1), we obtain
the definition of what it means for $M$ to be \emph{$n$\+good}
or \emph{good}.
 Specifically, it means that the forgetful functor $\S\contra
\rarrow\S\modl$, where $\S=\Hom_R(M,M)^\rop$ is the topological
ring from Example~\ref{accessible-additive-monads-examples}(3),
should induce isomorphisms on all the groups $\Ext^i$ with
$0\le i\le n$, or with $i\ge0$, respectively.
 An $R$\+module $M$ is good if and only if the triangulated functor
$\sD^-(\S\contra)\rarrow\sD^-(\S\modl)$ is fully faithful.
 If a left $R$\+module $M$ admits a set of generators of
cardinality~$\lambda$, then the left $R$\+module $M^{(\lambda)}$
is good.

\smallskip

 (3) Any finitely generated left $R$\+module $M$ is good.
 Indeed, in this case the monad $\boT_M$ is isomorphic to the monad
$\boT_S$ associated with the discrete ring $S=\Hom_R(M,M)^\rop$
as in Example~\ref{accessible-additive-monads-examples}(1),
so the functor $\Theta=\Hom_\sB(P,{-})\:\sB\rarrow S\modl$
is an equivalence of abelian categories.

\smallskip

 (4)~Let $R$ be a commutative ring and $S\subset R$ be a multiplicative
subset such that all elements of $S$ are nonzero-divisors in~$R$.
 Assume that $\pd_RS^{-1}R\le 1$ and set $K=S^{-1}R/R\in R\modl$.
 Then $\sB=R\modl_{S\ctra}=\R\contra$, where $\R=\Delta_S(R)=
\Lambda_S(R)$ (see
Example~\ref{comm-ring-mult-subset-contramodules-examples}(3)),
is the corresponding locally presentable abelian category with
enough projective objects such that $\sB_\proj\simeq\Add_{R\modl}(K)$,
and $P=\R\in\sB$ is its natural projective generator.
 Indeed, the functors $Q\longmapsto K\ot_RQ$ and $N\longmapsto
\Hom_R(K,N)$ establish an equivalence between the additive categories
$\sB_\proj\ni Q$ and $\Add_{R\modl}(K)\ni N$ assigning the object
$P\in\sB_\proj$ to the object $K\in\Add_{R\modl}(K)$ and the object
$P^{(X)}=\Delta_S(R[X])=\Lambda_S(R[X])\in\sB_\proj$ to the object
$K^{(X)}\in\Add_{R\modl}(K)$ for all sets~$X$
(cf., e.~g., \cite[Corollary~5.2]{PMat}).
 According to
Example~\ref{comm-ring-mult-subset-good-projective-generator}(1),
the projective generator $P\in\sB$ is good, so the $R$\+module
$K=S^{-1}R/R$ is good.

\smallskip

 (5)~More generally, let $R$ be an associative ring and $S\subset R$
be a multiplicative subset consisting of some central nonzero-divisors
in~$R$.
 Assume that $\pd_RS^{-1}R\le\nobreak1$ and set $K=S^{-1}R/R\in R\modl$.
 Then $\sB=R\modl_{S\ctra}=\R\contra$, where $\R=\Delta_S(R)=
\Lambda_S(R)$ (see
Example~\ref{central-mult-subset-contramodules-examples}(3)),
is the corresponding locally presentable abelian category with
enough projective objects such that $\sB_\proj\simeq\Add_{R\modl}(K)$,
and $P=\R\in\sB$ is its natural projective generator.
 Indeed, just as in~(4), the functors $Q\longmapsto K\ot_RQ$ and
$N\longmapsto \Hom_R(K,N)$ (where $K$ is viewed as
an $R$\+$R$\+bimodule) establish an equivalence between
the additive categories $\sB_\proj$ and $\Add_{R\modl}(K)$ assigning
$P\in\sB_\proj$ to $K\in\Add_{R\modl}(K)$ and
$P^{(X)}=\Delta_S(R[X])=\Lambda_S(R[X])\in\sB_\proj$ to
$K^{(X)}\in\Add_{R\modl}(K)$ for all sets~$X$
(cf.~\cite[Theorem~4.7]{FN} or~\cite[Theorem~1.3]{BP2}).
 According to
Example~\ref{central-mult-subset-good-projective-generator}(1),
the projective generator $P\in\sB$ is good, so the left $R$\+module
$K=S^{-1}R/R$ is good.

\smallskip

 (6)~Let $T$ be a \emph{good $n$\+tilting} left $R$\+module in
the sense of~\cite{BMT} (where $n\ge0$ is an integer).
 Then, according to~\cite[Theorem~2.2(2)]{BMT}, the triangulated
functor $\boR\Hom_R(T,{-})\:\sD(R\modl)\rarrow\sD(\S\modl)$ is fully
faithful.

 On the other hand, by~\cite[Proposition~2.3]{FMS}
or~\cite[Theorem~4.5]{Ba} (cf.\ the discussion
in~\cite[Proposition~8.2]{PS}), the triangulated functor
$\boR\Hom_R(T,{-})\:\sD(R\modl)\rarrow\sD(\S\contra)$ is
an equivalence of triangulated categories.
 (Notice that, by~\cite[Proposition~4.3]{Ba} and our
Examples~\ref{additive-monads-examples}(2)
and~\ref{accessible-additive-monads-examples}(3), the heart
$\sB\subset\sD^\b(R\modl)$ of the tilting t\+structure associated
with $T$ is equivalent to the abelian category $\S\contra$
\,\cite[Proposition~2.6 and Theorem~7.1]{PS}.)

 Thus the triangulated functor $\sD(\S\contra)\rarrow\sD(\S\modl)$
induced by the forgetful functor $\S\contra\rarrow\S\modl$ is
fully faithful, so $T$ is a good left $R$\+module in the sense of our
definition.
\end{exs}

\Section{$\kappa$-Flat Modules and the Fully Faithful Triangulated
Functor} \label{lambda-flat-secn}

 The aim of this section is to prove
Theorem~\ref{copower-good-projective-generator}.
 For this purpose, we develop the theory of $\kappa$\+flat modules
over associative rings, generalizing the Govorov--Lazard
characterization of flat modules to arbitrary regular cardinals.

 Let $R$ be an associative ring and $\kappa$~be a regular cardinal.

\begin{thm} \label{kappa-flat-modules}
 The following conditions on a left $R$\+module $F$ are equivalent:
\par
\textup{(a)} every morphism into $F$ from a $\kappa$\+presentable
left $R$\+module factorizes through a projective left $R$\+module; \par
\textup{(b)} every morphism into $F$ from a $\kappa$\+presentable
left $R$\+module factorizes through a free left $R$\+module with
less than~$\kappa$ generators; \par
\textup{(c)} $F$ is the colimit of a $\kappa$\+filtered diagram of
projective left $R$\+modules; \par
\textup{(d)} the cocone formed by all the morphisms into $F$ from
projective left $R$\+modules with less than~$\kappa$ generators (and
all the morphisms between the latter forming commutative triangles
with the morphism into~$F$) is a $\kappa$\+filtered colimit cocone; \par
\textup{(e)} the cocone formed by all the morphisms into $F$ from
free left $R$\+modules with less than~$\kappa$ generators (and
all the morphisms between the latter forming commutative triangles
with the morphism into~$F$) is a $\kappa$\+filtered colimit cocone.
\end{thm}

\begin{proof}
 (a)~$\Longrightarrow$~(b): given a left $R$\+module morphism
$E\rarrow F$, where $E$ is $\kappa$\+presentable, one first
factorizes the morphism $E\rarrow F$ through a projective left
$R$\+module $P$, then replaces $P$ with a free left $R$\+module
$P'$ in which $P$ is a direct summand, and finally replaces
$P'$ with its free submodule $P''$ with less than~$\kappa$
generators containing the image of the morphism $E\rarrow P'$.

 (c)~$\Longrightarrow$~(a): let $F=\varinjlim_i P_i$ be a presentation
of $F$ as the colimit of a $\kappa$\+filtered diagram of projective
left $R$\+modules, and let $E$ be a $\kappa$\+presentable left
$R$\+module.
 Then $\Hom_R(E,\varinjlim_i P_i)=\varinjlim_i\Hom_R(E,P_i)$, hence
any left $R$\+module morphism $E\rarrow F$ factorizes through one
of the modules~$P_i$.

 (b)~$\Longrightarrow$~(d): denote our cocone by $(P_i\to F)_{i\in I}$
(where $I$ is the set of all pairs $i=(P_i,f_i)$, with $P_i$
a projective left $R$\+module with less than~$\kappa$ generators
and $f_i\:P_i\rarrow F$ a left $R$\+module morphism). 
 Since the free left $R$\+module with one generator $R$ is a projective
left $R$\+module with less than~$\kappa$ generators, and every element
of $F$ belongs to the image of some left $R$\+module morphism
$R\rarrow F$, the natural morphism $\varinjlim_{i\in I} P_i\rarrow F$ is
surjective.

 Furthermore, given two morphisms $f_i\:P_i\rarrow F$ and
$f_j\:P_j\rarrow F$ in the cocone, and two elements $p'\in P_i$
and $p''\in P_j$ such that $f_i(p')=f_j(p'')\in F$, there exist
left $R$\+module morphisms $R\rarrow P_i$ and $R\rarrow P_j$
taking the unit element $1\in R$ to $p'$ and $p''$, respectively.
 Setting $P_k=R$ and denoting by $f_k\:P_k\rarrow F$ the morphism
taking~$1$ to $f_i(p')=f_j(p'')$, we have morphisms $k\to i$
and $k\to j$ in the diagram~$I$.
 If $I$~is filtered, there exists a morphism $f_l\:P_l\rarrow F$ in
the cocone and morphisms $i\to l$, $j\to l$ in the diagram~$I$
such that the square diagram $k\to i\to l$, \ $k\to j\to l$
is commutative.
 It follows that images of the elements $p'\in P_i$ and $p''\in P_j$
coincide in $P_l$, and hence also in $\varinjlim_{i\in I}P_i$.
 Thus the morphism $\varinjlim_{i\in I}P_i\rarrow F$ is an isomorphism
whenever $I$~is filtered.

 It remains to show that the diagram~$I$ is $\kappa$\+filtered
whenever (b)~holds.
 Any collection of less than $\kappa$ morphisms $L_i\rarrow F$,
where $L_i$ are projective left $R$\+modules with less than~$\kappa$
generators, factorizes through the morphism
$M=\bigoplus_i L_i\rarrow F$, and $M$ is also a projective left
$R$\+module with less than~$\kappa$ generators.

 Now let $h_i\:L\rarrow M$ be a collection of less than~$\kappa$
morphisms between two projective left $R$\+modules $L$ and $M$
with less than~$\kappa$ generators each.
 Then the coequalizer $E$ of the whole system of morphisms
$h_i\:L\rarrow M$ is a $\kappa$\+presentable left $R$\+module.
 Given left $R$\+module morphisms $f\:L\rarrow F$ and $g\:M\rarrow F$
forming commutative triangles with all the morphisms~$h_i$,
we have the induced left $R$\+module morphism $E\rarrow F$.
 Assuming~(b), the latter morphism factorizes through a projective
left $R$\+module $P$ with less than~$\kappa$ generators.
 Hence we have a morphism $M\rarrow P$ in the diagram $I$ whose
compositions with all the morphisms $h_i\:L\rarrow M$ are equal
to one and the same morphism $L\rarrow P$.

 The proof of (b)~$\Longrightarrow$~(e) is similar to
(b)~$\Longrightarrow$~(d).
 The implications (d)~$\Longrightarrow$~(c), \ 
(e)~$\Longrightarrow$~(c), and (b)~$\Longrightarrow$~(a) are obvious.
\end{proof}

 A left $R$\+module $F$ is said to be \emph{$\kappa$\+flat} if it
satisfies one of the equivalent conditions of
Theorem~\ref{kappa-flat-modules}.

\begin{lem} \label{kappa-flat-closure-properties}
\textup{(a)} The class of all $\kappa$\+flat left $R$\+modules is
closed under extensions, kernels of surjective morphisms, and
$\kappa$\+filtered colimits. \par
\textup{(b)} Any short exact sequence of $\kappa$\+flat left
$S$\+modules is a $\kappa$\+filtered colimit of split short exact
sequences of $\kappa$\+flat left $S$\+modules.
\end{lem}

\begin{proof}
 If $G=\varinjlim_i F_i$ is the colimit of a $\kappa$\+filtered diagram
of modules and $E$ is a $\kappa$\+presentable module, then any morphism
$E\rarrow G$ factorizes through one of the modules~$F_i$.
 If, in addition, any morphism $E\rarrow F_i$ factorizes through
a projective module, then any morphism $E\rarrow G$ also factorizes
through a projective module.

 Let $0\rarrow F\rarrow G\rarrow H\rarrow0$ be a short exact sequence
of left $R$\+modules.
 Suppose that the left $R$\+modules $G$ and $H$ are $\kappa$\+flat.
 Let $E$ be a $\kappa$\+presentable left $R$\+module and
$E\rarrow F$ be a left $R$\+module morphism.
 Then the composiiton $E\rarrow F\rarrow G$ factorizes through
a free left $R$\+module $P$ with less than~$\kappa$ generators.
 Now the composition $P\rarrow G\rarrow H$ factorizes through
the cokernel $C=P/E$ of the morphism $E\rarrow P$.
 Since the left $R$\+module $C$ is also $\kappa$\+presentable,
the morphism $C\rarrow H$ factorizes through a projective left
$R$\+module~$Q$.

 Pick an arbitrary lifting $Q\rarrow G$ of the morphism $Q\rarrow H$
and subtract the composition $P\rarrow C\rarrow Q\rarrow G$ from
the morphism $P\rarrow G$ that we have.
 The resulting morphism $P\rarrow G$ is annihilated by the composition
with the morphism $G\rarrow H$, and therefore factorizes through
the monomorphism $F\rarrow G$.
 We have obtained a morphism $P\rarrow F$ whose composition with
the morphism $E\rarrow P$ is equal to our original morphism
$E\rarrow F$.
 Thus the left $R$\+module $F$ is $\kappa$\+flat.

 Now suppose that the left $R$\+modules $F$ and $H$ are $\kappa$\+flat.
 Then $H$ is a $\kappa$\+filtered colimit of projective left
$R$\+modules $Q_i$, hence it follows that the short exact sequence
$0\rarrow F\rarrow G\rarrow H\rarrow0$ is a $\kappa$\+filtered colimit
of split short exact sequences $0\rarrow F\rarrow F\oplus Q_i
\rarrow Q_i\rarrow0$ (proving the assertion~(b)).
 In particular, the left $R$\+module $G$ is a $\kappa$\+filtered
colimit of left $R$\+modules isomorphic to $F\oplus Q_i$.
 The latter are obviously $\kappa$\+flat; and a $\kappa$\+filtered
colimit of $\kappa$\+flat modules is $\kappa$\+flat, as we have
already seen.
 Thus the left $R$\+module $G$ is $\kappa$\+flat.
\end{proof}

\begin{rem}
 We are not aware of a definition equivalent to our notion of
$\kappa$\+flatness appearing anywhere in the previously existing
literature, but certainly there are all kinds of similar or
related definitions known for many years.

 In particular, there is the classical notion of
a \emph{$\kappa$\+free} module~\cite{EM}.
 In the similar spirit (cf.~\cite{Ek}), one can define
\emph{$\kappa$\+projective} modules.
 Specifically, an $R$\+module $M$ is said to be
$\kappa$\+projective if it admits a system of submodules $M_\alpha$
(which is said to \emph{witness} the $\kappa$\+projectivity of~$M$)
such that (1)~every $R$\+module $M_\alpha$ is projective with less
than~$\kappa$ generators; (2)~every subset of $M$ of the cardinality
less than~$\alpha$ is contained in one of the modules~$M_\alpha$;
and (3)~the set of submodules $M_\alpha\subset M$ is closed under
unions of well-ordered chains of length smaller than~$\kappa$.
 An $R$\+module $M$ is said to be \emph{$\kappa$\+projective
in the weak sense} if it has a witnessing system of submodules
satisfying the conditions~(1) and~(2) \cite[Section~IV.1]{EM}.

 An $R$\+module is $\kappa$\+projective in the weak sense if and only
if it is the colimit of a $\kappa$\+filtered diagram of projective
$R$\+modules with less than~$\kappa$ generators \emph{and injective
morphisms between them}.
 Indeed, given a module that is $\kappa$\+projective in the weak
sense, all its projective submodules with less than $\kappa$ generators
form such a diagram.
 Thus $\kappa$\+projectivity is a stronger condition than
$\kappa$\+flatness.
 When $R$ is a left hereditary ring, there is no difference between
$\kappa$\+projectivity and $\kappa$\+projectivity in the weak sense
for left $R$\+modules.
 Assuming additionally that every left ideal in $R$ is generated by
less than~$\kappa$ elements (so any $\kappa$\+generated left
$R$\+module is $\kappa$\+presentable), these two conditions are also
equivalent to $\kappa$\+flatness.
 Indeed, let $F$ be a $\kappa$\+flat left $R$\+module.
 Given a subset of cardinality less than~$\kappa$ in $F$,
denote by $E$ the $R$\+submodule in $F$ generated by this subset.
 Then $E$ is $\kappa$\+presentable, so the injective morphism
$E\rarrow F$ factorizes through a projective left $R$\+module~$P$.
 It follows that $E$ is a submodule in $P$, hence a projective
module itself.

 Another classical concept closely related to $\kappa$\+flatness is
that of \emph{$\kappa$\+purity}.
 A left $R$\+module morphism $f\:P\rarrow Q$ is said to be
a \emph{$\kappa$\+pure epimorphism} if, for any $\kappa$\+presentable
left $R$\+module $E$, any left $R$\+module morphism $E\rarrow Q$
factorizes through~$f$.
 If this is the case, $Q$ is said to be a \emph{$\kappa$\+pure
quotient} of~$P$ \cite[Chapter~7]{JL}, \cite{AR2}.
 One can easily see that an $R$\+module is $\kappa$\+flat if and only
if it is a $\kappa$\+pure quotient of a projective $R$\+module.

 Finally, some of the assertions and proofs of our 
Theorem~\ref{kappa-flat-modules} and
Lemma~\ref{kappa-flat-closure-properties} resemble those of
the theory of \emph{modules with support in a subcategory} (of finitely
presented modules), as developed in~\cite{Len}
(see~\cite[Propositions~2.1\+-2]{Len}).
\end{rem}

 The following result, whose nonadditive version goes back
to~\cite[Section~2.2]{Isb}, was rediscovered and discussed
in the additive/abelian context in~\cite[Theorem~6.10]{PS}.

\begin{thm} \label{copower-0-good-PS}
 Let\/ $\sB$ be a locally $\kappa$\+presentable abelian category
and $P\in\sB$ be a $\kappa$\+presentable projective generator of $P$
(where $\kappa$~is a regular cardinal).
 Let $\lambda$ be a cardinal such that $\lambda^+\ge\kappa$ and
let $Q=P^{(\lambda)}\in\sB$ be the coproduct of $\lambda$~copies of
$P$ in\/~$\sB$.
 Denote by $S$ the\/ ring $\Hom_\sB(Q,Q)^\rop$.
 Then the exact functor\/ $\Theta=\Hom_\sB(Q,{-})\:\sB\rarrow S\modl$
corepresented by $Q$ is fully faithful.
\end{thm}

\begin{proof}[Sketch of proof]
 Let us identify $\sB$ with the category of modules over
the $\kappa$\+accessible monad $\boT_P\:X\longmapsto
\Hom_\sB(P,P^{(X)})$ on the category of sets.
 Let $Y$ be a set of the cardinality~$\lambda$; then the projective
object $P\in\sB$ corresponds to the free $\boT_P$\+module with one
generator $\boT_P(*)$ and the projective object $Q\in\sB$ corresponds
to the free $\boT_P$\+module with $\lambda$~generators $\boT_P(Y)$.
 The functor $\Theta\:\boT_P\modl\rarrow S\modl$ assigns to
a $\boT_P$\+module $B$ the $S$\+module $\Theta(B)=
\Hom_{\boT_P}(\boT_P(Y),B)=B^Y$.

 Let $B$ and $C$ be two $\boT_P$\+modules and $g\:B^Y\rarrow C^Y$ be
an $S$\+module morphism.
 Given an element $y\in Y$, denote by $\delta_B\:B\rarrow B^Y$
the diagonal embedding and by $\pr_{y,B}\:B^Y\rarrow B$ the projection
onto the component indexed by~$y$.
 Then the composition $\delta_B\pr_{y,B}\:B^Y\rarrow B^Y$ is equal to
the action of the element $s_y\in S$ corresponding to the composition
of the natural morphism $P^{(Y)}\rarrow P$ with the coproduct injection
of the $y$\+indexed component $P\rarrow P^{(Y)}$.
 Thus the morphism~$g$ forms a commutative diagram with the maps
$\delta_B\pr_{y,B}$ and~$\delta_C\pr_{y,C}$,
$$
\begin{diagram}
\node{B^Y}\arrow{s,l}{g}\arrow{e,t}{\pr_{y,B}}
\node{B}\arrow{e,t}{\delta_B}\node{B^Y}\arrow{s,l}{g} \\
\node{C^Y}\arrow{e,t}{\pr_{y,C}}
\node{C}\arrow{e,t}{\delta_C}\node{C^Y}
\end{diagram}
$$
 As this holds for all elements $y\in Y$, it follows easily that
there is a map $f\:B\rarrow C$ such that $g=f^Y$.

 Now, for any $Y$\+ary operation $t\in\boT_P(Y)$ in the monad
$\boT_P$, the composition of the natural morphism $P^{(Y)}\rarrow P$
with the morphism $t\:P\rarrow P^{(Y)}$ defines an element
$s_t\in S$.
 The action of~$s_t$ in $B^Y$ is equal to the composition of
the $Y$\+ary operation $t_{\boT_P}(B)\:B^Y\rarrow B$ in $B$ (see
Section~\ref{additive-monads-on-sets-secn}) with
the diagonal embedding $\delta_B\:B\rarrow B^Y$.
 Commutativity of the diagram
$$
\begin{diagram}
\node{B^Y}\arrow{s,l}{f^Y}\arrow{e,t}{t_{\boT_P}(B)}
\node{B}\arrow{e,t}{\delta_B}\node{B^Y}\arrow{s,l}{f^Y} \\
\node{C^Y}\arrow{e,t}{t_{\boT_P}(C)}
\node{C}\arrow{e,t}{\delta_C}\node{C^Y}
\end{diagram}
$$
means that the map $f\:B\rarrow C$ preserves the operation
$t_{\boT_P}$ in $B$ and~$C$.
 As this holds for all $t\in\boT_P(Y)$ and all operations in
the monad $\boT_P$ depend essentially on at most~$\lambda$
arguments, it follows that $f$~is a morphism of $\boT_P$\+modules. 
\end{proof}

 The argument deducing Theorem~\ref{copower-good-projective-generator}
from Theorem~\ref{copower-0-good-PS} is based on a technique summarized
in the following proposition.

 Let $\sA$ be an abelian category with enough projective objects and
$\sB\subset\sA$ be a full subcategory closed under the kernels and
cokernels in~$\sA$; so $\sB$ is also an abelian category and
the embedding functor $\sB\rarrow\sA$ is exact.
 Assume that there exists a functor $\Delta\:\sA\rarrow\sB$ left adjoint
to the fully faithful embedding functor $\sB\rarrow\sA$.
 Then the functor $\Delta$ takes projective objects in $\sA$ to
projective objects in $\sB$, and it follows easily that there are
enough projectives in~$\sB$ (cf.
Lemma~\ref{reflective-projective-generator}).
 Let $\boL_n\Delta\:\sA\rarrow\sB$, \ $n\ge0$, denote the left derived
functor of the right exact functor~$\Delta$.

\begin{prop} \label{fully-faithful-on-bounded-above-derived}
 The following four conditions are equivalent: \par
\textup{(a)} $\boL_n\Delta(P)=0$ for every object $P\in\sB_\proj
\subset\sA$ and all $n\ge1$; \par
\textup{(b)} $\boL_n\Delta(B)=0$ for every object $B\in\sB
\subset\sA$ and all $n\ge1$; \par
\textup{(c)} the triangulated functor\/ $\sD^-(\sB)\rarrow\sD^-(\sA)$
induced by the exact embedding functor\/ $\sB\rarrow\sA$ is fully
faithful; \par
\textup{(d)} the triangulated functor\/ $\sD^\b(\sB)\rarrow\sD^\b(\sA)$
induced by the exact embedding functor\/ $\sB\rarrow\sA$ is fully
faithful.
\end{prop}

\begin{proof}
 This is an infinite homological dimension version
of~\cite[Theorem~6.4]{PMat} (see also~\cite[proofs of Theorems~1.3
and~2.9]{Pmgm}).

 (a)~$\Longleftrightarrow$~(b):
 The composition $\sB\rarrow\sA\rarrow\sB$ of the embedding
$\sB\rarrow\sA$ with its left adjoint functor $\Delta\:\sA\rarrow
\sB$ is the identity functor $\Id_{\sB}$, since the embedding functor
$\sB\rarrow\sA$ is fully faithful.
 Let $P_\bu\rarrow B$ be a left projective resulution of an object
$B\in\sB$.
 In the assumption of~(a), since $\boL_n\Delta(P_i)=0$ for all $i\ge0$
and $n\ge1$, the complex $\Delta(P_\bu)$ computes the derived functor
$\boL_*\Delta(B)$.
 However, we have $\Delta(P_\bu)=P_\bu$, hence $\boL_n\Delta(B)=0$
for $n\ge1$.

 (b)~$\Longleftrightarrow$~(c):
 Denote by $\sA_{\Delta\adj}\subset\sA$ the full subcategory of all
objects $A\in\sA$ such that $\boL_n\Delta(A)=\nobreak0$ for all $n\ge1$.
 Then the full subcategory $\sA_{\Delta\adj}\subset\sA$ is closed under
extensions and the kernels of epimorphisms in~$\sA$; so
$\sA_{\Delta\adj}$ inherits an exact category structure from the abelian
category~$\sA$.
 Furthermore, every object of $\sA$ is the image of an epimorphism
from an object of~$\sA_{\Delta\adj}$.
 It follows that the triangulated functor $\sD^-(\sA_{\Delta\adj})
\rarrow\sD^-(\sA)$ induced by the exact embedding $\sA_{\Delta\adj}
\rarrow\sA$ is a triangulated equivalence (e.~g.,
by~\cite[Proposition~A.3.1(a)]{Pcosh}; or simply because both
the derived categories in question are equivalent to the homotopy
category of bounded above complexes of projective objects in~$\sA$).

 The restriction of the functor $\Delta$ to the exact subcategory
$\sA_{\Delta\adj}\subset\sA$ is an exact functor $\Delta\:
\sA_{\Delta\adj}\rarrow\sB$.
 Applying the functor $\Delta$ to bounded above complexes of objects
from $\sA_{\Delta\adj}$, one constructs a triangulated functor
$$
 \boL\Delta\:\sD^-(\sA)\lrarrow\sD^-(\sB),
$$
which is left adjoint to the triangulated functor $\sD^-(\sB)\rarrow
\sD^-(\sA)$ induced by the embedding $\sB\rarrow\sA$
\cite[Lemma~8.3]{Psemi}.

 Now the functor $\sD^-(\sB)\rarrow\sD^-(\sA)$ is fully faithful if
and only if its composition $\sD^-(\sB)\rarrow\sD^-(\sA)\rarrow
\sD^-(\sB)$ with its left adjoint functor $\boL\Delta\:\sD^-(\sA)\rarrow
\sD^-(\sB)$ is the identity functor $\sD^-(\sB)\rarrow\sD^-(\sB)$.
 In the situation of~(b), we have $\sB\subset\sA_{\Delta\adj}$.
 Since the composition $\sB\rarrow\sA\rarrow\sB$ is the identity
functor, it follows that so is the composition
$\sD^-(\sB)\rarrow\sD^-(\sA)\rarrow\sD^-(\sB)$.

 Conversely, if $\sD^-(\sB)\rarrow\sD^-(\sA)\rarrow\sD^-(\sB)$ is
the identity functor, we obviously have $\boL_n\Delta(B)=0$
for all $B\in\sB$ and $n\ge1$.

 The implication (c)~$\Longrightarrow$~(d) is obvious.

 (d)~$\Longrightarrow$~(b):
 Fix an integer $m\ge0$, and denote by $\sD^{\b,\ge-m}(\sB)\subset
\sD^\b(\sB)$ the full subcategory of bounded complexes with
the cohomology objects concentrated in the cohomological
degrees~$\ge-m$.
 The embedding functor $\sD^{\b,\ge-m}(\sB)\rarrow\sD^-(\sB)$ has
a left adjoint, which is the canonical truncation functor
$\tau^{\ge-m}\:\sD^-(\sB)\rarrow\sD^{\b,\ge-m}(\sB)$ (taking a complex
$\dotsb\rarrow B^{-m-2}\rarrow B^{-m-1}\rarrow B^{-m}\rarrow B^{-m+1}
\rarrow\dotsb$ to the complex $\dotsb\rarrow0\rarrow0\rarrow
B^{-m}/\im B^{-m-1}\rarrow B^{-m+1}\rarrow\dotsb$).
 Hence the composition of functors $\sD^{\b,\ge-m}(\sB)
\rarrow\sD^-(\sB)\rarrow\sD^-(\sA)$ also has a left adjoint, which
can be computed as the composition of the two left adjoints
$\tau^{\ge-m}\boL\Delta$.

 On the other hand, the same functor $\sD^{\b,\ge-m}(\sB)\rarrow
\sD^-(\sA)$ can be obtained as the composition $\sD^{\b,\ge-m}(\sB)
\rarrow\sD^\b(\sB)\rarrow\sD^\b(\sA)\rarrow\sD^-(\sA)$.
 In the assumption of~(d), the functor $\sD^\b(\sB)\rarrow\sD^\b(\sA)$
induced by $\Theta$ is fully faithful.
 Since the embedding functors $\sD^{\b,\ge-m}(\sB)\rarrow\sD^\b(\sB)$
and $\sD^\b(\sA)\rarrow\sD^-(\sA)$ are fully faithful, too, it follows
that so is the composition $\sD^{\b,\ge-m}(\sB)\rarrow\sD^-(\sA)$.

 Therefore, the composition of the two adjoint functors
$\sD^{\b,\ge-m}(\sB)\rarrow\sD^-(\sA)\rarrow\sD^{\b,\ge-m}(\sB)$ is
isomorphic to the identity functor.
 In other words, it means that the adjunction morphism
$B\rarrow\tau^{\ge-m}\boL\Delta(B)$ is an isomorphism for every
$B\in\sB$.
 Since $m\ge0$ is an arbitrary integer, it follows that
$\boL_n\Delta(B)=0$ for all $n\ge0$.
\end{proof}

\begin{proof}[Proof of Theorem~\ref{copower-good-projective-generator}]
 Denote by $S$ the ring $\Hom_\sB(Q,Q)^\rop$.
 First of all, we already know from
Theorem~\ref{copower-0-good-PS} that the functor $\Theta=\Hom_\sB(Q,{-})
\:\sB\rarrow S\modl$ is fully faithful.
 So $Q$ is a $0$\+good projective generator of~$\sB$.

 To prove that $Q$ is a good projective generator, we have to check
that the triangulated functor $\sD^\b(\sB)\rarrow\sD^\b(S\modl)$
induced by $\Theta$ is fully faithful.
 It was explained in the proof of
Proposition~\ref{fully-faithful-b-implies-minus} that
the exact embedding functor $\Theta\:\sB\rarrow S\modl$ has a left
adjoint functor $\Delta\:S\modl\rarrow\sB$.
 According to Proposition~\ref{fully-faithful-on-bounded-above-derived},
it remains to show that the projective objects of the abelian
full subcategory $\sB\subset S\modl$ are adjusted to~$\Delta$, that is
$\boL_n\Delta(\Theta(P))=0$ for all $P\in\sB_\proj$ and all $n\ge1$.

 As any projective object in $\sB$ is a direct summand of an object of
the form $Q^{(X)}$, where $X$ is some set, it suffices to check that
the left $S$\+module $\Theta(Q^{(X)})$ is adjusted to $\Delta$ for
all sets~$X$.
 Now, the functor $\Theta$ preserves $\lambda^+$\+filtered colimits
(since the object $Q\in\sB$ is $\lambda^+$\+presentable).
 Hence we have $\Theta(Q^{(X)})=\varinjlim_Z\Theta(Q^{(Z)})$,
where the colimit is taken over all the subsets $Z\subset X$ of
cardinality not exceeding~$\lambda$.

 Furthermore, $Q^{(Z)}\simeq Q$ in $\sB$ for all nonempty sets $Z$ of
cardinality~$\le\nobreak\lambda$, so the left $S$\+module
$\Theta(Q^{(Z)})$ is isomorphic to~$S$.
 By Lemma~\ref{kappa-flat-closure-properties}(a), it follows that
the left $S$\+module $\Theta(Q^{(X)})$ is $\lambda^+$\+flat for
all sets~$X$, as a $\lambda^+$\+filtered colimit of free left
$S$\+modules.
 So the functor $\Theta$ takes the projective objects of $\sB$ to
$\lambda^+$\+flat left $S$\+modules (this observation improves upon
the result of~\cite[Lemma~6.13]{PS}, where it was noticed that $\Theta$
takes the projective objects of $\sB$ to flat left $S$\+modules).

 Let us show that all the $\lambda^+$\+flat left $S$\+modules are
adjusted to~$\Delta$, that is $\boL_n\Delta(F)=\nobreak0$ for all
$\lambda^+$\+flat left $S$\+modules $F$ and all $n\ge1$.
 Since the projective left $S$\+modules are $\lambda^+$\+flat and
the class of all $\lambda^+$\+flat left $S$\+modules is closed under
the kernels of surjective morphisms, we only need to check
that the functor $\Delta$ preserves exactness of short exact
sequences of $\lambda^+$\+flat left $S$\+modules.

 By Lemma~\ref{kappa-flat-closure-properties}(b),
every short exact sequence of $\lambda^+$\+flat $S$\+modules
is a $\lambda^+$\+filtered colimit of split short exact sequences.
 It remains to point out that the functor $\Delta$ preserves all
colimits, and that $\lambda^+$\+filtered colimits are exact in $\sB$
(e.~g., because they are preserved by the conservative exact functor
$\Theta$ and exact in $R\modl$; cf.~\cite[Proposition~1.59]{AR}
and~\cite[Definition~2.1]{PR}).
\end{proof}

\Section{Nonabelian Subcategories Defined by Perpendicularity
Conditions} \label{nonabelian-secn}

 The aim of this section is to show what can happen if one relaxes
the conditions in the definitions of right $n$\+perpendicular
subcategories in Section~\ref{perpendicular-subcategories-secn}.
 We start with one positive assertion before proceeding to present
various counterexamples.

\begin{lem} \label{1-perpendicular-cocomplete}
 Let\/ $\sA$ be a locally presentable abelian category and\/
$\sB\subset\sA$ be a right $n$\+perpendicular subcategory
to a set of objects or morphisms in\/~$\sA$, where $n\ge0$.
 Then\/ $\sB$ is an additive category with kernels, cokernels,
infinite direct sums and products.
 The kernels and products in\/ $\sB$ coincide with those
computed in\/ $\sA$, while the cokernels and coproducts in\/ $\sB$
can be obtained by applying the reflector\/ $\Delta\:\sA\rarrow\sB$
to the cokernels and coproducts computed in\/~$\sA$.
\end{lem}

\begin{proof}
 In view of Lemma~\ref{n+1-perpendicular-implies-n}(b), it suffices
to consider the case $n=0$.
 Then, by~\cite[Theorem~1.39]{AR} (cf.\ the discussion in the proof of
Lemma~\ref{perpendicular-to-set-loc-pres-and-reflective}),
the category $\sB$ is locally presentable and reflective as a full
subcategory in~$\sA$.
 By the definition and by~\cite[Corollary~1.28]{AR}, any locally
presentable category is complete and cocomplete.
 It remains to observe that the full subcategory $\sB\subset\sA$
is closed under limits
(Lemma~\ref{right-perpendicular-closure-properties}(a)),
while the reflector $\Delta\:\sA\rarrow\sB$, being a left adjoint
functor, preserves colimits.
 Before we finish, let us recall that, when $n\ge2$, the cokernels in
$\sB$ also coincide with those computed in $\sA$ by
Lemma~\ref{right-perpendicular-closure-properties}(c).
\end{proof}

 Our counterexamples are produced by the following construction
inspired by Example~\ref{sheaves-cosheaves-ex}.
 Let $R\rarrow S$ be a homomorphism of associative rings.
 Denote by $\sA$ the abelian category whose objects are triples
$(N,M,f)$, where $N$ is a left $S$\+module, $M$ is a left $R$\+module,
and $f\:N\rarrow M$ is an $R$\+module morphism.
 The morphisms $(N',M',f')\rarrow(N'',M'',f'')$ in $\sA$ are pairs
consisting of an $S$\+module morphism $N'\rarrow N''$ and
an $R$\+module morphism $M'\rarrow M''$ forming a commutative
square with $f'$ and~$f''$.
 One can easily interpret $\sA$ as the category of modules
$A=T\modl$ over an appropriate matrix ring~$T$.
 
 Denote by $E$ the object $(S,0,0)\in\sA$, where $S$ is viewed as
a free let $S$\+module with one generator.

\begin{lem} \label{nonabelian-counterex-ext-computed}
 Let $L=(N,M,f)$ be an object of\/~$\sA$.
 Then the following Ext computations hold: \par
\textup{(a)} $\Hom_\sA(E,L)=\ker(N\to\Hom_R(S,M))$, \par
\textup{(b)} $\Ext_\sA^1(E,L)=\coker(N\to\Hom_R(S,M))$, \par
\textup{(c)} $\Ext_\sA^i(E,L)=\Ext_R^{i-1}(S,M)$ for $i\ge2$, \par
\noindent where the $S$\+module morphism $N\rarrow\Hom_R(S,M)$ is
induced by the $R$\+module morphism $N\rarrow M$.
\end{lem}

\begin{proof}
 There is a short exact sequence
$$
 0\lrarrow (0,S,0)\lrarrow(S,S,\id_S)\lrarrow(S,0,0)\lrarrow0
$$
in the category~$\sA$, with a projective object $(S,S,\id_S)\in\sA$.
 Computing $\Hom_\sA(E,L)$ and $\Ext^1_\sA(E,L)$ in terms of this
one-step resolution of the object $E$ produces the natural
isomorphisms~(a\+b).
 Furthermore, let $\dotsb\rarrow P_1\rarrow P_0\rarrow S\rarrow 0$ be
a projective resolution of the left $R$\+module~$S$.
 Then the following is a projective resolution of the object $E\in\sA$:
$$
 \dotsb\lrarrow(0,P_1,0)\lrarrow(0,P_0,0)\lrarrow (S,S,\id_S)
 \lrarrow (S,0,0)\lrarrow0.
$$
 Computing $\Ext^i_\sA(E,L)$ in terms of this resolution produces
the natural isomorphism~(c) as well.
\end{proof}

 Given a left $R$\+module $G$ and an integer $n\ge0$, we denote by
$G^{\perp_{1..n}}\subset R\modl$ the full subcategory in $R\modl$ formed by
all the $R$\+modules $M$ such that $\Ext_R^i(G,M)=0$ for all
$1\le i\le n$.
 So, in particular, we have $G^{\perp_{1..0}}=R\modl$.
 Similarly, we denote by $E^{\perp_{0..n}}\subset\sA$ the full subcategory
formed by all the objects $L\in\sA$ such that $\Ext_\sA^i(E,L)=0$
for all $0\le i\le n$.
 For $n=\infty$, the notation $G^{\perp_{1..\infty}}\subset R\modl$ and
$E^{\perp_{0..\infty}}\subset\sA$ also has the obvious meaning.

\begin{cor} \label{nonabelian-counterex-cor}
 The functor assigning to an $R$\+module $M$ the object
$(\Hom_R(S,M),\allowbreak M,f_M)\in\sA$, where $f_M\:\Hom_R(S,M)
\rarrow M$ is the $R$\+module morphism induced by the ring homomorphism
$R\rarrow S$, provides an equivalence between the full subcategory
$E^{\perp_{0..1}}\subset\sA$ and the abelian category $R\modl$.
 For every $n\ge1$, this equivalence identifies the full subcategories
$E^{\perp_{0..n+1}}\subset\sA$ and $S^{\perp_{1..n}}\subset R\modl$ (where $S$
is viewed as a left $R$\+module), providing a category equivalence
between them.
\end{cor}

\begin{proof}
 The first assertion follows from
Lemma~\ref{nonabelian-counterex-ext-computed}(a\+b),
and the second one from
Lemma~\ref{nonabelian-counterex-ext-computed}(c).
\end{proof}

\begin{ex} \label{nonabelian-counterex-general}
 Let $R$ be an associative ring and $G$ be a left $R$\+module admitting
a right $R$\+module structure which makes it an $R$\+$R$\+module.
 Set $S=R\oplus G$ to be the trivial ring extension of $R$ by $G$
(so the product of any two elements of $G$ is zero in $S$, while
the products of two elements of $R$ and of one element from $R$
and one element from $G$ in $S$ are defined using the ring structure
on $R$ and the $R$\+$R$\+bimodule structure on~$G$).
 Then the above construction provides an associative ring $T$ with
a left $T$\+module $E$ such that, by
Corollary~\ref{nonabelian-counterex-cor}, the full subcategory
$E^{\perp_{0..1}}\subset T\modl$ is equivalent to $R\modl$, and this
equivalence restricts to equivalences between the full subcategories
$E^{\perp_{0..n+1}}\subset T\modl$ and $G^{\perp_{1..n}}\subset R\modl$
for all $n\ge1$.

 Notice that, by Lemma~\ref{nonabelian-counterex-ext-computed},
the projective dimensions of the left $R$\+module $G$ and the left
$T$\+module $E$ are related by the rule $\pd_TE=\pd_RG+1$.
\end{ex}

\begin{exs}
 (1)~In the context of Example~\ref{nonabelian-counterex-general},
let $R$ be a commutative ring and set $G=\bigoplus_{s\in R}R[s^{-1}]$.
 Then $\pd_RG\le1$ and the orthogonal class $G^{\perp_{1..1}}=
G^{\perp_{1..\infty}}\subset R\modl$ is called the class of all
\emph{contraadjusted} $R$\+modules~\cite[Section~1.1]{Pcosh},
\cite[Sections~2 and~8]{Pcta}.
 The full subcategory $G^{\perp_{1..1}}\subset R\modl$ is equivalent
to the category of all contraherent cosheaves on $\Spec R$
(cf.\ Example~\ref{sheaves-cosheaves-ex}).

 Set $R=\boZ$ to be the ring on integers, and consider the abelian
group $A=\boZ_p^{(\omega)}$, that is, the direct sum of a countable
family of copies of the group of $p$\+adic integers (where $p$~is
a fixed prime number).
 Set $B=\mathbb Q\ot_\boZ A$ and $C=\mathbb Q/\boZ\ot_\boZ A=
\coker(A\to B)$.
 Being divisible (\,$=$~injective) abelian groups, $B$ and $C$ are
obviously contraadjusted.
 We claim that the surjective morphism $B\rarrow C$ has
\emph{no} kernel in the full subcategory of contraadjusted abelian
groups $G^{\perp_{1..1}}\subset\boZ\modl$.

 Indeed, suppose that $K\rarrow B$ is such a kernel.
 Since the class of contraadjusted modules is closed under
the passages to arbitrary quotient modules, the image of the morphism
$K\rarrow B$ is also contraadjusted.
 Hence the morphism $K\rarrow B$ has to be injective, and we can
consider $K$ as a subgroup in~$B$.
 Since the composition $K\rarrow B\rarrow C$ has to vanish, we have
$K\subset A$.
 Now, we have $A=\boZ_p^{(\omega)}$, and every summand $\boZ_p\subset A$
in this direct sum is a contraadjusted abelian group.
 Hence $K$ contains every one of these direct summands, so it follows
that $K=A$.
 However, $A$ is \emph{not} a contraadjusted abelian group, as
$\Ext^1_R(\boZ[p^{-1}],A)\ne0$ \cite[Section~12]{Pcta}.

 Thus we have constructed an associative ring $T$ and
a left $T$\+module $E$ of projective dimension~$2$ such that in
the full subcategory $\sB=E^{\perp_{0..2}}=E^{\perp_{0..\infty}}\subset T\modl$
there are morphisms which have \emph{no kernel}.
 Let us emphasize what this means: not only the full subcategory $\sB$
is not closed under kernels in $T\modl$, but some morphisms in $\sB$
do not even have any kernel internal to~$\sB$.

\smallskip

 (2)~Again in the context of Example~\ref{nonabelian-counterex-general},
let $R$ be a commutative ring and set $G=\bigoplus_I R/I$, where
the direct sum is taken over all the ideals $I\subset R$.
 Then, by Baer's criterion, the orthogonal class $G^{\perp_{1..1}}=
G^{\perp_{1..\infty}}\subset R\modl$ is the class of all injective
$R$\+modules.
 Choose $R$ to be a commutative ring of global dimension~$3$, and let
$A$ be a $R$\+module of projective dimension~$3$.
 Choose exact sequences of $R$\+modules $0\rarrow A\rarrow J\rarrow B
\rarrow 0$ and $0\rarrow B\rarrow H\rarrow C\rarrow0$, where $J$ and
$H$ are injective $R$\+modules.
 We claim that the morphism $J\rarrow H$ has \emph{no} cokernel in
the full subcategory of injective modules $G^{\perp_{1..1}}\subset R\modl$.

 Indeed, suppose that $H\rarrow K$ is such a cokernel.
 Then the composition $J\rarrow H\rarrow K$ vanishes, so we have
an $R$\+module morphism $C\rarrow K$.
 Furthermore, for every injective $R$\+module $K'$ any any $R$\+module
morphism $C\rarrow K'$ there exists a unique $R$\+module morphism
$K\rarrow K'$ making the triangle diagram $C\rarrow K\rarrow K'$
commutative.
 Choosing $K'$ to be an injective $R$\+module such that there is
an injective $R$\+module morphism $C\rarrow K'$, we see that
the morphism $C\rarrow K$ is injective.
 Choosing $K'$ to be an injective $R$\+module such that there is
an injective $R$\+module morphism $K/C\rarrow K'$, we conclude
that the morphism $C\rarrow K$ is an isomorphism.
 Hence $C$ is an injective $R$\+module, which contradicts the assumption
that $\pd_RA=3$.

 Thus we have constructed an associative ring $T$ and
a left $T$\+module $E$ of projective dimension~$4$ such that
$E^{\perp_{0..2}}=E^{\perp_{0..4}}=E^{\perp_{0..\infty}}\subset T\modl$ and
in the full subcategory $\sB=E^{\perp_{0..2}}\subset T\modl$ there are
morphisms which have \emph{no cokernel}.
 As in~(1), this means that not only the full subcategory $\sB$ is
not closed under cokernels in $T\modl$, but some morphisms in $\sB$
do not even have any cokernel internal to~$\sB$.
\end{exs}

\begin{ex}
 This example is a na\"\i ve version of the category of contraherent
cosheaves on the projective line $\mathbb P^1_k$ (where $k$~is a field).
 Consider two polynomial rings $R'=k[x]$ and $R''=k[x^{-1}]$ in
the variables $x$ and~$x^{-1}$, and the ring of Laurent
polynomials $S=k[x,x^{-1}]$ containing $R'$ and $R''$ as subrings.
 Denote by $\sC$ the abelian category whose objects are quintuples
$(N,M',M'',f',f'')$, where $N$ is an $S$\+module, $M'$ is
an $R'$\+module, $M''$ is an $R''$\+module, $f'\:N\rarrow M'$ is
an $R'$\+module morphism, and $f''\:N\rarrow M''$ is
an $R''$\+module morphism.
 Morphisms in the category $\sC$ are defined in the obvious way
(similar to the construction of the category $\sA$ above).
 It is not difficult to construct a matrix ring $U$ such that
$\sC=U\modl$.

 Consider the two objects $E'=(S,0,S,0,\id_S)$ and
$E''=(S,S,0,\id_S,0)\in\sC$, and set $E=E'\oplus E''$.
 Then $E$ is an object of projective dimension~$2$ in $\sC$,
and the full subcategory $E^{\perp_{0..1}}\subset\sC$ is equivalent to
the category of all pairs of modules $M'$ over $R'$ and $M''$
over $R''$ endowed with an isomorphism of $S$\+modules
$\Hom_{R'}(S,M')\simeq\Hom_{R''}(S,M'')$.
 The full subcategory $E^{\perp_{0..2}}\subset E^{\perp_{0..1}}$ is defined
by the additional conditions $\Ext^1_{R'}(S,M')=0$ and
$\Ext^1_{R''}(S,M'')=0$ (cf.\ Example~\ref{sheaves-cosheaves-ex}).

 By Lemma~\ref{1-perpendicular-cocomplete}, the full subcategory
$\sB=E^{\perp_{0..1}}\subset U\modl$ is an additive category with kernels
and cokernels (and, also, infinite direct sums and products).
 Let us show that the category $\sB$ is \emph{not} abelian.
 Indeed, consider the two objects $A=(0,R',R'',0,0)$ and
$C=(k,k,k,\id_k,\id_k)$, where $x$ acts in~$k$ by the identity
operator.
 Then there is a morphism $g\:A\rarrow C$ in $\sB$ whose components
$R'\rarrow k$ and $R''\rarrow k$ take $x^n$ to~$1$ for all $n\in\boZ$.
 Now, the cokernel of~$g$ vanishes, so $C$ is the kernel of
the cokernel of~$g$.
 On the other hand, the kernel of~$g$ is the morphism $h\:A\rarrow A$
with the components $1-x\:R'\rarrow R'$ and $1-x^{-1}\:R''\rarrow R''$.
 The cokernel of the kernel of~$g$ is isomorphic to the direct sum
of two copies of~$C$.
\end{ex}


\bigskip
\begin{thebibliography}{99}
\smallskip

\bibitem{AR}
 J.~Ad\'amek, J.~Rosick\'y.
   Locally presentable  and accessible categories.
London Math.\ Society Lecture Note Series~189,
Cambridge University Press, 1994.

\bibitem{AR2}
 J.~Ad\'amek, J.~Rosick\'y.
   On pure quotients and pure subobjects.
\textit{Czechoslovak Math.\ Journ.}\ \textbf{54 (129)}, \#3,
p.~623--636, 2004.

\bibitem{Ba}
 S.~Bazzoni.
   The $t$-structure induced by an $n$-tilting module.
\textit{Transactions of the Amer.\ Math.\ Soc.}\
\textbf{371}, \#9, p.~6309--6340, 2019.
\texttt{arXiv:1604.00797 [math.RT]}

\bibitem{BMT}
 S.~Bazzoni, F.~Mantese, A.~Tonolo.
   Derived equivalence induced by infinitely generated $n$\+tilting
modules.
\textit{Proceedings of the Amer.\ Math.\ Soc.}\ \textbf{139}, \#12,
p.~4225--4234, 2011. \texttt{arXiv:0905.3696 [math.RA]}

\bibitem{BP2}
 S.~Bazzoni, L.~Positselski.
   Matlis category equivalences for a ring epimorphism.
\textit{Journ.\ of Pure and Appl.\ Algebra} \textbf{224}, \#10,
article ID~106398, 25~pp., 2020.  \texttt{arXiv:1907.04973 [math.RA]}

\bibitem{EP}
 A.~I.~Efimov, L.~Positselski.
   Coherent analogues of matrix factorizations and relative
singularity categories.
\textit{Algebra and Number Theory} \textbf{9}, \#5,
p.~1159--1292, 2015.  \texttt{arXiv:1102.0261 [math.CT]}

\bibitem{Ek}
 P.~C.~Eklof.
   Set-theoretic methods: the uses of Gamma invariants.
In: Abelian groups (Cura\c cao, 1991), \textit{Lecture Notes in
Pure and Appl.\ Math.} \textbf{146}, Dekker, New York, 1993,
p.~43--53.

\bibitem{EM}
 P.~C.~Eklof, A.~H.~Mekler.
   Almost free modules: Set-theoretic methods.
Revised edition.  North-Holland Mathematical Library, 65.
Amsterdam, 2002.

\bibitem{Emm}
 I.~Emmanouil.
   Mittag-Leffler condition and the vanishing of $\varprojlim^1$.
\textit{Topology} \textbf{35}, \#1, p.~267--271, 1996.

\bibitem{EE}
 E.~Enochs, S.~Estrada.
   Relative homological algebra in the category of quasi-coherent
sheaves.
\textit{Advances in Math.} \textbf{194}, \#2, p.~284--295, 2005.

\bibitem{FN}
 A.~Facchini, Z.~Nazemian.
   Equivalence of some homological conditions for ring epimorphisms.
\textit{Journ.\ of Pure and Appl.\ Algebra} \textbf{223}, \#4,
p.~1440--1455, 2019.  \texttt{arXiv:1710.00097 [math.RA]}

\bibitem{FMS}
 L.~Fiorot, F.~Mattiello, M.~Saor\'\i n.
   Derived equivalences induced by nonclassical tilting objects.
\textit{Proceedings of the Amer.\ Math.\ Soc.} \textbf{145}, \#4,
p.~1505--1514, 2017.  \texttt{arXiv:1511.06148 [math.RT]}

\bibitem{GL}
 W.~Geigle, H.~Lenzing.
   Perpendicular categories with applications to representations
and sheaves.
\textit{Journ.\ of Algebra} \textbf{144}, \#2, p.~273--343, 1991.

\bibitem{HAR}
 M.~H\'ebert, J.~Ad\'amek, J.~Rosick\'y.
    More on orthogonality in locally presentable categories.
\textit{Cahiers de topologie et g\'eom\'etrie diff\'erentielle
cat\'egoriques} \textbf{42}, \#1, p.~51--80, 2001.

\bibitem{HR}
 M.~H\'ebert, J.~Rosick\'y.
    Uncountable orthogonality is a closure property.
\textit{Bulletin of the London Math.\ Soc.} \textbf{33}, \#6,
p.~685--688, 2001.

\bibitem{Isb}
 J.~R.~Isbell.
   Adequate subcategories.
\textit{Illinois Journ.\ of Math.}\ \textbf{4}, \#4, p.~541--552, 1960.

\bibitem{Jan}
 U.~Jannsen.
   Continuous \'etale cohomology.
\textit{Mathematische Annalen} \textbf{280}, \#2, p.~207--245, 1988.

\bibitem{JL}
 C.~U.~Jensen, H.~Lenzing.
   Model-theoretic algebra (with particular emphasis on fields,
rings, modules).
Algebra, Logic, and Applications, 2.
 Gordon and Breach Science Publishers, New York, 1989.

\bibitem{Kra}
 H.~Krause.
   A Brown representability theorem via coherent functors.
\textit{Topology} \textbf{41}, \#4, p.~853--861, 2002.

\bibitem{Len}
 H.~Lenzing.
   Homological transfer form finitely presented to infinite modules.
In: Abelian group theory (Honolulu, 1983), \textit{Lecture Notes
in Math.}\ \textbf{1006}, Springer, 1983, p.~734--761.

\bibitem{Mit}
 B.~Mitchell.
   Rings with several objects.
\textit{Advances in Math.}\ \textbf{8}, \#1, p.~1--161, 1972.

\bibitem{PSY}
 M.~Porta, L.~Shaul, A.~Yekutieli.
   On the homology of completion and torsion.
\textit{Algebras and Representation Theory} \textbf{17}, \#1,
p.~31--67, 2014.  \texttt{arXiv:1010.4386 [math.AC]}\,.
Erratum in \textit{Algebras and Representation Theory} \textbf{18},
\#5, p.~1401--1405, 2015.  \texttt{arXiv:1506.07765 [math.AC]}

\bibitem{Psemi}
 L.~Positselski.
   Homological algebra of semimodules and semicontramodules:
Semi-infinite homological algebra of associative algebraic structures.
 Appendix~C in collaboration with D.~Rumynin; Appendix~D in
collaboration with S.~Arkhipov.
 Monografie Matematyczne vol.~70, Birkh\"auser/Springer Basel, 2010. 
xxiv+349~pp. \texttt{arXiv:0708.3398 [math.CT]}

\bibitem{Pweak}
 L.~Positselski.
   Weakly curved A${}_\infty$-algebras over a topological local ring.
\textit{M\'emoires de la Soci\'et\'e Math\'ematique de France}
\textbf{159}, 2018.  vi+206~pp.  \texttt{arXiv:1202.2697 [math.CT]}

\bibitem{Pcosh}
 L.~Positselski.
   Contraherent cosheaves.
Electronic preprint \texttt{arXiv:1209.2995 [math.CT]}.

\bibitem{Prev}
 L.~Positselski.
   Contramodules.
\textit{Confluentes Math.}\ \textbf{13}, \#2, p.~93--182, 2021.
\texttt{arXiv:1503.00991 [math.CT]}

\bibitem{Pmgm}
 L.~Positselski.
   Dedualizing complexes and MGM duality.
\textit{Journ.\ of Pure and Appl.\ Algebra} \textbf{220}, \#12,
p.~3866--3909, 2016.  \texttt{arXiv:1503.05523 [math.CT]}

\bibitem{Pcta}
 L.~Positselski.
   Contraadjusted modules, contramodules, and reduced cotorsion modules.
\textit{Moscow Math.\ Journ.}\ \textbf{17}, \#3, p.~385--455, 2017.
\texttt{arXiv:1605.03934 [math.CT]}

\bibitem{PMat}
 L.~Positselski.
   Triangulated Matlis equivalence.
\textit{Journ.\ of Algebra and its Appl.}\ \textbf{17}, \#4,
article ID~1850067, 2018.  \texttt{arXiv:1605.08018 [math.CT]}

\bibitem{Psm}
 L.~Positselski.
   Smooth duality and co-contra correspondence.
\textit{Journ.\ of Lie Theory} \textbf{30}, \#1, p.~85--144,
2020.  \texttt{arXiv:1609.04597 [math.CT]}

\bibitem{Pcoun}
 L.~Positselski.
   Flat ring epimorphisms of countable type.
\textit{Glasgow Math.\ Journ.} \textbf{62}, \#2, p.~383--439, 2020.
\texttt{arXiv:1808.00937 [math.RA]}

\bibitem{PR}
 L.~Positselski, J.~Rosick\'y.
   Covers, envelopes, and cotorsion theories in locally presentable
abelian categories and contramodule categories.
\textit{Journ.\ of Algebra} \textbf{483}, p.~83--128, 2017.
\texttt{arXiv:1512.08119 [math.CT]}

\bibitem{PSl0}
 L.~Positselski, A.~Sl\'avik.
   Flat morphisms of finite presentation are very flat.
\textit{Annali di Matem.\ Pura ed Appl.}\ \textbf{199}, \#3,
p.~875--924, 2020.  \texttt{arXiv:1708.00846 [math.AC]}
 
\bibitem{PSl}
 L.~Positselski, A.~Sl\'avik.
   On strongly flat and weakly cotorsion modules.
\textit{Math.\ Zeitschrift} \textbf{291}, \#3--4, p.~831--875, 2019.
\texttt{arXiv:1708.06833 [math.AC]}

\bibitem{PS}
 L.~Positselski, J.~\v St\!'ov\'\i\v cek.
   The tilting-cotilting correspondence.
\textit{Internat.\ Math.\ Research Notices} \textbf{2021}, \#1,
p.~189--274, 2021.  \texttt{arXiv:1710.02230 [math.CT]} 

\bibitem{PS2}
 L.~Positselski, J.~\v St\!'ov\'\i\v cek.
   $\infty$-tilting theory.
\textit{Pacific Journ.\ of Math.} \textbf{301}, \#1,
p.~297--334, 2019.  \texttt{arXiv:1711.06169 [math.CT]}

\bibitem{Sch}
 P.~Schenzel.
   Proregular sequences, local cohomology, and completion.
\textit{Math.\ Scand.}\ \textbf{92}, \#2, p.~161--180, 2003.

\bibitem{V1}
 E.~M.~Vitale.
   Localizations of algebraic categories.
\textit{Journ.\ of Pure and Appl.\ Algebra} \textbf{108}, \#3,
p.~315\+-320, 1996.

\bibitem{V2}
 E.~M.~Vitale.
   Localizations of algebraic categories~II.
\textit{Journ.\ of Pure and Appl.\ Algebra} \textbf{133}, \#3,
p.~317\+-326, 1998.

\end{thebibliography}
\end{document}